\documentclass[a4paper,reqno]{article}
\usepackage{amsmath,amsfonts,amssymb,amsthm,amsrefs,bbm,graphicx,enumerate,geometry}
\usepackage{mathrsfs}
\usepackage{thmtools}
\usepackage{thm-restate}

\usepackage{subfigure} %插入多图时用子图显示的宏包

\usepackage{color}

\usepackage{framed} 
\usepackage{hyperref}  
\hypersetup{
    linktoc=page,
    linkcolor=red,          % color of internal links
    citecolor=blue,        % color of links to bibliography
    filecolor=blue,      % color of file links
    urlcolor=cyan,
   colorlinks= true          % color of external links
}

\numberwithin{figure}{section}
\numberwithin{equation}{section}

\newtheorem{thm}{Theorem}[section]

\newtheorem{lemma}[thm]{Lemma}

\newtheorem{cor}[thm]{Corollary}
\newtheorem{defn}[thm]{Definition}
\newtheorem{example}[thm]{Example}
\newtheorem{remark}[thm]{Remark}

\newtheorem{conjecture}[thm]{Conjecture}

\renewcommand{\epsilon}{\varepsilon}

\newcommand{\Q}{\mathbb{Q}}

\def\Re{{\rm Re}\,}

\def\<#1{\langle #1\rangle}
\begin{document}{\allowdisplaybreaks[4]}

%%%%%%%%%%%%%%%%%%%%%%%%%%%%%
%%%%%%%%%%%%%%%%%%%%%%%%%%%%%

\title{Multiple radial SLE(0) and classical Calogero-Sutherland system}
\author{
    Jiaxin Zhang\footnotemark[1] 
   }
\renewcommand{\thefootnote}{\fnsymbol{footnote}}

\footnotetext[1]{{\bf zhangjx.prob@gmail.com} Department of Mathematics, California Institute of Technology}

\maketitle

\begin{abstract}
 
We develop a theory of multiple radial $\mathrm{SLE}(0)$—a smooth system of curves in a simply connected domain $\Omega$ with marked boundary points $z_1, \ldots, z_n \in \partial \Omega$ and a marked interior point $q$—arising as the deterministic limit of random multiple radial $\mathrm{SLE}(\kappa)$ systems.

We construct multiple radial $\mathrm{SLE}(0)$ systems by starting from the stationary relations, which arise heuristically as the $\kappa \to 0$ limit of partition functions. By constructing the field integrals of motion for the Loewner dynamics, we show that the traces of multiple radial SLE(0) systems are the horizontal trajectories of an equivalence class of quadratic differentials. These trajectories have limiting ends at the boundary points $\{z_1,z_2,\ldots,z_n\}$.

The stationary relations connect the classification of multiple radial SLE(0) systems to the enumeration of critical points of the master function of trigonometric Knizhnik-Zamolodchikov (KZ) equations. 

In the deterministic case of $\kappa=0$, we show that the Loewner dynamics with a common parametrization of capacity form a special class of classical Calogero-Sutherland systems, restricted to a submanifold of phase space defined by the Lax matrix.

\par \
\end{abstract}

\newpage
\tableofcontents

\newpage
\section{Introduction}
\subsection{Background}
\
\indent 
The Schramm-Loewner evolution SLE($\kappa$) with $\kappa>0$ is a one-parameter family of random conformally invariant curves in the plane describing interfaces within conformally invariant systems arising from statistical physics, as introduced in \cites{Sch00, LSW04, Smi06, Sch07, SS09}. 
Conformal field theory (CFT), a quantum field theory invariant under conformal transformations, is also widely used to study critical phenomena, see \cites{Car96,FK04}. SLE and the multiple SLE systems can be coupled to conformal field theories (CFT) through the SLE-CFT correspondence, which serves as a powerful tool for predicting phenomena and computing important quantities of SLE($\kappa$) and multiple SLE($\kappa$) systems from the CFT perspective, as demonstrated in references like \cites{BB03a, Car03, FW03, FK04, Dub15a, Pel19}. The parameter $\kappa$ measures the roughness of these fractal curves and determines the central charge $c(\kappa)=(3 \kappa-8)(6-\kappa) / 2 \kappa$ of the associated CFT. 

In recent years, there has been tremendous interest in multiple SLE systems. 
These systems describe families of non-intersecting SLE curves with prescribed pairwise connections among boundary and interior points. 
In particular, multiple chordal SLE---the case with $2n$ marked boundary points and no interior points---has been thoroughly studied.

\begin{itemize}
  \item \textbf{Probabilistic constructions and classification.} 
  Works such as \cites{Dub06, KL07, Law09b, PW20} established partition functions, commutation relations, and the general framework for multiple chordal SLE($\kappa$) systems, thereby providing a rigorous probabilistic basis for the theory.  

  \item \textbf{Connections to CFT.} 
  In parallel, \cites{FK15a, Pel19, Pel20} investigated the correspondence with conformal field theory, interpreting partition functions from the CFT perspective and highlighting their role as conformal blocks.  

  \item \textbf{Deterministic limit.} 
  On the one hand, \cites{PW20} derived large deviation principles for multiple chordal SLE($\kappa$) curves from a probabilistic viewpoint. 
  On the other hand, \cites{ABKM20} identified integrals of motion for multiple chordal SLE($0$) curves via the SLE--CFT correspondence. 
  Together, these complementary approaches give a complete description of the classical limit.
\end{itemize}

Multiple radial SLE is a family of random multi-curve systems in a simply connected domain $\Omega$, with marked boundary points $z_1, \ldots, z_n \in \partial \Omega$ and a marked interior point $q$. 
In contrast to the chordal case, the theory of multiple radial SLE systems has been comparatively less developed.

\begin{itemize}
  \item \textbf{Mathematical progress.} Recent contributions such as \cites{HL21, WW24} initiated the study of multiple radial partition functions and commutation relations in special cases.

  \item \textbf{Physics perspectives.} Parallel discussions in the physics literature \cites{Car04, DC07, SKFZ11, FKZ12} studied the multiple radial SLE systems from the conformal field theory perspective but without full mathematical justification.  
\end{itemize}

Building on the above literature, the present paper advances the study of multiple radial SLE($0$) systems as the deterministic limits of multiple radial SLE($\kappa$) curves. We investigate the structure of the multiple radial SLE(0) from four different perspectives:
\begin{itemize}
    \item Stationary relations and critical points of master functions
    \item Traces as horizontal trajectories of quadratic differentials forming link patterns.
    \item Enumeration and classification. 
    \item Relations to classical Calogero-Sutherland system.
\end{itemize}

The core principle throughout our study of the multiple radial SLE system is the SLE-CFT correspondence. SLE and multiple SLE systems can be coupled to a conformal field in two key aspects:
\begin{itemize} 
\item The level-two degeneracy equations for the conformal fields coincide with the null vector equations for the SLE partition functions. 
\item The correlation functions of the conformal fields serve as martingale observables for the SLE processes. \end{itemize}

\subsection{Multiple radial SLE(0) systems}
We first define multiple radial SLE($0$) curves as natural geometric objects without reference to multiple radial SLE($\kappa$) systems.

The defining properties of this ensemble of curves are geometric commutation and conformal invariance.

\begin{defn}
Let $\gamma_1,\dotsc,\gamma_{n}$ be simple disjoint smooth curves starting from $\{z_1,z_2,\ldots,z_n\}$ which are $n$ distinct points counterclockwise on the unit circle $\partial \mathbb{D}$. 
\begin{itemize}
\item[(i)] Each curve can be individually generated by a Loewner chain. In angular coordinate, let $z_j=e^{i\theta_j}$, then 
the Loewner equation for the covering map $h_t(z)$ of $g_t(z)$ (i.e. $e^{ih_t(z)}= g_t( e^{iz})$) is given by

\begin{equation}
\partial_t h_t(z)=\cot(\frac{h_t(z)-\theta_j(t)}{2}), \quad h_0(z)=z,   
\end{equation}

and the driving function $\theta_j(t)$ evolve as

$$
\left\{\begin{array}{l}
\mathrm{~d} \theta_{j}(t)=U_j\left(\theta_{1}(t),\theta_{2}(t),\ldots, \theta_{j}(t),\ldots,\theta_{n}(t)\right) \mathrm{d} t \\
\mathrm{~d} \theta_{k}(t)=\cot \left(\left(\theta_{k}(t)-\theta_{j}(t)\right) / 2\right) \mathrm{d} t,  k\neq j
\end{array}\right.
$$
where $U_j(\boldsymbol{\theta}): \mathfrak{X}^n \rightarrow \mathbb{R}$ is assumed to be smooth  
in the chamber $$\mathfrak{X}^n=\left\{(\theta_1,\theta_2,\ldots,\theta_n) \in \mathbb{R}^n \mid \theta_1<\theta_2<\ldots<\theta_n<\theta_1+2\pi\right\}$$
\item[(ii)] The curves geometrically commute, meaning that the same collection of curves can be generated by applying the individual Loewner chains in any chosen order.
For example, we can first map out $\gamma_{\left[0, t_i\right]}^{(i)}$ using $h_{t_i}^{(i)}$, then mapping out $h_{t_i}^{(i)}\left(\gamma_{\left[0, t_j\right]}^{(j)}\right)$, or vice versa. The images are the same regardless of the order in which we map out the curves. 
\item[(iii)] Each curve $\gamma_j$ is Möbius invariant in $\mathbb{D}$. This means that if $\gamma_j$ is the curve generated by a Loewner flow and initial data $\boldsymbol{\theta}$, then its image $\phi\left(\gamma_j\right)$ under a conformal automorphism $\phi$ of $\mathbb{D}$ is, up to a time change, generated by the same flow with initial data $\phi(\boldsymbol{\theta})=\left(\phi\left(\theta_1\right), \ldots, \phi\left(\theta_n\right)\right)$.
Our definition for multiple radial SLE(0) can be naturally extended to an arbitrary simply-connected domain $\Omega$ with a marked interior point $u$ via a conformal uniformizing map $\phi: \Omega \rightarrow \mathbb{D}$, sending $u$ to $0$.

\end{itemize}
\end{defn}

Under these dynamics, the driving function $\theta_j(t)$ evolves according to $U_j(\boldsymbol{\theta})$, while the points $\theta_k(t)$, for $k \neq j$, follow the Loewner chain generated by $\theta_j(t)$. We define a differential operator corresponding to the curve $\gamma_j$ by
\[
\mathcal{M}_j = U_j(\boldsymbol{\theta}) \partial_j + \sum_{k \neq j} \cot\left( \frac{\theta_k - \theta_j}{2} \right) \partial_k, \quad j = 1, \ldots, n.
\]

For $\kappa = 0$, we can also derive the commutation relations; see Section~\ref{commutation when kappa=0 section} for details. However, it is important to emphasize a key distinction between the cases $\kappa > 0$ and $\kappa = 0$: in the case $\kappa = 0$, the conditions $\partial_j U_k = \partial_k U_j$ are not consequences of the commutation relations. These conditions are equivalent to the existence of a smooth potential function $\mathcal{U}(\boldsymbol{\theta}) : \mathfrak{X}^n \rightarrow \mathbb{R}$ such that
\[
U_j = \partial_j \mathcal{U},
\]
where the chamber $\mathfrak{X}^n$ is defined by
\[
\mathfrak{X}^n = \left\{ (\theta_1, \theta_2, \ldots, \theta_n) \in \mathbb{R}^n \mid \theta_1 < \theta_2 < \ldots < \theta_n < \theta_1 + 2\pi \right\}.
\]

If we view the multiple radial $\mathrm{SLE}(0)$ system as the classical limit of a random multiple radial $\mathrm{SLE}(\kappa)$ system, then for the latter, we have shown that the drift term $b_j(\boldsymbol{\theta})$ takes the form
\[
b_j(\boldsymbol{\theta}) = \kappa \frac{\partial \log \mathcal{Z}(\boldsymbol{\theta})}{\partial \theta_j} ,
\]
where $\mathcal{Z}(\boldsymbol{\theta})$ is a positive function satisfying the null vector equations. The idea is that, as $\kappa \to 0$, the limit
\[
\lim_{\kappa \to 0} \mathcal{Z}(\boldsymbol{\theta})^\kappa = \mathcal{U}(\boldsymbol{\theta})
\]
exists (at least for suitably chosen partition functions).

Therefore, we typically assume the existence of such a potential $\mathcal{U}(\boldsymbol{\theta})$ with $U_j(\boldsymbol{\theta}) = \partial_j \mathcal{U}(\boldsymbol{\theta})$ when defining a multiple radial $\mathrm{SLE}(0)$ system.

This observation suggests that not all multiple radial $\mathrm{SLE}(0)$ systems admit a quantization or arise as classical limits of random multiple $\mathrm{SLE}(\kappa)$ systems.

\subsection{Stationary relations, integral of motions and quadratic differentials}

We construct multiple radial SLE(0) systems through stationary relations. We heuristically demonstrate how stationary relations naturally emerge when normalizing the partition function for the multiple radial SLE($\kappa$) system as $\kappa \rightarrow 0$, as discussed in Section \ref{classical limit of multiple radial SLE system}.

\begin{defn}[Stationary relations] \label{Stationary relations spin}
Let $\boldsymbol{z} = \{z_1, z_2, \ldots, z_n\}$ be distinct points on the unit circle, and let $\boldsymbol{\xi} = \{\xi_1, \xi_2, \ldots, \xi_m\}$ be a set of involution-symmetric marked points.  
In the unit disk $\mathbb{D}$, the stationary relations are given by  
\begin{equation}
 -\sum_{j=1}^{n} \frac{2}{\xi_k - z_j} + \sum_{l \neq k} \frac{4}{\xi_k - \xi_l} + \frac{n - 2m + 2}{\xi_k} = 0, \quad k = 1, 2, \ldots, m.
\end{equation}
In angular coordinates, setting $z_i = e^{i\theta_i}$ for $i = 1, 2, \ldots, n$ and $\xi_k = e^{i\zeta_k}$ for $k = 1, 2, \ldots, m$, the stationary relations take the form  
\begin{equation}
\sum_{j=1}^{n} \cot\left(\frac{\zeta_k - \theta_j}{2} \right) = \sum_{l \neq k} 2 \cot\left(\frac{\zeta_k - \zeta_l}{2} \right), \quad k = 1, 2, \ldots, m.
\end{equation}  
\end{defn}

Based on the stationary relations, we now define the multiple radial SLE(0) systems.

\begin{defn}[Multiple radial SLE(0) Loewner chain] \label{multiple radial SLE(0) Loewner chain via stationary relations}
Given growth points $\boldsymbol{z} = \{z_1, z_2, \dots, z_n\}$ on the unit circle, a marked interior point $u = 0$, and involution-symmetric screening charges $\{\xi_1, \xi_2, \dots, \xi_m\}$ that solve the \textbf{stationary relations}, we define the multiple radial SLE(0) Loewner chain as follows:

Let $\boldsymbol{\nu} = (\nu_1, \dots, \nu_n)$ be a set of parametrizations for the capacity, where each $\nu_i: [0, \infty) \to [0, \infty)$ is assumed to be measurable.

In the unit disk $\mathbb{D}$ with $u = 0$, we define the multiple radial SLE(0) Loewner chain as a normalized conformal map $g_t = g_t(z)$, with the initial condition $g_0(z) = z$ and the evolution given by the Loewner equation
\begin{equation}
\partial_t g_t(z) = \sum_{j=1}^n \nu_j(t) g_t(z) \frac{z_j(t) + g_t(z)}{z_j(t) - g_t(z)}, \quad g_0(z) = z.
\end{equation}

The Loewner chain for the covering map $h_t(z) = -i \log(g_t(e^{iz}))$ is given by
\begin{equation}
\partial_t h_t(z) = \sum_{j=1}^n \nu_j(t) \cot\left( \frac{h_t(z) - \theta_j(t)}{2} \right), \quad h_0(z) = z.    
\end{equation}

The driving functions $\theta_j(t)$, for $j = 1, \dots, n$, evolve as
\begin{equation}\label{multiple SLE 0 driving}
\dot{\theta}_j = \nu_j(t) \frac{\partial \log \mathcal{Z}(\boldsymbol{\theta}, \boldsymbol{\zeta})}{\partial \theta_j} + \sum_{k \neq j} \nu_k(t) \cot\left( \frac{\theta_j - \theta_k}{2} \right),
\end{equation}
where
\[
\mathcal{Z}(\boldsymbol{\theta},\boldsymbol{\zeta}) := \prod_{1 \leq j < k \leq n} \sin^2\left(\frac{\theta_j - \theta_k}{2}\right) 
\prod_{1 \leq s < t \leq m} \sin^8\left(\frac{\zeta_s - \zeta_t}{2}\right) 
\prod_{k=1}^{n} \prod_{l=1}^m \sin^{-4}\left(\frac{\theta_k - \zeta_l}{2}\right).
\]
The logarithmic derivative of $\mathcal{Z}(\boldsymbol{\theta}, \boldsymbol{\zeta})$ with respect to $\theta_j$ (treating $\theta$ and $\zeta$ as independent variables) is given by:
\begin{equation}
   \frac{\partial \mathcal{Z}(\boldsymbol{\theta},\boldsymbol{\zeta})}{\partial \theta_j}=\sum_{k\neq j} \cot\left(\frac{\theta_j-\theta_k}{2}\right)-2\sum_{l}\cot
   \left(\frac{2}{\theta_j-\zeta_l}\right)
\end{equation}
for $j=1,\ldots,n$,
The flow map $g_t$ is well-defined up to the first time $\tau$ at which $z_j(t) = z_k(t)$ for some $1 \leq j < k \leq n$. For each $z \in \mathbb{C}$, the process $t \mapsto g_t(z)$ is well-defined up to the time $\tau_z \wedge \tau$, where $\tau_z$ is the first time at which $g_t(z) = z_j(t)$. The hull associated with this Loewner chain is denoted by
\[
K_t = \left\{ z \in \overline{\mathbb{D}} : \tau_z \leq t \right\}.
\]

\end{defn}

\begin{remark}
The above definition of the multiple radial SLE(0) system is dynamic and local. We can define a multiple radial SLE(0) system for arbitrary initial positions of involution-symmetric $\boldsymbol{\xi}$ without assuming the stationary relations. When $\boldsymbol{\xi}$ solves the \textbf{stationary relations}, we will show that for any parametrization $\boldsymbol{\nu}(t)$, the traces are always the horizontal trajectories of a quadratic differential $Q(z)dz^2$, as stated in Theorem (\ref{traces as horizontal trajectories}). This provides another understanding of reparametrization symmetry (commutation relations) when $\kappa=0$.

\end{remark}

To characterize the traces of multiple radial SLE(0) systems, we introduce a class of quadratic differentials, denoted \( \mathcal{QD}(\boldsymbol{z}) \), with prescribed zeros at \( \boldsymbol{z} = \{z_1, z_2, \ldots, z_n\} \). These quadratic differentials are defined on the Riemann sphere and exhibit involution symmetry.

\begin{defn}[Quadratic differentials with prescribed zeros]
\label{trace quadratic differential}
Let \( \boldsymbol{z} = \{z_1, z_2, \ldots, z_n \} \) be distinct points on the unit circle. The class of quadratic differentials, denoted by \( \mathcal{QD}(\boldsymbol{z}) \).

\begin{enumerate}
    \item symmetric under the involution \( z^* = \frac{1}{\bar{z}} \), meaning
    \[
    \overline{Q(z^*)}\overline{(dz^*)^2} = Q(z)dz^2.
    \]

    \item  distinct zeros at \( \{z_1, z_2, \ldots, z_n\} \), each of order 2.

    \item  distinct finite poles at \( \{\xi_1, \ldots, \xi_m\} \), each of order 4, and the residues vanish (Residue-free condition):
    \[
    \text{Res}_{\xi_j}(\sqrt{Q(z)}dz) = 0, \quad \text{for } j = 1, \ldots, m.
    \]

    \item  poles of order \( n + 2 - 2m \) at the marked points \( 0 \) and \( \infty \). This ensures the total difference between the number of zeros and poles is \( -4 \).
\end{enumerate}

Here, the poles \( \{\xi_1, \ldots, \xi_m\} \) are finite, meaning they do not coincide with \( 0 \) or \( \infty \). The quadratic differential $Q(z) \in \mathcal{QD}(\boldsymbol{z})$ must take the following form:
\end{defn}
 $$
Q(z)= 
\frac{\prod_{k=1}^{m}\xi_{k}^2}{ \prod_{j=1}^{n}z_j}
z^{2m-n-2}\frac{\prod_{j=1}^{ n}\left(z-z_j\right)^2}{\prod_{k=1}^{m}\left(z-\xi_k\right)^4},
$$

By considering the primitive of $F(z)=\int \sqrt{Q(z)}dz$, we find that the residue-free quadratic differentials are natural generalization of rational functions, specifically designed to address the monodromy at 0, see section (\ref{residue free quadratic differential}).

The geometry of the horizontal trajectories of a quadratic differential \( Q(z) \in \mathcal{QD}(\boldsymbol{z}) \) is described as follows:

In the main theorem (\ref{traces as horizontal trajectories}), we show that the traces of the multiple radial SLE(0) systems correspond precisely to the horizontal trajectories of the class of residue-free quadratic differentials $Q(z) \in \mathcal{QD}(\boldsymbol{z})$ with limiting ends at $\boldsymbol{z}=\{z_1,z_2,\ldots,z_n\}$.

\begin{thm}\label{traces as horizontal trajectories}
Let $\boldsymbol{z}=\{z_1,z_2,\ldots,z_{n}\}$ be distinct growth points on the unit circle and screening charges $\xi =\{\xi_1,\xi_2,\ldots,\xi_{m} \}$ involution symmetric and solve the stationary relations.

There exists an $Q(z) \in \mathcal{QD}(\boldsymbol{z})$ with $\boldsymbol{\xi}$ as poles and $\boldsymbol{z}$ as zeros, 
the hulls $K_t$ generated by the Loewner flows with parametrization $\boldsymbol{\nu}(t)$ are subsets of the horizontal trajectories of $Q(z)dz^2$ with limiting ends at $\boldsymbol{z}$, up to any time $t$ before the collisions of any poles or critical points. Up to any such time
$$
Q(z) \circ g_t^{-1} \in \mathcal{QD}(\boldsymbol{z}(t)) 
$$
where $\boldsymbol{z}(t)$ is the location of the critical points at time $t$ under the multiple radial Loewner flow with parametrization $\boldsymbol{\nu}(t)$ .

\end{thm}

The key ingredient in the proof of theorem (\ref{traces as horizontal trajectories}) is the integral of motion for the Loewner flows. This integral of motion, denoted by $N_t(z)$, arises as the classical limit of a martingale observable inspired by conformal field theory. A detailed explanation of this construction can be found in Section \ref{Multiple radial Martingale Observable}. This approach can also be extended to various multiple SLE systems. For systematic and rigorous study of such conformal field theories, please refer to \cites{KM13,KM21}.

The closure of horizontal trajectories of $Q(z) \in \mathcal{QD}(\boldsymbol{z})$ with limiting ends at zeros $\{z_1,z_2,\ldots,z_n\}$ form radial topological link patterns, see theorem (\ref{horizontal trajectories form link pattern}) for detailed proof and section \ref{underscreening}, section \ref{overscreening} for figures illustrating the traces of the multiple radial SLE(0) systems.

The classification of multiple radial SLE($0$) is linked to the enumeration of the master function for trigonometric KZ equations. This connection touches on a rich and intricate area of enumerative geometry. 

As shown in \cites{S02a, S02b, SV03}, these studies establish connections between the space of critical points and tensor products of Verma modules and other algebraic structures, providing deeper insights into the monodromy of the KZ equations. They also demonstrate that the critical points of the master function can be interpreted as solutions to the Bethe Ansatz equations, thereby linking the study of KZ equations with integrable systems. 

Furthermore, when all ${z_1, z_2, \ldots, z_n}$ lie on the real line, it is shown in \cite{MTV09} that the enumeration of these critical points is equivalent to the real Shapiro-Shapiro conjecture in real enumerative geometry.

This is part of our ongoing research, and based on \cite{MV08}, we propose several illuminating conjectures about the enumeration problem in section \ref{Enumerative geometry of radial multiple SLE(0)}.

\subsection{Relations to classical Calegero-Sutherland systems}

From the Hamiltonian point of view, we show that the multiple radial SLE(0) Loewner growing with common parametrization of capacity (i.e. $\nu_j(t)=1$) are a special type of classical Calogero-Sutherland system. 

The \textbf{stationary relations} can be interpreted as initial conditions for the particles and $n$ quadratic null vector equations as $n$ null vector Hamiltonians, which are related to the classical Calegro-Sutherland Hamiltonian via the lax pair.
Furthermore, these null vector Hamiltonians induce commuting Hamiltonian flows along the submanifolds defined as the intersection of their level sets.

\begin{thm}  
\label{CS results kappa=0}  
From a dynamical system perspective, the driving functions of multiple radial SLE(0) systems are given by:
\[  
\dot{\theta}_j = U_j(\boldsymbol{\theta}) + \sum_{k \neq j} \cot\left(\frac{\theta_j - \theta_k}{2}\right),  
\]  
where \( U_j \) satisfies the quadratic null vector equation for a constant \( h \):  
\begin{equation}  
\frac{1}{2}U_j^2 + \sum_{k \neq j} \cot\left(\frac{\theta_k - \theta_j}{2}\right) U_k - \sum_{k \neq j} \frac{3}{2 \sin^2\left(\frac{\theta_j - \theta_k}{2}\right)} = h.  
\end{equation}  

\begin{itemize}  

\item[(i)] By introducing the momentum function \( p_j \), defined as  
\begin{equation}  
p_j = U_j + \sum_{k \neq j} \cot\left(\frac{\theta_j - \theta_k}{2}\right),  
\end{equation}  
we can reformulate the multiple radial SLE(0) system as a Calogero-Sutherland system. The momentum \( p_j \) satisfies the null vector Hamiltonian equation:  
\begin{equation}  
\mathcal{H}_j(\boldsymbol{\theta}, \boldsymbol{p}) = \frac{1}{2} p_j^2 - \sum_{k \neq j} \left(p_j + p_k\right) f_{jk} + \sum_k \sum_{l \neq k} f_{jk} f_{jl} - 2 \sum_{k \neq j} f_{jk}^2 = h - \frac{3(n - 1)}{2} - C_{n-1}^2.  
\end{equation}  

The total null vector Hamiltonian \( \mathcal{H} = \sum_j \mathcal{H}_j \) is equivalent to the classical Calogero-Sutherland Hamiltonian:  
\begin{equation}  
\mathcal{H} = \sum_j \frac{p_j^2}{2} - \sum_{1 \leq j < k \leq n} \frac{4}{\sin^2\left(\frac{\theta_j - \theta_k}{2}\right)} = n h - \frac{n(n^2 - 1)}{6}.  
\end{equation}  

\item[(ii)] The commutation relations between different growth pairs are expressed in terms of the Poisson bracket:  
\begin{equation}  
\left\{\mathcal{H}_j, \mathcal{H}_k\right\} = \frac{1}{f_{jk}^2} \left(\mathcal{H}_k - \mathcal{H}_j\right).  
\end{equation}  

Consequently, the vector flows \( X_{\mathcal{H}_j} \) induced by the Hamiltonians \( \mathcal{H}_j \) commute along the submanifolds \( N_c \):  
\begin{equation}  
N_c = \left\{ (\boldsymbol{\theta}, \boldsymbol{p}) : \mathcal{H}_j(\boldsymbol{\theta}, \boldsymbol{p}) = c, \, \text{for all } j \right\}.  
\end{equation}  

\end{itemize}  
\end{thm}

This relationship is a classical analog of the relation between multiple radial $SLE(\kappa)$ and quantum Calogero-Sutherland system, first discovered in \cite{DC07}.

\newpage

\section{Coulomb gas correlation and rational $SLE(\kappa)$}

\subsection{Schramm Loewner evolutions}

In this section, we briefly recall the basic defintions and properties of the chordal and radial SLE. We will describe the radial Loewner chain in $\mathbb{D}$, where $\mathbb{D}=\{z \in \mathbb{C}|| z \mid<1\}$ and chordal Loewner chain in $\mathbb{H}=\{ \textnormal{Im}(z) >0\}$.

\begin{defn}[Conformal radius]
The conformal radius of a simply connected domain $\Omega$ with respect to a point $z \in \Omega$, defined as
$$
\operatorname{CR}(\Omega, z):=\left|f^{\prime}(0)\right|,
$$
where $f: \mathbb{D} \rightarrow \Omega$ is a conformal map from the open unit disk $\mathbb{D}$ onto $\Omega$ with $f(0)=z$.
\end{defn}
\begin{defn}[Capacity in $\mathbb{D}$]

For any compact subset $K$ of $\overline{\mathbb{D}}$ such that $\mathbb{D} \backslash K$ is simply connected and contains 0 , let $g_K$ be the unique conformal map $\mathbb{D} \backslash K \rightarrow \mathbb{D}$ such that $g_K(0)=0$ and $g_K^{\prime}(0)>0$ . The conformal radius of $\mathbb{D} \backslash K$ is
$$
\operatorname{CR}(\mathbb{D} \backslash K):=\left(g_K^{\prime}(0)\right)^{-1} .
$$

The capacity of $K$ is
$$
\operatorname{cap}(K)=\log g_K^{\prime}(0)=-\log \mathrm{CR}(\mathbb{D} \backslash K,0).
$$

\end{defn}

\begin{defn}[Capacity in $\mathbb{H}$]
    For any compact subset $K \subset \overline{\mathbb{H}}$ such that $\mathbb{H} \backslash K$ is a simply connected domain. The half-plane capacity of a hull $K$ is the quantity
$$
\operatorname{hcap}(K):=\lim _{z \rightarrow \infty} z\left[g_K(z)-z\right] \text {, }
$$
where $g_K: \mathbb{H} \backslash K \rightarrow \mathbb{H}$ is the unique conformal map satisfying the hydrodynamic normalization $g(z)=z+O\left(\frac{1}{z}\right)$ as $z \rightarrow \infty$.
\end{defn}

\begin{defn}[Radial Loewner chain] Let $g_t$ satisfies the radial Loewner equation

\begin{equation}
\partial_t g_t(z)=g_t(z) \frac{\mathrm{e}^{\mathrm{i} \theta_t}+g_t(z)}{\mathrm{e}^{\mathrm{i} \theta_t}-g_t(z)}, \quad g_0(z)=z,
\end{equation}

where $t \mapsto \theta_t$ is real continuous and called the driving function. 
Let $K_t$ be the set of points $z$ in $\mathbb{D}$ such that the solution $g_s(z)$ blows up before or at time $t$.  $K_t$ is called the radial SLE hull driven by $\theta_t$.

Radial Loewner chain in arbitrary simply connected domain $\Omega \subsetneq \mathbb{C}$ with a marked interior point $u \in D$, is defined via a conformal map from $\mathbb{D}$ onto $\Omega$ sending 0 to $u$.

\end{defn}
\begin{defn}[Radial SLE($\kappa$)]
For $\kappa \geq 0$, the radial SLE($\kappa$) is the random Loewner chain in $\mathbb{D}$ from $1$ to $0$ driven by:
\begin{equation}
    \theta_t =\sqrt{\kappa}B_t,
\end{equation}
where $B_t$ is the standard Brownian motion.
\end{defn}

\begin{defn}[Characterization of radial SLE]
The radial SLE is a family $\mathbb{P}(\mathbb{D};\zeta,0)$ of probability measures on curves $\eta:[0, \infty) \rightarrow \overline{\mathbb{D}}$ with $\eta(0)=\zeta$ and parametrized by capacity satisfies the following properties:
\begin{itemize}
    \item (Conformal invariance) For all $a \in \mathbb{R}$, let $\rho_a(z)=\mathrm{e}^{\mathrm{i} a} z$ be the rotation map $\mathbb{D} \rightarrow \mathbb{D}$, the pullback measure $\rho_a^* \mathbb{P}(\mathbb{D};\zeta,0)=\mathbb{P}(\mathbb{D};e^{-ia}\zeta,0)$. From this, we may extend the definition to $\mathbb{P}(\Omega; a,b)$ in any simply connected domain $\Omega$ with an interior marked point $u$ by pulling back using a uniformizing conformal map $\Omega \rightarrow \mathbb{D}$ sending $u$ to 0.
\item (Domain Markov property) given an initial segment $\gamma[0, \tau]$ of the radial $\operatorname{SLE}_\kappa$ curve $\gamma \sim \mathbb{P}(\Omega ; x, y)$ up to a stopping time $\tau$, the conditional law of $\gamma[\tau, \infty)$ is the law $\mathbb{P}\left(\Omega \backslash K_\tau ; \gamma(\tau), 0\right)$ of the $\mathrm{SLE}_\kappa$ curve in the complement of the hull $K_\tau$ from the tip $\gamma(\tau)$ to $0$.
\item (Reflection symmetry) Let $\iota: z \mapsto \bar{z}$ be the complex conjugation, then $\mathbb{P}(\zeta,0) \sim \iota^* \mathbb{P}(\overline{\zeta},0)$.
\end{itemize}
\end{defn}

\begin{defn}[Chordal Loewner chain] Let $g_t$ satisfies the chordal Loewner equation

\begin{equation}
\partial_t g_t(z)= \frac{2}{g_t(z)-\xi(t)}, \quad g_0(z)=z,
\end{equation}

where $t \mapsto \xi_t$ is continuous and called the driving function. 
Let $K_t$ be the set of points $z$ in $\mathbb{H}$ such that the solution $g_s(z)$ blows up before or at time $t$.  $K_t$ is called the chordal SLE hull driven by $\xi_t$

Chordal Loewner chain in arbitrary simply connected domain $\Omega \subsetneq \mathbb{C}$ from $a$ to $b$, is defined via a uniformizing conformal map from $\Omega$ onto $\mathbb{H}$ sending $a$ to $0$ and $b$ to $\infty$.
\end{defn}

\begin{defn}[Chordal SLE($\kappa$)]
For $\kappa \geq 0$, the chordal SLE($\kappa$) is the random Loewner chain in $\mathbb{H}$ from $0$ to $\infty$ driven by
\begin{equation}
    \xi_t =\sqrt{\kappa}B_t,
\end{equation}
where $B_t$ is the standard Brownian motion.
\end{defn}

\begin{defn}[Characterization of Chordal SLE]
Chordal SLE is a family of probability measures on curves $\mathbb{P}(\mathbb{H};a,b)$ $\eta:[0, \infty] \rightarrow \overline{\mathbb{H}}$ with $\eta(0)=a, \eta(\infty)=b$ and parametrized by capacity satisfies the following properties:
    \begin{itemize}
     \item (Conformal invariance) $\rho(z)\in {\rm Aut}(\mathbb{H})$, the pullback measure $\rho^* \mathbb{P}(a,b)=\mathbb{P}(\mathbb{H};\rho(a),\rho(b))$. From this, we may extend the definition of to  $\mathbb{P}(\mathbb{H};z_1,z_2)$ in any simply connected domain $\Omega$ with two boundary points $z_1,z_2$  by pulling back using a uniformizing conformal map $\Omega \rightarrow \mathbb{H}$ sending $z_1$ to $a$ and $z_2$ to $b$.
    \item (Domain Markov property) given an initial segment $\gamma[0, \tau]$ of the $\operatorname{SLE}_\kappa$ curve $\gamma \sim \mathbb{P}(\Omega ; x, y)$ up to a stopping time $\tau$, the conditional law of $\gamma[\tau, \infty)$ is the law $\mathbb{P}\left(\Omega \backslash K_\tau ; \gamma(\tau), y\right)$ of the $\mathrm{SLE}_\kappa$ curve in the complement of the hull $K_\tau$ from the tip $\gamma(\tau)$ to $y$.
    \end{itemize} 
\end{defn}
\subsection{Coulomb gas correlation on Riemann sphere}

To define more general SLE processes beyond the chordal and radial SLEs, we introduce the concept of Coulomb gas correlations. These correlations serve as partition functions for various SLE processes and play a central role in conformal field theory.

We define the Coulomb gas correlations as the (holomorphic) differentials with conformal dimensions $\lambda_j=\sigma_j^2 / 2-\sigma_j b$ at $z_j$ (including infinity) and with values
$$
\prod_{\substack{j<k \\ z_j, z_k \neq \infty}}\left(z_j-z_k\right)^{\sigma_j \sigma_k}, \quad\left(z_j \in \widehat{\mathbb{C}}\right)
$$
in the identity chart of $\mathbb{C}$ and the chart $z \mapsto-1 / z$ at infinity. If $\sigma_j\sigma_k \notin 2\mathbb{Z}$, the Coulomb gas differential is multi-valued; in this case, we choose a single-valued branch.
After explaining this definition, we prove that under the neutrality condition, $\sum \sigma_j=2 b$, the Coulomb gas correlation functions are conformally invariant with respect to the Möbius group ${\rm Aut}(\widehat{\mathbb{C}})$.

\begin{defn}[Differential]\label{differential}
A local coordinate chart on a Riemann surface $M$ is a conformal map $\phi: U \rightarrow \phi(U) \subset \mathbb{C}$ on an open subset $U$ of $M$. A differential $f$ is an assignment of a smooth function $(f \| \phi): \phi(U) \rightarrow \mathbb{C}$ to each local chart $\phi: U \rightarrow \phi(U)$.  $f$ is a differential of conformal dimensions $\left[\lambda, \lambda_*\right]$ if for any two overlapping charts $\phi$ and $\tilde{\phi}$, we have:

\begin{equation}
(f \| \phi)=\left(h^{\prime}\right)^\lambda\left(\overline{h^{\prime}}\right)^{\lambda_*} (\tilde{f} \circ h \| \tilde{\phi}),
\end{equation}

where $h=\tilde{\phi} \circ \phi^{-1}: \phi(U \cap \tilde{U}) \rightarrow \tilde{\phi}(U \cap \tilde{U})$ is the transition map.

\begin{defn}[Neutrality Condition]
A divisor $\boldsymbol{\sigma} : \widehat{\mathbb{C}} \to \mathbb{R}$ is said to satisfy the \emph{neutrality condition} $\mathrm{(NC)}_b$ if
\begin{equation}
\int \boldsymbol{\sigma} = 2b,
\end{equation}
for some $b \in \mathbb{R}$. In the context of $\mathrm{SLE}_\kappa$, the parameter $b$ is related to $\kappa > 0$ by
\begin{equation}
b = \sqrt{\frac{8}{\kappa}} - \sqrt{\frac{\kappa}{2}}.
\end{equation}
\end{defn}

\begin{defn}[Coulomb gas correlations for a divisor on the Riemann sphere] Let the divisor
$$
\boldsymbol{\sigma}=\sum \sigma_j \cdot z_j,
$$
where $\left\{z_j\right\}_{j=1}^n$ is a finite set of distinct points on $\widehat{\mathbb{C}}$.
The Coulomb gas correlation $C_{(b)}[\boldsymbol{\sigma}]$ is a differential of conformal dimension $\lambda_j$ at $z_j$, given by

\begin{equation}
\lambda_j=\lambda_b\left(\sigma_j\right) \equiv \frac{\sigma_j^2}{2}-\sigma_j b ,
\end{equation}

where $\lambda_b(\sigma)=\frac{\sigma^2}{2}-\sigma b \quad(\sigma \in \mathbb{C})$
whose value is given by

\begin{equation}
C_{(b)}[\boldsymbol{\sigma}] =\prod_{j<k}\left(z_j-z_k\right)^{\sigma_j \sigma_k} ,
\end{equation}
where the product is taken over all finite $z_j$ and $z_k$. 

This defines a holomorphic function of $\boldsymbol{z}$ on the configuration space
\[
\mathbb{C}_{\mathrm{distinct}}^n = \left\{ \boldsymbol{z} = (z_1, \ldots, z_n) \in \mathbb{C}^n \,\middle|\, z_j \neq z_k \text{ for } j \neq k \right\}.
\]
In general, the function is multivalued, and one must choose a single-valued branch for each factor $(z_j-z_k)^{\sigma_j\sigma_k}$, except in special cases where all $\sigma_j$ are integers. If all $\sigma_j$ are even integers, the function becomes single-valued and independent of the ordering of the product. In the special case where $\sigma_j = 1$ for all $j$, the correlation function coincides with the Vandermonde determinant.

\end{defn}

\begin{thm}[see \cite{KM21} thm (\textcolor{red}{2.2})]
Under the neutrality condition $\left(\mathrm{NC}_b\right)$, the differentials $C_{(b)}[\boldsymbol{\sigma}]$ are Möbius invariant on $\hat{\mathbb{C}}$.
\end{thm}

\subsection{Coulomb gas correlation in a simply connected domain}

In this section, we define the Coulomb gas correlation differential in a simply connected domain.

\begin{defn}[Symmetric Riemann surface]
A symmetric Riemann surface is a pair $(S, j)$ consisting of a Riemann surface $S$ and an anticonformal involution $j$ on $S$. The latter means that $j$ : $S \rightarrow S$ is an anti-analytic map with $j \cdot j=$ id (the identity map). 
\end{defn}

The principal example for us is the symmetric Riemann surface obtained by taking the Schottky double of a simply connected domain domain. The construction of this is briefly as follows. (See section \textcolor{red}{2.2}, \cite{SS54}, \textcolor{red}{II.3E}, \cite{AS60} for details.)    

\begin{defn}[Schottky double]
Let $\Omega \subsetneq \mathbb{C}$ be a simply connected domain in $\mathbb{C}$ with $\Gamma=\partial \Omega$ consisting of prime ends. Take copy $\tilde{\Omega}$ of $\Omega$ and weld $\Omega$ and $\tilde{\Omega}$ together along $\Gamma$ so that a compact topological surface $\Omega^{Double}=\Omega \cup \Gamma \cup \tilde{\Omega}$ is obtained. If $z \in \Omega$ let $\tilde{z}$ denote the corresponding point on $\tilde{\Omega}$. Then an involution $j$ on $\Omega^{Double}$ is defined by
$$
\begin{array}{ll}
j(z)=\tilde{z} & \text { and } \\
j(\tilde{z})=z & \text { for } z \in \Omega, \\
j(z)=z & \text { for } z \in \Gamma .
\end{array}
$$

The conformal structure on $\tilde{\Omega}$ will be the opposite to that on $\Omega$, which means that the function $\tilde{z} \mapsto \overline{z}$ serves as a local variable on $\tilde{\Omega}$, and $j$ becomes anti-analytic. 

For $p\in \partial \Omega$, let $\phi:U\subset \overline{\Omega} \rightarrow \phi(U)$ be a local boundary chart at $p$, 
let $\tilde{U}$ be the corresponding subset in $\tilde{\Omega}$, 
then $\tilde{\phi}:\tilde{U} \subset \tilde{\Omega} \rightarrow \overline{\phi}(\tilde{U})$ is a local chart at $\tilde{p}$.
Then we can define a local chart $\tau$ for $\Omega^{Double}$ at boundary point $p$ by 
$$ \tau(z)=\left\{
\begin{aligned}
&\phi(z), z \in U \\
& \overline{\phi(z)}, z \in \tilde{U}.
\end{aligned}
\right.
$$

Thus, the conformal structure on $\Omega^{Double}$, inherited from $\mathbb{C}$, extends in a natural way across $\Gamma$ to a conformal structure on all of $\Omega^{Double}$. This makes $\Omega^{Double}$ into a symmetric Riemann sphere.

For example, we identify $\widehat{\mathbb{C}}$ with the Schottky double of $\mathbb{H}$ or that of $\mathbb{D}$. Then the corresponding involution $j$ is $j_{\mathbb{H}}: z \mapsto z^*=\bar{z}$ for $\Omega=\mathbb{H}$ and $j_{\mathbb{D}}: z \mapsto z^*=1 / \bar{z}$ for $\Omega=\mathbb{D}$.
  
\end{defn}

\begin{defn}[Double divisor]
Suppose $\Omega$ is a simply connected domain $(\Omega \subsetneq \mathbb{C})$.

A double divisor $\left(\boldsymbol{\sigma^{+}}, \boldsymbol{\sigma^{-}}\right)$ is a pair of divisor in  $\overline{\Omega}$

 \begin{equation}
  \boldsymbol{\sigma}^{+}=\sum \sigma_j^{+} \cdot z_j, \boldsymbol{\sigma}^{-}=\sum \sigma_j^{-} \cdot z_j.
 \end{equation}

We introduce an equivalence relation for double divisors:

\begin{equation}
    \left(\boldsymbol{\sigma_{1}^{+}}, \boldsymbol{\sigma_{1}^{-}}\right) \sim \left(\boldsymbol{\sigma_{2}^{+}}, \boldsymbol{\sigma_{2}^{-}}\right)
\end{equation}
 if and only if
 \begin{equation}
    \boldsymbol{\sigma_{1}^{+}}+ \boldsymbol{\sigma_{1}^{-}} = \boldsymbol{\sigma_{2}^{+}}+ \boldsymbol{\sigma_{2}^{-}} \quad on \ \partial \Omega.
\end{equation}
Thus, we may choose a representative $\boldsymbol{\sigma}^{-}$ from each equivalence class  that is supported in $\Omega$, i.e., $\sigma_j^{-}=0$ if $z_j \in \partial \Omega$ .

\end{defn}

\begin{defn}
Suppose $\Omega$ is a simply connected domain $(\Omega \subsetneq \mathbb{C})$, let $\partial \Omega$ be its Carathéodory boundary (prime ends) and consider the Schottky double $S=\Omega^{\text {double }}$, which equips with the canonical involution $\iota \equiv \iota_\Omega: S \rightarrow S, z \mapsto z^*$. 

Then, for a double divisor $\left(\boldsymbol{\sigma}^{+}, \boldsymbol{\sigma}^{-}\right)$, we define the associated divisor on the Schottky double $S$ by
\begin{equation}
\boldsymbol{\sigma} = \boldsymbol{\sigma}^{+} + \boldsymbol{\sigma}_*^{-}, \quad \text{where} \quad \boldsymbol{\sigma}_*^{-} := \sum \sigma_j^{-} \cdot z_j^*,
\end{equation}
and each $z_j^*$ denotes the image of $z_j$ under the canonical involution $\iota$ of $S$. Accordingly, $\boldsymbol{\sigma}_*^{-}$ is the pushforward of $\boldsymbol{\sigma}^{-}$ under $\iota$.

\end{defn}

\begin{defn}[Neutrality condition]
 A double divisor $\left(\boldsymbol{\sigma}^{+}, \boldsymbol{\sigma}^{-}\right)$ satisfies the neutrality condition $\left(\mathrm{NC}_b\right)$  if 
\begin{equation}
\int\boldsymbol{\sigma}=\int \boldsymbol{\sigma}^{+}+ \int \boldsymbol{\sigma}^{-}=2b.
\end{equation}   
\end{defn}

\begin{defn}[Coulomb gas correlation for a double divisor in a simply connected domain]
For a double divisor $(\boldsymbol{\sigma^+},\boldsymbol{\sigma^-})$, let $\boldsymbol{\sigma}=\boldsymbol{\sigma^+}+ \boldsymbol{\sigma}_*^{-}$ be its corresponding divisor in the Schottky double $S$, we define the Coulomb gas correlation of the double divisor $(\boldsymbol{\sigma^+},\boldsymbol{\sigma^-})$ by
\end{defn}
\begin{equation}
 C_{\Omega}\left[\boldsymbol{\sigma}^{+}, \boldsymbol{\sigma}^{-}\right](\boldsymbol{z}):=C_S[\boldsymbol{\sigma}] .  
\end{equation}

We often omit the subscripts $\Omega, S$ to simplify the notations. 

If the double divisor $(\boldsymbol{\sigma}^{+}, \boldsymbol{\sigma}^{-})$ satisfies the neutrality condition $(\mathrm{NC}_b)$, then the Coulomb gas correlation function
$
C_{\Omega}\left[\boldsymbol{\sigma}^{+}, \boldsymbol{\sigma}^{-}\right]
$
is a well-defined differential on $\Omega$, with conformal weights $\left[\lambda_j^{+}, \lambda_j^{-}\right]$ at each point $z_j \in \Omega$.

If $z_j \in \partial \Omega$, then the differential is with respect to a boundary chart: that is, a local conformal map from a neighborhood of $z_j$ in $\Omega$ to the upper half-plane $\mathbb{H}$, sending $z_j$ to a boundary point of $\mathbb{H}$. The derivative $\partial_{z_j}$ is then defined as the holomorphic derivative in this local coordinate.

\begin{equation}
\lambda^{+}_{j}=\lambda_b\left(\sigma^{+}_j\right) \equiv \frac{(\sigma^{+}_j)^2}{2}-\sigma_{j}^{+} b,\quad
\lambda^{-}_{j}=\lambda_b\left(\sigma_j\right) \equiv \frac{(\sigma^{-}_{j})^2}{2}-\sigma^{-}_j b.
\end{equation}

By conformal invariance of the Coulomb gas correlation differential $C_S[\boldsymbol{\sigma}]$ on the Riemann sphere under M{\"o}bius transformation, the Coulomb gas correlation differential $C_{\Omega}\left[\boldsymbol{\sigma}^{+}, \boldsymbol{\sigma}^{-}\right](\boldsymbol{z})$ is invariant under $Aut(\Omega)$.
\end{defn}

\begin{thm}[see \cite{KM21} thm (\textcolor{red}{2.4})]
 Under the neutrality condition $\left(\mathrm{NC}_b\right)$, the value of the differential $C_{\mathbb{H}}\left[\boldsymbol{\sigma}^{+}, \boldsymbol{\sigma}^{-}\right]$ in the identity chart of $\mathbb{H}$ (and the chart $z \mapsto-1 / z$ at infinity) is given by
 
 \begin{equation}
C_{\mathbb{H}}\left[\boldsymbol{\sigma}^{+}, \boldsymbol{\sigma}^{-}\right]=\prod_{j<k}\left(z_j-z_k\right)^{\sigma_j^{+} \sigma_k^{+}}\left(\bar{z}_j-\bar{z}_k\right)^{\sigma_j^{-} \sigma_k^{-}} \prod_{j, k}\left(z_j-\bar{z}_k\right)^{\sigma_j^{+} \sigma_i^{-}},
\end{equation}

where the products are taken over finite $z_j$ and $z_k$.
\end{thm}

\begin{example}
We have
\begin{itemize}
    \item[(i)] if $\boldsymbol{\sigma}^{-}=\mathbf{0}$, then (up to a phase)
$$
C_{\mathbb{H}}\left[\boldsymbol{\sigma}^{+}, \mathbf{0}\right]=\prod_{j<k}\left(z_j-z_k\right)^{\sigma_j^{+} \sigma_k^{+}};
$$
\item[(ii)] if $\boldsymbol{\sigma}^{-}=\overline{\boldsymbol{\sigma}^{+}}$, then (up to a phase)
$$
C_{\mathbb{H}}\left[\boldsymbol{\sigma}^{+}, \overline{\boldsymbol{\sigma}^{+}}\right]=\prod_{j<k}\left|\left(z_j-z_k\right)^{\sigma_j^{+} \sigma_k^{+}}\left(z_j-\bar{z}_k\right)^{\sigma_j^{+} \overline{\sigma_k^{+}}}\right|^2 \prod_{\operatorname{Im} z_j>0}\left(2 \operatorname{Im} z_j\right)^{\left|\sigma_j^{+}\right|^2} ;
$$
\item[(iii)] if $\boldsymbol{\sigma}^{-}=-\overline{\boldsymbol{\sigma}^{+}}$, then (up to a phase)
$$
C_{\mathbb{H}}\left[\boldsymbol{\sigma}^{+},-\overline{\boldsymbol{\sigma}^{+}}\right]=\prod_{j<k}\left|\left(z_j-z_k\right)^{\sigma_j^{+} \sigma_k^{+}}\left(z_j-\overline{z_k}\right)^{-\sigma_j^{+} \overline{\sigma_k^{+}}}\right|^2 \prod_{\operatorname{Im} z_j>0}\left(2 \operatorname{Im} z_j\right)^{-\left|\sigma_j^{+}\right|^2} .
$$

where the products are taken over finite $z_j$ and $z_k$.
\end{itemize}
\end{example}

\begin{thm}[see \cite{KM21} thm (\textcolor{red}{2.5})]
Under the neutrality condition $\left(\mathrm{NC}_b\right)$, the value of the differential $C_{\mathbb{D}}\left[\boldsymbol{\sigma}^{+}, \boldsymbol{\sigma}^{-}\right]$ in the identity chart of $\mathbb{D}$ is given by
\begin{equation}
C_{\mathbb{D}}\left[\boldsymbol{\sigma}^{+}, \boldsymbol{\sigma}^{-}\right]=\prod_{j<k}\left(z_j-z_k\right)^{\sigma_j^{+} \sigma_k^{+}}\left(\bar{z}_j-\bar{z}_k\right)^{\sigma_j^{-} \sigma_k^{-}} \prod_{j, k}\left(1-z_j \bar{z}_k\right)^{\sigma_j^{+} \sigma_k^{-}},
\end{equation}
where the product is taken over finite $z_j$ and $z_k$.
\end{thm}

\subsection{Rational $SLE_\kappa[\boldsymbol{\sigma}]$}

\begin{defn}[Rational SLE] \label{rational SLE}
In the unit disk $\mathbb{D}$, let $e^{i\theta} \in \partial\mathbb{D}$ be the growth point, and let $u_1 = e^{i\theta_1}, u_2 = e^{i\theta_2}, \dots, u_k = e^{i\theta_k} \in \overline{\mathbb{D}}$ be marked points. A symmetric double divisor $(\boldsymbol{\sigma}^+, \boldsymbol{\sigma}^-)$ assigns a charge distribution on $e^{i\theta}$ and $\{u_1, \dots, u_k\}$, where
\[
\boldsymbol{\sigma}^+ = a \cdot e^{i\theta} + \sum_{j=1}^{k} \sigma_j \cdot u_j,
\quad \text{and} \quad
\boldsymbol{\sigma}^- = \overline{\boldsymbol{\sigma}^+}|_{\mathbb{D}},
\]
and the total charge satisfies the neutrality condition $(\mathrm{NC}_b)$.

We define the rational $\mathrm{SLE}_\kappa[\boldsymbol{\sigma}]$ as a random normalized conformal map $g_t(z)$, with initial condition $g_0(z) = z$ and normalization $g_t'(0) = e^{-t}$. It evolves according to the radial Loewner equation:
\[
\partial_t g_t(z) = g_t(z) \frac{e^{i\theta(t)} + g_t(z)}{e^{i\theta(t)} - g_t(z)}, \quad g_0(z) = z.
\]

Let $h_t(z)$ be the covering map of $g_t(z)$, defined via
\[
e^{i h_t(z)} = g_t(e^{i z}),
\]
so that $h_0(z) = z$, and
\[
\partial_t h_t(z) = \cot\left( \frac{h_t(z) - \theta(t)}{2} \right).
\]

The driving function $\theta(t)$ evolves as
\[
d\theta(t) = \frac{\partial \log \mathcal{Z}(\theta)}{\partial \theta} \, dt + \sqrt{\kappa} \, dB_t,
\]
where the Coulomb gas partition function is
\begin{equation} \label{Coulomb-gas-theta}
\mathcal{Z}(\boldsymbol{\theta}) = \prod_{j < k} \sin\left( \frac{\theta_j - \theta_k}{2} \right)^{\sigma_j \sigma_k} 
\cdot \prod_j e^{\frac{i}{2} \sigma_j (\sigma_0 - \sigma_\infty)\theta_j}.
\end{equation}

The flow map $g_t$ is well-defined up to the first time $\tau$ when $\zeta(t) = g_t(w)$ for some $w$ in the support of $\boldsymbol{\sigma}$. For any $z \in \mathbb{D}$, the process $t \mapsto g_t(z)$ is defined up to time $\tau_z \wedge \tau$, where $\tau_z$ is the first time such that $g_t(z) = e^{i\theta(t)}$. Define the associated hull by
\[
K_t = \left\{ z \in \overline{\mathbb{D}} : \tau_z \le t \right\}.
\]

Furthermore, the law of the rational $\mathrm{SLE}_\kappa[\boldsymbol{\sigma}]$ Loewner chain is invariant under Möbius transformations $\mathrm{Aut}(\mathbb{D})$, up to a time change, due to the conformal covariance of the Coulomb gas correlation functions. Thus, we define rational $\mathrm{SLE}_\kappa[\boldsymbol{\sigma}]$ in a general simply connected domain $\Omega$ via a conformal map $\phi: \Omega \to \mathbb{D}$ by pulling back the flow.
\end{defn}

In definition (\ref{rational SLE}), we define the rational SLE from the perspective of the partition function. This approach helps us to understand the SLE within the framework of conformal field theory and can be naturally extended to various settings, including multiple SLE($\kappa$) systems.

\begin{example}Double divisor for chordal and radial $\mathrm{SLE}(\kappa, \rho)$, where $\xi$ denotes the growth point and $q$ is the marked boundary point (in the chordal case) or interior point (in the radial case). 

\end{example}

\begin{figure}[htbp]
\centering
\begin{minipage}[b]{0.49\textwidth}
\centering
\includegraphics[width=5.9cm]{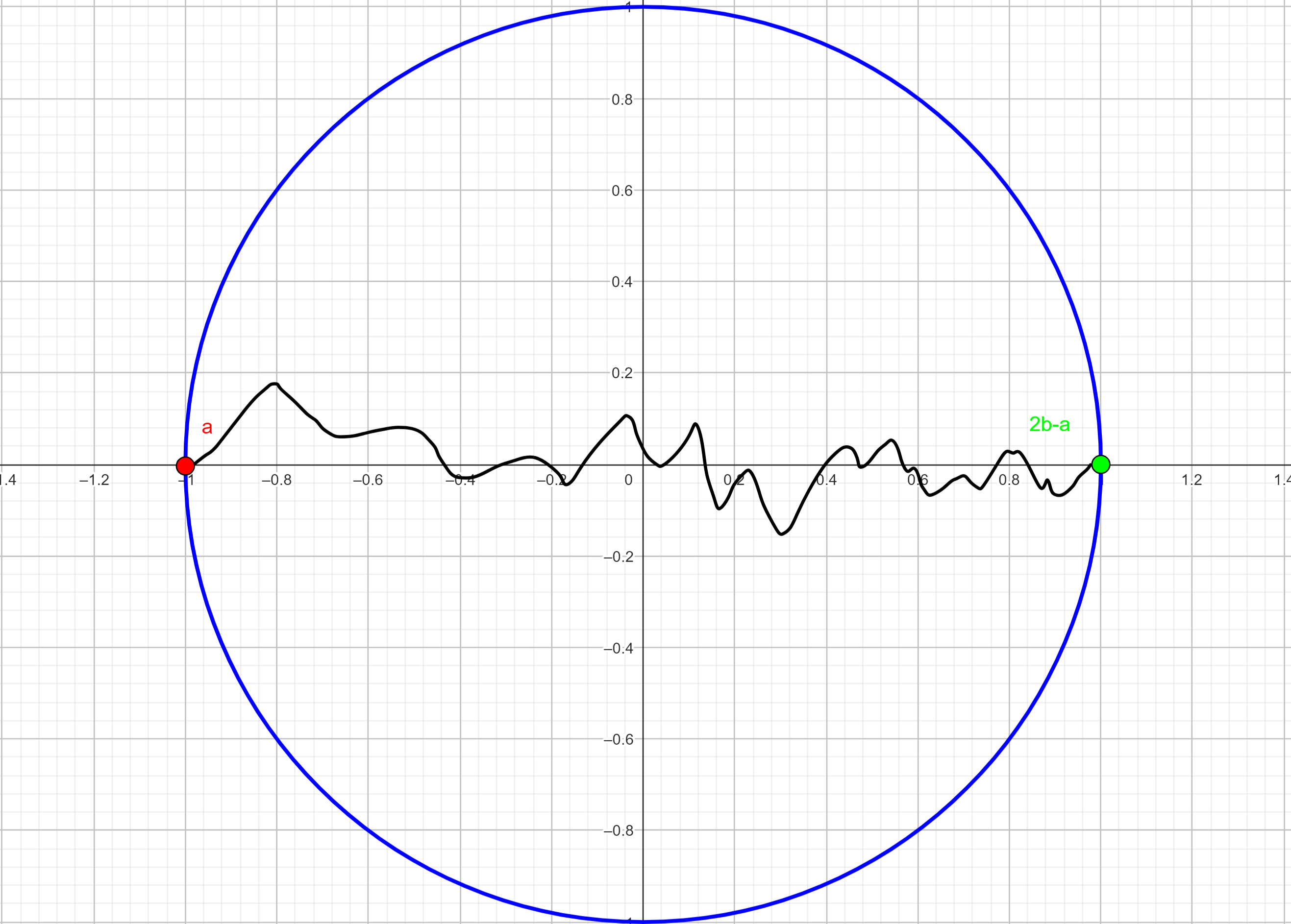}
\caption{Chordal SLE($\kappa$) $\boldsymbol{\sigma}^+= a \cdot \xi+ (2b-a) \cdot q$,
$\boldsymbol{\sigma}^-= 0$}
\end{minipage}
\begin{minipage}[b]{0.49\textwidth}
\centering
\includegraphics[width=5.9cm]{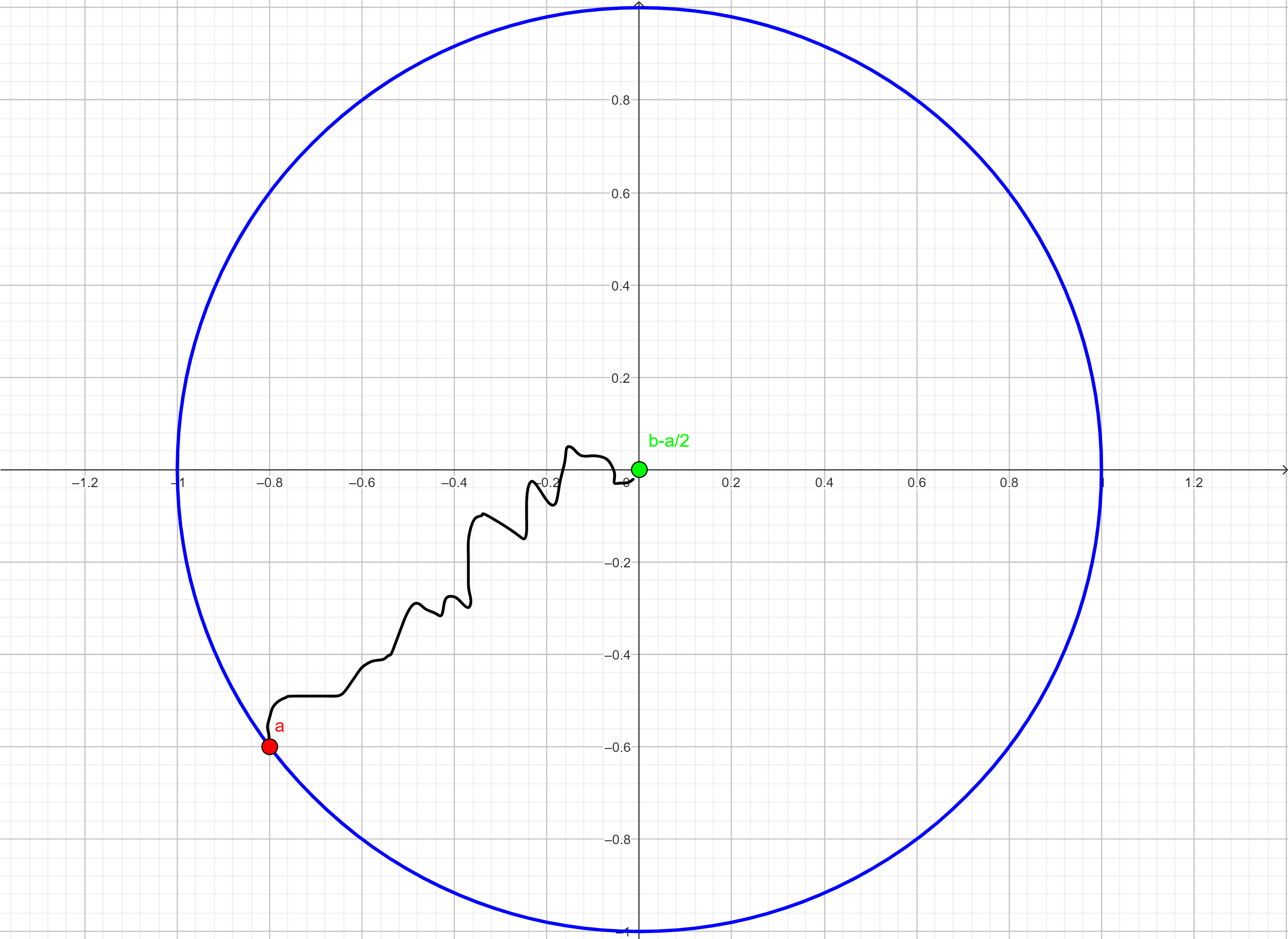}
\caption{Radial SLE($\kappa$)
$\boldsymbol{\sigma}^+= a \cdot \xi+ (b-a) \cdot q$,
$\boldsymbol{\sigma}^-= b \cdot q$}
\end{minipage}

\end{figure}

In addition to the aforementioned definition, another widely used equivalent is known as SLE($\kappa$,$\rho$). We prove the equivalence between rational $SLE_{\kappa}[\boldsymbol{\sigma^+},\boldsymbol{\sigma^-}]$ and SLE($\kappa$,$\rho$) in the following theorem.

\begin{defn}[Radial SLE($\kappa,\rho$)] Let $\xi$ be the growth point on the unit circle, and let
\[
\boldsymbol{\rho} = \sum_{j=1}^{n} \rho_j \delta_{u_j} + \sigma_0 \cdot \delta_0 + \sigma_{\infty} \cdot \delta_{\infty}
\]
be a divisor on $\widehat{\mathbb{C}}$, where $\rho_j \in \mathbb{C}$, and the divisor $\boldsymbol{\rho}$ is symmetric under inversion, i.e.,
\[
\boldsymbol{\rho}(z) = \overline{\boldsymbol{\rho}\left(\frac{z}{|z|^2}\right)} \quad \text{for all } z \in \widehat{\mathbb{C}}.
\]
We say $\boldsymbol{\rho}$ satisfies the neutrality condition for $\mathrm{SLE}(\kappa,\rho)$ if
\[
\int \boldsymbol{\rho} = \kappa - 6.
\].

Define the radial $\operatorname{SLE}(\kappa, \xi, \boldsymbol{\rho})$ Loewner chain by

\begin{equation}
\partial_t g_t(z)=g_t(z)\frac{\xi(t)+g_t(z)}{\xi(t)-g_t(z)}, \quad g_0(z)=z.
\end{equation}

Let $\xi(t)=e^{i\theta(t)}$, $u_j=e^{iq_j}$ and $h_t(z)$ be the covering map of $g_t(z)$ (i.e. $h_t(z)=g_t(e^{iz})$) , then the Loewner differential equation for $h_t(z)$ is given by
\begin{equation}
\partial_t h_t(z)=\cot(\frac{h_t(z)-\theta(t)}{2}), \quad h_0(z)=z,  
\end{equation}
 the driving function $\theta(t)$ evolves as

\begin{equation}
d\theta(t)=\sqrt{\kappa}dB_t+\sum_j \rho_j \cot( \frac{\theta(t)-q_j(t)}{2}).
\end{equation}

\end{defn}
Note that although the lifts of $\theta(t)$ in universal cover are not unique, different lifts lead to the same differential equation for $h_t(z)$ by periodicity $\cot(z+k\pi)=\cot(z)$, $k \in \mathbb{Z}$.

\begin{thm}
For a symmetric double divisor  $\boldsymbol{\sigma^+}=a\cdot \xi+ \sum \sigma_j \cdot u_j$ and $\boldsymbol{\sigma^-} = \overline{\boldsymbol{\sigma}^+}|_{\Omega}$ satisfying neutrality condition ($NC_{b}$), let $\boldsymbol{\rho}=\sum_{j=1}^{m}\rho_j \cdot u_j$ where $\rho_j = (\kappa a )\sigma_j$. Then two definitions $SLE_{\kappa}[\boldsymbol{\sigma^+},\boldsymbol{\sigma^-}]$ and SLE($\kappa,\rho$) are equivalent.
\end{thm}

\begin{proof}
The equivalence in one chart can be verified by directly computing the drift term in the Loewner equation. The conformal invariance of $\mathrm{SLE}(\kappa,\rho)$ under the neutrality condition ($NC_b$), where the divisor $\boldsymbol{\rho}$ consists of real charges, is established in \cite{SW05}. Moreover, their argument extends naturally to the case where the charges $\boldsymbol{\rho}$ are complex.
\end{proof}

\subsection{Classical limit of Coulomb gas correlation}
Now, we extend our definition of Coulomb gas correlation to $\kappa=0$ by normalizing the Coulomb gas correlation.

\begin{defn}[Normalized Coulomb gas correlations for a divisor on the Riemann sphere] Let the divisor
$$
\boldsymbol{\sigma}=\sum \sigma_j \cdot z_j,
$$
where $\left\{z_j\right\}_{j=1}^n$ is a finite set of distinct points on $\widehat{\mathbb{C}}$.
The normalized Coulomb gas correlation $C[\boldsymbol{\sigma}]$ is a differential of conformal dimension $\lambda_j$ at $z_j$ by

Let $\lambda(\sigma)=\sigma^2+2\sigma  \quad(\sigma \in \mathbb{R})$.
\begin{equation}
\lambda_j=\lambda_b\left(\sigma_j\right) \equiv \sigma_j^2+2\sigma_j  ,
\end{equation}

whose value is given by

\begin{equation}
C[\boldsymbol{\sigma}] =\prod_{j<k}\left(z_j-z_k\right)^{2\sigma_j \sigma_k} ,
\end{equation}
where the product is taken over all finite $z_j$ and $z_k$.

\end{defn}

\begin{remark}
    The normalized Coulomb gas correlation can be viewed as taking the $\kappa \rightarrow 0$ limit of the divisor $\sqrt{2\kappa}\boldsymbol{\sigma}$, the Coulomb gas correlation function $C_{(b)}[\boldsymbol{\sigma}]^{\kappa}$, and conformal dimension $\kappa \lambda_j$.
\end{remark}

\begin{defn}[Neutrality condition] A divisor $\boldsymbol{\sigma}: \widehat{\mathbb{C}} \rightarrow \mathbb{R}$ satisfies the neutrality condition if
\begin{equation}
\int \boldsymbol{\sigma}=-2.
\end{equation}
\begin{thm}
Under the neutrality condition $\int\boldsymbol{\sigma}=-2$, the normalized Coulomb gas correlation differentials $C[\boldsymbol{\sigma}]$ are Möbius invariant on $\hat{\mathbb{C}}$.
\end{thm}

\begin{proof}
By direct computation, similar to the $\kappa>0$ case.
\end{proof}

\begin{defn}[Coulomb gas correlation for a double divisor in a simply connected domain]
For a double divisor $(\boldsymbol{\sigma^+},\boldsymbol{\sigma^-})$, let $\boldsymbol{\sigma}=\boldsymbol{\sigma^+}+ \boldsymbol{\sigma}_*^{-}$ be its corresponding divisor in the Schottky double $S$, we define the Coulomb gas correlation of the double divisor $(\boldsymbol{\sigma^+},\boldsymbol{\sigma^-})$ by:
\end{defn}
\begin{equation}
 C_{\Omega}\left[\boldsymbol{\sigma}^{+}, \boldsymbol{\sigma}^{-}\right](\boldsymbol{z}):=C_S[\boldsymbol{\sigma}] .  
\end{equation}

We often omit the subscripts $\Omega, S$ to simplify the notations. 

If the double divisor $(\boldsymbol{\sigma}^{+},\boldsymbol{\sigma}^{-})$ satisfies the neutrality condition, then $C\left[\boldsymbol{\sigma}^{+}, \boldsymbol{\sigma^{-}}\right]$is a well-defined differential with conformal dimensions $\left[\lambda_j^{+}, \lambda_j^{-}\right]$at $z_j$.

\begin{equation}
\lambda^{+}_{j}=\lambda\left(\sigma^{+}_j\right) \equiv \frac{(\sigma^{+}_j)^2}{2}+2\sigma_{j}^{+} ,\quad
\lambda^{-}_{j}=\lambda \left(\sigma_j\right) \equiv \frac{(\sigma^{-}_{j})^2}{2}+2\sigma^{-}_j .  
\end{equation}

By conformal invariance of the Coulomb gas correlation differential $C_S[\boldsymbol{\sigma}]$ on the Riemann sphere under M{\"o}bius transformation, the Coulomb gas correlation differential $C_{\Omega}\left[\boldsymbol{\sigma}^{+}, \boldsymbol{\sigma}^{-}\right](\boldsymbol{z})$ is invariant under $Aut(\Omega)$.
\end{defn}

\begin{defn}[Neutrality condition]
 A double divisor $\left(\boldsymbol{\sigma}^{+}, \boldsymbol{\sigma}^{-}\right)$ satisfies the neutrality condition  if 
\begin{equation}
\int\boldsymbol{\sigma}=\int \boldsymbol{\sigma}^{+}+ \int \boldsymbol{\sigma}^{-}=-2.
\end{equation}   
\end{defn}

\begin{thm}
 Under the neutrality condition $\int \boldsymbol{\sigma}^{+}+ \int \boldsymbol{\sigma}^{-}=-2$ , the value of the differential $C_{\mathbb{H}}\left[\boldsymbol{\sigma}^{+}, \boldsymbol{\sigma}^{-}\right]$ in the identity chart of $\mathbb{H}$ (and the chart $z \mapsto-1 / z$ at infinity) is given by
 
 \begin{equation}
C_{\mathbb{H}}\left[\boldsymbol{\sigma}^{+}, \boldsymbol{\sigma}^{-}\right]=\prod_{j<k}\left(z_j-z_k\right)^{2\sigma_j^{+} \sigma_k^{+}}\left(\bar{z}_j-\bar{z}_k\right)^{2\sigma_j^{-} \sigma_k^{-}} \prod_{j, k}\left(z_j-\bar{z}_k\right)^{2\sigma_j^{+} \sigma_i^{-}},
\end{equation}

where the products are taken over finite $z_j$ and $z_k$.
\end{thm}

\begin{thm}
Under the neutrality condition $\int \boldsymbol{\sigma}^{+}+ \int \boldsymbol{\sigma}^{-}=-2$, the value of the differential $C_{\mathbb{D}}\left[\boldsymbol{\sigma}^{+}, \boldsymbol{\sigma}^{-}\right]$ in the identity chart of $\mathbb{D}$ is given by
\begin{equation}
C_{\mathbb{D}}\left[\boldsymbol{\sigma}^{+}, \boldsymbol{\sigma}^{-}\right]=\prod_{j<k}\left(z_j-z_k\right)^{2\sigma_j^{+} \sigma_k^{+}}\left(\bar{z}_j-\bar{z}_k\right)^{2\sigma_j^{-} \sigma_k^{-}} \prod_{j, k}\left(1-z_j \bar{z}_k\right)^{2\sigma_j^{+} \sigma_k^{-}} ,
\end{equation}
where the product is taken over finite $z_j$ and $z_k$.
\end{thm}

\subsection{Rational $SLE_{0}[\boldsymbol{\sigma}]$}
\begin{defn}[Rational $\mathrm{SLE}_0$] \label{rational SLE0}
In the unit disk $\mathbb{D}$, let $e^{i\theta} \in \partial \mathbb{D}$ be the growth point, and let $u_1, u_2, \dots, u_m \in \overline{\mathbb{D}}$ be marked points. A symmetric double divisor $(\boldsymbol{\sigma}^+, \boldsymbol{\sigma}^-)$ assigns a charge distribution on $e^{i\theta}$ and $\{u_1, \dots, u_k\}$, where
\[
\boldsymbol{\sigma}^+ = a \cdot e^{i\theta} + \sum_{j=1}^{k} \sigma_j \cdot u_j,
\quad \text{and} \quad
\boldsymbol{\sigma}^- = \overline{\boldsymbol{\sigma}^+}|_{\mathbb{D}},
\]
and the total charge satisfies the neutrality condition $\int\boldsymbol{\sigma}=-2$.

We define the rational $\mathrm{SLE}_0[\boldsymbol{\sigma}]$ Loewner chain as a normalized conformal map $g_t(z)$ with initial conditions $g_0(z) = z$ and $g_t'(0) = e^{-t}$. The evolution of $g_t$ is governed by the Loewner differential equation:
\[
\partial_t g_t(z) = g_t(z) \frac{e^{i\theta(t)} + g_t(z)}{e^{i\theta(t)} - g_t(z)}, \quad g_0(z) = z.
\]

In the angular coordinate, let $h_t(z)$ be the covering map of $g_t(z)$ defined by $e^{i h_t(z)} = g_t(e^{i z})$. Then $h_t(z)$ evolves according to
\[
\partial_t h_t(z) = \cot\left( \frac{h_t(z) - \theta(t)}{2} \right), \quad h_0(z) = z.
\]

The driving function $\theta(t)$ evolves according to
\[
d\theta(t) = \frac{\partial \log \mathcal{Z}(\boldsymbol{\theta})}{\partial \theta} \, dt.
\]

where the Coulomb gas partition function is
\begin{equation} 
\mathcal{Z}(\boldsymbol{\theta}) = \prod_{j < k} \sin\left( \frac{\theta_j - \theta_k}{2} \right)^{\sigma_j \sigma_k} 
\cdot \prod_j e^{\frac{i}{2} \sigma_j (\sigma_0 - \sigma_\infty)\theta_j}.
\end{equation}

The flow $g_t$ is well-defined up to the first time $\tau$ at which $w(t) = g_t(w)$ for some $w$ in the support of $\boldsymbol{\sigma}$. For each $z \in \overline{\mathbb{D}}$, the process $t \mapsto g_t(z)$ is well-defined up to $\tau_z \wedge \tau$, where $\tau_z$ is the first time such that $g_t(z) = e^{i\theta(t)}$. Denote
\[
K_t = \left\{ z \in \overline{\mathbb{D}} : \tau_z \leq t \right\}
\]
as the hull associated with the Loewner chain.

Furthermore, the rational $\mathrm{SLE}_0$ Loewner chain is invariant under Möbius transformations in $\mathrm{Aut}(\mathbb{D})$ (up to a time reparameterization), due to the conformal invariance of the Coulomb gas correlation. Consequently, rational $\mathrm{SLE}_0[\boldsymbol{\sigma}]$ in any simply connected domain $\Omega$ is defined by pulling back via a conformal map $\phi: \Omega \rightarrow \mathbb{D}$.
\end{defn}

In definition (\ref{rational SLE0}), we introduce the definition of $SLE_{0}[\beta]$ as a natural extension of $SLE_{\kappa}[\beta]$ to $\kappa=0$. The main ingredient in our definition is the normalized Coulomb gas as the partition function.

Now, we introduce another widely used definition ${\rm SLE}(0,\boldsymbol{\rho})$ which is a natural extension of ${\rm SLE}(\kappa,\boldsymbol{\rho})$ to $\kappa=0$.
We prove the equivalence between rational $SLE_0[\boldsymbol{\sigma}]$ and ${\rm SLE}(0,\boldsymbol{\rho})$ in the end.

\begin{defn}[${\rm SLE}(0,\boldsymbol{\rho})$]
 Let $w$ be the growth point on $\partial \mathbb{D}$ and $\boldsymbol{\rho}=\sum_{i=1}^{n} \rho_j \delta_{u_j} + \sigma_0 \cdot 0 + \sigma_{\infty} \cdot \infty$ be a divisor on $\widehat{\mathbb{C}}$ that is symmetric under involution, i.e. $\boldsymbol{\rho}(z)=\boldsymbol{\rho}(\frac{z}{|z|^2})$ for all $z$ and $\int \boldsymbol{\rho}= -6$. Define the radial $\operatorname{SLE}(0,w, \boldsymbol{\rho})$ Loewner chain by
\begin{equation}
\partial_t g_t(z)=g_t(z)\frac{w(t)+g_t(z)}{w(t)-g_t(z)}, \quad g_0(z)=z,
\end{equation}
where the driving function $w(t)$ evolves as

\begin{equation}
\dot{w}(t)=w(t)\sum_j \rho_j\frac{g_t(u_j)+w(t)}{g_t(u_j)-w(t)}, \quad z(0)=z_0.
\end{equation}

In the angular coordinate, $w(t)=e^{i\theta(t)}$ and $u_j(t)= e^{i q_j(t)}$, let $h_t(z)$ be the covering conformal map of $g_t(z)$ (i.e. $e^{ih_t(z)}=g_t(e^{iz}))$.

Then the Loewner differential equation for $h_t(z)$ is 

\begin{equation}
\partial_t h_t(z)=\cot(\frac{h_t(z)-\theta_t}{2}), \quad h_0(z)=z,     z\in{\overline{\mathbb{H}}},
\end{equation}

where the driving function $\theta_t$ evolves as
\begin{equation}
\dot{\theta}_t=\sum_j \rho_j \cot( \frac{\theta_t-q_j(t)}{2}), \quad x(0)=x_0.
\end{equation}

\end{defn}

\begin{thm}
  For an involution symmetric divisor $\boldsymbol{\sigma}= w + \sum_{j=1}^{m}\sigma_j \cdot z_j$ satisfying neutrality condition $\int \boldsymbol{\sigma}= -2$, let $\boldsymbol{\rho}=2\sum_{j=1}^{m}\sigma_j \cdot z_j$, then  $\int \boldsymbol{\rho}= -6$ and two definitions $SLE_{0}[\sigma]$ and SLE($0,\rho$) are equivalent.  
\end{thm}

\begin{proof}
The equivalence in one chart can be verified by directly computing the drift term in the Loewner equation. The conformal invariance of $\mathrm{SLE}(\kappa,\rho)$ under the neutrality condition ($NC_b$), where the divisor $\boldsymbol{\rho}$ consists of real charges, is established in \cite{SW05}. Moreover, their argument extends naturally to the case where the charges $\boldsymbol{\rho}$ are complex.
\end{proof}

\section{Commutation relations and conformal invariance} \label{Commutation relations and conformal invariance}

\subsection{Transformation of Loewner flow under coordinate change}
\label{transformation of Loewner under coordinate change}

In this section we show that the Loewner chain of a curve, when viewed in a different coordinate chart, is a time reparametrization of the Loewner chain in the standard coordinate chart but with different initial conditions. This result serves as a preliminary step towards understanding the local commutation relations and the conformal invariance of multiple SLE($\kappa$) systems.

We are particularly interested the following cases, which 
\begin{itemize}
    \item For local commutation relations, we study the Loewner chain for $\gamma_j$ in the coordinates induced by the conformal maps $h_{k, t}$ for $k \neq j$, 
    \item For conformal invariance with respect to ${\rm Aut} (\mathbb{D})$, we study the Loewner chain for each $\gamma_j$ in the coordinate induced by Möbius transformations $\tau \in {\rm}\mathbb{D}$.

\end{itemize}

Let us briefly review how Loewner chains transform under coordinate changes.

\begin{thm}[Deterministic Loewner chain coordinate change]
\label{Loewner coordinate change in angular coordinate}
In the angular coordinate, we assume that $\gamma(0)=\theta \in \mathbb{R}$, $\gamma(t)$ is generated by the Loewner chain:

\begin{equation}
\quad \partial_t h_t(z)=\cot(\frac{h_t(z )-W_t}{2}), \dot{\theta}_t=b\left(W_t ; h_t\left(W_1\right), \ldots, h_t\left(W_n\right)\right) \quad h_0(z)=z, W_0=\theta
\end{equation}

where $\dot{W}_t=b\left(W_t , g_t\left(z_1\right), \ldots, g_t\left(W_m\right)\right)$ for some $b: \mathbb{R} \times \mathbb{C}^n \rightarrow \mathbb{R}$. 

Under a conformal transformation $\tau: \mathcal{N} \rightarrow \mathbb{H}$, defined in a neighborhood $\mathcal{N}$ of $\theta$ such that $\gamma[0, T] \subset \mathcal{N}$ and such that $\tau$ sends $\partial \mathcal{N} \cap \mathbb{R}$ to $\mathbb{R}$, the Loewner chain of the image curve $\tilde{\gamma}(t)=\tau \circ \gamma(t)$ is as follows for $0 \leq t \leq T$. Let $\tilde{h}_t$ denote the unique conformal transformation associated to $\tilde{\gamma}[0, t]$ and $\Psi_t=\tilde{h}_t \circ \tau \circ h_t^{-1}$, then it can be computed that $\tilde{h}_t(z)$ evolves as

\begin{equation}
\partial_t \tilde{h}_t(z)=\cot(\frac{\tilde{h}_t(z)-\tilde{W}_t}{2}) \Psi_t^{\prime}\left(W_t\right)^2, \quad \tilde{h}_0(z)=z, \widetilde{W}_0=\tau(W_0),
\end{equation}

where $\widetilde{W}_t=\widetilde{h}_t \circ \tau \circ \gamma(t)=\widetilde{h}_t \circ \tau \circ h_t^{-1}\left(W_t\right)=\Psi_t\left(W_t\right)$ is the driving function for the new flow. Note that $\widetilde{W}_0=\tau\left(W_0\right)$. The equation for $\partial_t \widetilde{h}_t(z)$ shows that $\widetilde{\gamma}$ is parameterized so that its unit disk capacity satisfies $\operatorname{hcap}(\widetilde{\gamma}[0, t])=2 \sigma(t)$, where

\begin{equation}
\sigma(t):=\int_0^t \Psi_s^{\prime}\left(W_s\right)^2 d s .
\end{equation}
\end{thm}

\begin{thm}[Stochastic Loewner chain coordinate change]
\label{random coordinate change}
If the driving function $W_t$ is given by:
\begin{equation}
dW_t= \sqrt{\kappa}dB_t+ b(W_t;\Psi_t(W_1),\ldots,\Psi_t(W_n))
\end{equation}
Let $\Psi_t$ defined as above and
$$\widetilde{W}_t=\Psi_t\left(W_t\right)$$ then, using the identity 
$\left(\partial_t \Psi_t\right)\left(W_t\right)= -3\Psi_{t}''(W_t)$, see proposition (4.43) in \cite{Law05}
\begin{equation}
\begin{aligned}
d \widetilde{W}_t= &\left(\partial_t \Psi_t\right)\left(W_t\right) d t+\Psi_t^{\prime}\left(W_t\right) d W_t+\frac{\kappa}{2} \Psi_t^{\prime \prime}\left(W_t\right) d t \\
=&\sqrt{\kappa}h_{t}'(W_t)dB_t+ \Psi_{t}'(W_t)b(W_t;\Psi_t(W_1),\ldots,\Psi_t(W_n))+\frac{\kappa-6}{2} \Psi_t^{\prime \prime}\left(W_t\right) d t 
\end{aligned}
\end{equation}

By changing time from $t$ to $s(t)=\int|\Psi_t^{\prime}\left(W_t\right)|^2 dt$, 
we have

\begin{equation}
d \widetilde{W}_s= \sqrt{\kappa}dB_s+ \Psi_{t(s)}^{\prime}(W_s)^{-1} b(W_{s};\Psi_{t(s)}(W_1),\ldots,\Psi_{t(s)}(W_n))ds+\frac{\kappa-6}{2} \frac{\Psi_{t(s)}^{\prime \prime}\left(W_{s}\right)}{\Psi_{t(s)}^{\prime }(W_s)^2} d s
\end{equation}
\end{thm}

\begin{remark}\label{drift term pre schwarz form}
By theorem (\ref{random coordinate change}), for conformal transformation $\tau$, the drift term in the marginal law is a pre-schwarz form, i.e.
$b=\tau^{\prime} \widetilde{b} \circ \tau+ \frac{6-\kappa}{2}\left(\log \tau^{\prime}\right)^{\prime}$.
\end{remark}

\begin{cor}\label{gamma 1 gamma2 capacity change}
Let $\gamma$, $\tilde{\gamma}$ be two hulls starting at $e^{ix} \in \partial \mathbb{D}$ and $e^{iy} \in \partial \mathbb{D}$ with capacity $\epsilon$ and $c \epsilon$ , let $g_{\epsilon}$ be the normalized map removing $\gamma$ and $\tilde{\epsilon}= \operatorname{hcap}(g_{\epsilon}\circ \gamma(t))$, then we have:

\begin{equation}
\tilde{\varepsilon}=
c \varepsilon\left(1-\frac{\varepsilon}{\sin^2(\frac{x-y}{2})}\right)+o\left(\varepsilon^2\right)
\end{equation}

\end{cor}

\begin{proof}
Locally, we can define $h_t(z)=-i\log(g_t(e^{iz}))$.
Then from the Loewner equation, $\partial_t h_t^{\prime}(w)=- \frac{h_t^{\prime}(w)}{2\sin^2\left(\frac{h_t(w)-x_t}{2}\right)}$, which implies $h_{\varepsilon}^{\prime}(y)=1-\frac{\varepsilon}{2\sin^2(\frac{y-x}{2})}+o(\varepsilon)$. By applying conformal transformation $h_{\epsilon}(y)$, we get:
$$
\tilde{\varepsilon}=c\epsilon( h'_{\epsilon}(y)^2+ o(\epsilon)) =c \varepsilon\left(1-\frac{\varepsilon}{\sin^2(\frac{x-y}{2})}\right)+o\left(\varepsilon^2\right)
$$

\end{proof}

\subsection{Local commutation relation and null vector equations in $\kappa=0$ case}
\label{commutation when kappa=0 section}

When $\kappa=0$, we also derive the commutation relations for multiple radial SLE(0) systems which is self-consistent without relying on the results for $\kappa > 0$

\begin{thm} \label{commutation for kappa=0}
Let $\gamma_1,\dotsc,\gamma_{n}$ be simple curves that are generated by Loewner flows and $\mathcal{U}(\boldsymbol{\theta}): \mathfrak{X}^n \rightarrow \mathbb{R}$ is $C^2$ smooth. We define the differential operator 
$\mathcal{M}_j =U_j\partial_j +\sum_{k \neq j} \cot(\frac{\theta_k-\theta_j}{2})\partial_k$. If the curves locally geometrically commute, then the vector fields $\mathcal{M}_j$ satisfy the commutation relations.

\begin{equation} \label{commutation relation for generators kappa=0}
[\mathcal{M}_i,\mathcal{M}_j]= \frac{1}{\sin^2(\frac{\theta_i-\theta_j}{2})}(\mathcal{M}_j-\mathcal{M}_i)
\end{equation}

Moreover, under the additional assumption that $\partial_j U_k= \partial_{k} U_j$ for all $j,k$, then there exists a smooth function $\mathcal{U}(\boldsymbol{\theta})$ such that $U_j =\partial_j \mathcal{U}$.
The commutation relations hold for $\mathcal{L}_j$ if and only if there exists a common constant $h$ such that
\begin{equation} \label{null vector equation for kappa 0}
\frac{1}{2}U^2_j+ \sum_{k\neq j}\cot(\frac{\theta_k-\theta_j}{2})U_k-\sum_{k\neq j}\frac{3}{2\sin^2{(\frac{\theta_j-\theta_k}{2}})}= h
\end{equation} 
\end{thm}

\begin{proof}[Proof of theorem \ref{commutation for kappa=0}]

To study the commutation relations, we focus on growing two hulls from growth points $x,y$ and relabeling other points as marked points $\boldsymbol{z}$. 

The definition of commutation implies that for sufficiently small $s, t>0$ the normalizing map for the hull $\gamma_1[0, t] \cup \gamma_2[0, s]$ is the composition of the Loewner maps for each individual hull $\gamma_1[0, t]$ or $\gamma_2[0, s]$, when applied in either order. In removing $\gamma_1[0, t]$ we are considering the coordinate change $h_{1, t}$, which leads to
\begin{equation} \label{cap change}
\sigma_{1, 2}^{t, s}=\operatorname{hcap}\left(\tilde{\gamma}_2[0, s]\right)=\int_0^s\left(f_{1, 2}^{t, u}\right)^{\prime}\left(y(u)\right)^2 d u 
\end{equation}
where $y(u)$ is the position of $y$ at time $u$.
In this case $f_{1, 2}^{t, s}=\tilde{h}_{2, s} \circ h_{1, t} \circ h_{2, s}^{-1}$. With this notation in hand, commutation implies that
$$
h_{2, \sigma_{1, 2}^{t, s}} \circ h_{1, t}=h_{1, \sigma_{2, 1}^{s, t}} \circ h_{2, s}.
$$
On the left-hand side the driving function first evolves according to the dynamics of $\mathcal{L}_x$ for $t$ units of time and then $\mathcal{L}_y$ for $\sigma_{1, 2}^{t, s}$ units of time. The right-hand side is analogous. These Loewner maps can be the same only if the driving function move to the same position when the maps are applied in either order. In our setup, the motion of the driving functions is fully determined by the motion of the points in $\theta$, so a necessary condition for these maps to be the same is that
\begin{equation} \label{cm flow}
e^{\sigma_{1, 2}^{t, s} \mathcal{M}_y} e^{t \mathcal{M}_x}=e^{\sigma_{2, 1}^{s, t} \mathcal{M}_x} e^{s \mathcal{M}_y}
\end{equation}
where $t \mapsto e^{t \mathcal{M}}$ denotes the flow map corresponding to the dynamics $\mathcal{M}$. The commutation relation (\ref{commutation for kappa=0}), as we now explain, is a straightforward consequence of this identity. From (\ref{cap change}) we obtain
$$
\begin{aligned}
\sigma_{1, 2}^{t, s}=s\left(\left(f_{1, 2}^{t, 0}\right)^{\prime}\left(y\right)+O(s)\right)=s\left(h_{1, t}^{\prime}\left(x\right)+O(s)\right) & =s\left(1-\frac{t}{\sin^2(\frac{x-y} {2})}+o(t)+O(s)\right) \\
& =s-\frac{st}{\sin^2(\frac{\theta_{x}-\theta_y} {2})}+o(s t)+O\left(s^2\right)
\end{aligned}
$$
where the constants in the error terms may depend on $x$ and $y$. Now use the above to expand (\ref{cm flow}) in powers of $s$ and $t$, and compare coefficients of $st$, we obtain the desired commutation relations for generators (\ref{commutation relation for generators kappa=0})
\begin{equation} 
[\mathcal{M}_x,\mathcal{M}_y]= \frac{1}{\sin^2(\frac{x-y}{2})}(\mathcal{M}_y-\mathcal{M}_x)
\end{equation}

Expanding the infinitesimal commutation relation:
\begin{equation}\label{expanding commutation kappa =0}
\begin{aligned}
&[\mathcal{M}_x, \mathcal{M}_y]+\frac{1}{\sin^2(\frac{y-x}{2})}(\mathcal{M}_x-\mathcal{M}_y)=
\\ 
& \left[\cot(\frac{y-x}{2})\frac{\partial U_{ x}}{\partial x}+\sum_i \cot(\frac{y-z_i}{2})\frac{\partial U_{i}}{\partial x}-\frac{1}{2}\frac{\partial (U_y)^2}{\partial x}+\frac{U_x}{2\sin^2(\frac{y-x}{2})}-\frac{3\cos(\frac{x-y}{2})}{2\sin^3(\frac{x-y}{2})}\right] \partial_x \\
-&\left[\cot(\frac{x-y}{2})\frac{\partial U_{ y}}{\partial y}+\sum_i \cot(\frac{x-z_i}{2})\frac{\partial U_{i}}{\partial y}-\frac{1}{2}\frac{\partial (U_x)^2}{\partial y}+\frac{U_y}{2\sin^2(\frac{y-x}{2})}-\frac{3\cos(\frac{y-x}{2})}{2\sin^3(\frac{y-x}{2})} \right] \partial_y
\end{aligned}
\end{equation}
The right hand side of (\ref{expanding commutation kappa =0}) equal to $0$ implies the null vector equations
\begin{equation}
\left\{\begin{array}{l}
\frac{1}{2} U_{x}^2 +\sum_i \cot(\frac{z_i-x}{2}) U_i+\cot(\frac{y-x}{2}) U_y+ \left(-\frac{3}{2\sin^2(\frac{y-x}{2})}+h_1(x, z)\right) =0 \\
\frac{1}{2}U_{y}^2 +\sum_i \cot(\frac{z_i-y}{2}) U_i+\cot(\frac{x-y}{2}) U_x+\left(-\frac{3}{2\sin^2(\frac{y-x}{2})}+h_2(y, z)\right)=0
\end{array}\right.
\end{equation}

Note that $\partial_j U_k= \partial_{k} U_j$ do not naturally follow from the commutation relation. This condition is equivalent to
 the existence of a function $\mathcal{U}(\boldsymbol{\theta})$ such that
$$
U_j=\partial_j \mathcal{U}(\boldsymbol{\theta})
$$
 $\mathcal{U}: \mathfrak{Y}^n \rightarrow \mathbb{R}$ is smooth 
in the chamber $$\mathfrak{Y}^n=\left\{(\theta_1,\theta_2,\ldots,\theta_n) \in \mathbb{R}^n \mid \theta_1<\theta_2<\ldots<\theta_n<\theta_1+2\pi\right\}$$

\begin{lemma}\label{two point commutation kappa=0}
Suppose there exists $\mathcal{U}$ such that $U_j = \partial_j \mathcal{U}$,  then for adjacent $x,y$ (no marked points are between $x$,$y$), if the system
\begin{equation}
\left\{\begin{array}{l}
\frac{1}{2} U_{x}^2 +\sum_i \cot(\frac{z_i-x}{2}) U_i+\cot(\frac{y-x}{2}) U_y+ \left(-\frac{3}{2\sin^2(\frac{y-x}{2})}+h_1(x, z)\right) =0 \\
\frac{1}{2}U_{y}^2 +\sum_i \cot(\frac{z_i-y}{2}) U_i+\cot(\frac{x-y}{2}) U_x+\left(-\frac{3}{2\sin^2(\frac{y-x}{2})}+h_2(y, z)\right)=0
\end{array}\right.
\end{equation}
admits a non-vanishing solution, then:

The functions $h_1, h_2$ can be written as $h_1(x, z)=h(x, z), h_2(y, z)=h(y, z)$

\end{lemma}

Define two operators $\mathcal{L}_1$,$\mathcal{L}_2$ by:
$$
\mathcal{L}_1=\frac{U_x}{2} \partial_{x} +\sum_i \cot(\frac{z_i-x}{2})\partial_i +\cot(\frac{y-x}{2})\partial_y- \frac{3}{2\sin^2(\frac{y-x}{2})}
$$
$$
\mathcal{L}_2= \frac{U_y}{2}\partial_{y} +\sum_i \cot(\frac{z_i-y}{2}) \partial_i +\cot(\frac{x-y}{2}) \partial_x -\frac{3}{2\sin^2(\frac{x-y}{2})}
$$
$\mathcal{U}(\boldsymbol{\theta})$ is annihilated by all operators in the left ideal generated by $(\mathcal{L}_1+h_1), (\mathcal{L}_2+h_2)$, including in particular their commutator:

$$
\begin{aligned}
\mathcal{L} & =\left[\mathcal{L}_1+h_1, \mathcal{L}_2+h_2\right]+\frac{1}{\sin^2(\frac{x-y}{2})} ((\mathcal{L}_1+h_1)-(\mathcal{L}_2+h_2)) \\
&=\left(\left(\cot(\frac{y-x}{2}) \partial_y+\sum_i \cot(\frac{z_i-x}{2}) \partial_i\right) h_2(y,z)-\left(\cot(\frac{x-y}{2}) \partial_x+\sum_i \cot(\frac{z_1-y}{2}) \partial_i\right) h_1(x,z)\right)\\
&+\frac{4\left(h_1-h_2\right)}{\sin^2(\frac{x-y}{2})}
\end{aligned}
$$

The operator $\mathcal{L}$ is an operator of order 0, thus a function that must vanish identically. 

If $x,y$ are adjacent (no marked points are between $x$,$y$), consider the pole of $\mathcal{L}$ at $x=y$.  The second-order pole must  vanish, this implies $h_1(x, z)=h(x, z), h_2(y, z)=h(y, z)$ for some function $h$.

Let us return to the proof of the theorem (\ref{commutation for kappa=0}) for multiple radial SLE(0) systems with $n$ distinct growth points $z_1=e^{i\theta_1},z_2=e^{i\theta_2},\ldots,z_n=e^{i\theta_n}$.

We grow two SLEs from $\theta_i,\theta_j$, $i\neq j$ and treat the rest as marked points denoted by $z$. 

The commutation relation between two SLEs implies that the infinitesimal generator satisfies
\begin{equation}
[\mathcal{M}_i, \mathcal{M}_j]=\frac{1}{\sin^2(\frac{\theta_i-\theta_j}{2})}(\mathcal{M}_j-\mathcal{M}_i)
\end{equation}

and the null vector equations 

\begin{equation} \label{trigonometric integrability kappa=0}
\left\{\begin{array}{l}
\frac{1}{2} U_{i}^2+\sum_k \cot(\frac{z_k-\theta_i}{2})U_k+\cot(\frac{\theta_j-\theta_i}{2})U_j+\left(  \frac{3}{2\sin^2(\frac{\theta_j-\theta_i}{2})}+h_i(\theta_i, \boldsymbol{z})\right)=0 \\
\frac{1}{2} U_{j}^2+\sum_k \cot(\frac{z_k-\theta_j}{2}) U_k+\cot(\frac{\theta_i-\theta_j}{2}) U_i+\left( \frac{3}{2\sin^2(\frac{\theta_i-\theta_j}{2})}+h_j(\theta_j,  \boldsymbol{z})\right)=0
\end{array}\right.
\end{equation}

We may write the first equation in (\ref{trigonometric integrability kappa=0}) as
\begin{equation}
    \frac{1}{2} U_{i}^2+\sum_k \cot(\frac{z_k-\theta_i}{2})U_k+\cot(\frac{\theta_j-\theta_i}{2})U_j=- \frac{3}{2\sin^2(\frac{\theta_j-\theta_i}{2})}-h_i(\theta_i,  z)
\end{equation}
By the integrability condition theorem (\ref{two point commutation kappa=0}), $h_i(\theta_i,z)$ does not depend on $\theta_j$,

Since integrability conditions hold for all $ j \neq i$, by subtracting all $ \frac{3}{2\sin^2(\frac{\theta_j-\theta_i}{2})}$ term, we obtain that 

\begin{equation}
    \frac{1}{2} U_{i}^2+\sum_j \cot(\frac{\theta_j-\theta_i}{2})U_j=-\sum_j \frac{3}{2\sin^2(\frac{\theta_j-\theta_i}{2})}-h_i(\theta_i)
\end{equation}

where $h_i=h_i(\theta_i)$ only depends on $\theta_i$.

Moreover, by the integrability condition,  $h_i=h_{i+1}$ for every pair of $1\leq i \leq n-1$, this implies $h_1=h_2=\ldots=h_n= h$.

\begin{equation}
   h(\theta_i)=- \frac{1}{2} U_{i}^2-\sum_j \cot(\frac{\theta_j-\theta_i}{2})U_j-\sum_{j\neq i} \frac{3}{2\sin^2(\frac{\theta_j-\theta_i}{2})}
\end{equation}
Rotation invariance of $U_i$ and $U_j$ implies that $h(\theta_i)$ must also be rotation invariant and thus a constant. This completes the proof of the theorem (\ref{commutation for kappa=0}).

\end{proof}

\section{Multiple radial SLE(0) system as the limit of multiple radial SLE($\kappa$) system}

\subsection{Classical limits and stationary relations} \label{classical limit of multiple radial SLE system}

In this section, we construct multiple radial SLE(0) systems by heuristically taking the classical limits of multiple radial SLE($\kappa$) systems. Our construction of multiple radial SLE(0) systems are self-consistent and does not depend on resolving the classical limit.

A key concept in this construction is the stationary relations that naturally emerge as a result of normalizing the partition functions.
For multiple radial SLE($\kappa$) system, we have shown that the drift term $b_j(\boldsymbol{\theta})$ is given by 
$$b_j(\boldsymbol{\theta}) = \kappa \frac{\partial \log \mathcal{Z}(\boldsymbol{\theta})}{\partial \theta_j}$$
where $\mathcal{Z}(\boldsymbol{\theta})$ is a positive function satisfying the null vector equations. 

As $\kappa \rightarrow 0$, we need to normalize the partition function. We expect that, for some suitably chosen partition functions, the limit $\mathcal{Z}(\boldsymbol{\theta})^{\kappa}$ exists as $\kappa \rightarrow 0$.

Recall that the radial ground solution is given by the multiple contour integral:
\begin{equation}
\mathcal{J}_{\alpha}(\boldsymbol{\theta})=\oint_{\mathcal{C}_1} \ldots \oint_{\mathcal{C}_n} \Phi_\kappa(\boldsymbol{\theta}, \boldsymbol{\zeta}) d \zeta_m \ldots d \zeta_1 .
\end{equation}
where $\Phi_\kappa(\boldsymbol{\theta}, \boldsymbol{\zeta})$ is the multiple radial SLE($\kappa$) master function, and  $\mathcal{C}_1,\mathcal{C}_2,\ldots,\mathcal{C}_m$ are non-intersecting Pochhammer contours. Pure partition functions $\mathcal{Z}(\boldsymbol{\theta})_{\alpha}$ are linear combinations of radial ground solutions $\mathcal{J}_{\alpha}(\boldsymbol{\theta})$. Therefore, heuristically by the steepest decent method,

\begin{equation}
\lim_{\kappa \rightarrow 0}\mathcal{Z}_\alpha(\boldsymbol{\theta})^{\kappa}= \lim_{\kappa \rightarrow 0}\left(\oint_{\mathcal{C}_1} \ldots \oint_{\mathcal{C}_m} \Phi(\boldsymbol{\theta}, \boldsymbol{\zeta})^{\frac{1}{\kappa}}d\boldsymbol{\zeta}\right)^{\kappa}
\end{equation}
where $\Phi(\boldsymbol{\theta},\boldsymbol{\zeta})$ is the multiple radial SLE(0) master function,
\begin{equation} \label{multiple radial SLE(0) master function}
\Phi(\boldsymbol{\theta},\boldsymbol{\zeta}) := 
\prod_{1 \leq j < k \leq n} \sin^2\left( \frac{\theta_j - \theta_k}{2} \right)
\prod_{1 \leq s < t \leq m} \sin^8\left( \frac{\zeta_s - \zeta_t}{2} \right)
\prod_{k=1}^{n} \prod_{l=1}^{m} \sin^{-4}\left( \frac{\theta_k - \zeta_l}{2} \right).
\end{equation}

The contour integral $\oint_{\mathcal{C}_1} \ldots \oint_{\mathcal{C}_m} \Phi(\boldsymbol{\theta}, \boldsymbol{\zeta})^{\frac{1}{\kappa}}d\boldsymbol{\zeta}$ is approximated by the value of $\Phi(\boldsymbol{\theta}, \boldsymbol{\zeta})$ at the stationary phase, i.e., the critical points of $\Phi(\boldsymbol{\theta}, \boldsymbol{\zeta})$.

\begin{conjecture} We conjecture that for pure partition function  As $\kappa \rightarrow 0$, the limit  $\underset{\kappa \rightarrow 0}{\lim}\mathcal{Z}_{\alpha}^{\kappa}(\boldsymbol{\theta}) $ exist and concentrate on critical points of the master function. i.e.

 \begin{equation}
\underset{\kappa \rightarrow 0}{\lim}\mathcal{Z}_{\alpha}(\boldsymbol{\theta})^{\kappa}=  \Phi(\boldsymbol{\theta}, \boldsymbol{\zeta})     
\end{equation}   
 where $\boldsymbol{\zeta}$ is a critical point of the multiple radial SLE(0) master function $\Phi(\boldsymbol{\theta}, \boldsymbol{\zeta})$. 
\end{conjecture}

\begin{defn}[Stationary relations] \label{Stationary relations}
Let $\boldsymbol{z} = \{z_1, z_2, \ldots, z_n\}$ be distinct points on the unit circle, $\boldsymbol{\xi} = \{\xi_1, \xi_2, \ldots, \xi_m\}$ be involution-symmetric and $\eta$ be the spin. 
\begin{itemize}
    \item In the unit disk, 
\begin{equation}
 -\sum_{j=1}^{n}\frac{2}{\xi_k-z_j}  +\sum_{l\neq k} \frac{4}{\xi_k-\xi_l}+\frac{n -2m+2}{\xi_k}=0
\end{equation}

\item  In angular coordinates,  let $z_i = e^{i\theta_i}$ for $i = 1, 2, \ldots, n$ and $\xi_k = e^{i\zeta_k}$ for $k = 1, 2, \ldots, m$. 

\begin{equation}
\sum_{j=1}^{n}\cot(\frac{\zeta_k-\theta_j}{2})=\sum_{l\neq k} 2\cot(\frac{\zeta_k-\zeta_l}{2})
\end{equation} 

\end{itemize}
\end{defn}

\begin{thm}[Conformal Invariance of Stationary Relations] 
    The stationary relations are preserved under M{\"o}bius transformations.
\end{thm}

\begin{proof}
    The M{\"o}bius invariance of the stationary relations follows naturally from the M{\"o}bius invariance of the master functions as Coulomb gas correlation functions.
\end{proof}

\begin{defn}[Multiple radial SLE($0$)  Loewner chain]
Let $\Omega$ be a simply connected domain,  $z_1,z_2,\ldots,z_n \in \partial \Omega$ be growth points, $u \in \Omega$ a marked interior point, and $\boldsymbol{\nu}=\left(\nu_1, \ldots, \nu_n\right)$ be a set of parametrization of capacity, where each $\nu_i:[0, \infty) \rightarrow[0, \infty)$ is assumed to be measurable. 

In the unit disk $\Omega=\mathbb{D}$, $u=0$, we define the multiple radial SLE($\kappa$) 
Loewner chain as a random normalized conformal map $g_t=g_t(z)$, with $g_{0}(z)=z$ and $g_{t}'(0)=e^{-\int_{0}^{t}\sum_{j}\nu_j(s)ds}$ whose evolution is described by the Loewner differential equation:
\begin{equation}
\partial_t g_t(z)=\sum_{j=1}^n \nu_j(t)g_t(z)\frac{z_j(t)+g_t(z)}{z_j(t)-g_t(z)}, \quad g_0(z)=z,
\end{equation}

The flow map $g_t$ is well-defined up until the first time $\tau$ at which $z_j(t)=z_k(t)$ for some $1 \leq j < k \leq n$, while for each $z \in \mathbb{C}$ the process $t \mapsto g_t(z)$ is well-defined up until $\tau_z \wedge \tau$, where $\tau_z$ is the first time at which $g_t(z)=z_j(t)$. Denote $K_t=\left\{z \in \overline{\mathbb{H}}: \tau_z \leq t\right\}$ be the hull associated to this Loewner chain.

In angular coordinate, let $z_j=e^{i\theta_j}$, then
the Loewner equation for $h_t(z)= -i\log (g_t( e^{iz}))$ is given by

\begin{equation}
\partial_t h_t(z)=\sum_{j=1}^n \nu_j(t)\cot(\frac{h_t(z)-\theta_j(t)}{2}), \quad h_0(z)=z,   
\end{equation}

and the driving functions $\theta_j(t), j=1, \ldots, n$, evolve as

\begin{equation}
\dot{\theta}_j= \nu_j(t) \frac{\partial {\rm log} \mathcal{Z}(\boldsymbol{\theta})}{\partial \theta_j}+\sum_{k \neq j} \nu_k(t)\cot(\frac{\theta_j-\theta_k}{2})+ \sqrt{\kappa \nu_j(t)}dB_j(t)
\end{equation}
where $\mathcal{Z}(\boldsymbol{\theta})$ is the partition function for the multiple radial SLE($\kappa$) system.

Moreover, the law of the curves is conformally invariant (up to a time change) and reparametrization symmetric.
See section (\ref{commutation when kappa=0 section}) for a detailed explanation.

Our definition for multiple radial SLE($0$) Loewner chain can be naturally extended to an arbitrary simply-connected domain $\Omega$ with a marked interior point $u$ via a conformal uniformizing map $\phi: \Omega \rightarrow \mathbb{D}$, sending $u$ to $0$.
\end{defn}

\subsection{Stationary relations imply commutation relations in $\kappa=0$ case} \label{Stationary relations imply commutation}
Our definition of multiple radial SLE(0) system views it as a dynamical system of $\boldsymbol{\xi}$ and $\boldsymbol{z}$, and we input stationary relations as constraints for the initial position of $\boldsymbol{\xi}$.
In section \ref{Stationary relations imply commutation}, we show that the stationary relations between \( \boldsymbol{\zeta} \) and \( \boldsymbol{\theta} \) naturally imply the existence of partition functions only depends on $\boldsymbol{\theta}$. Thus, our construction of multiple radial SLE(0) systems involving $n+m$ points, can be reduced to a system only involve the growth points $\boldsymbol{z}$, this fits into the general framework of multiple radial SLE(0) systems.

 \begin{thm} \label{Stationary relation imply partition function}
Let \( z_i = e^{i\theta_i} \) for \( i = 1, 2, \ldots, n \) represent distinct boundary points in angular coordinates, and let \( \xi_k = e^{i\zeta_k} \) for \( k = 1, 2, \ldots, m \). Suppose \( \boldsymbol{\zeta} \) depends smoothly on \( \boldsymbol{\theta} \) as a solution to the stationary relations. Define the partition function
\[
\mathcal{Z}(\boldsymbol{\theta}) := \prod_{1 \leq j < k \leq n} \sin^2\left(\frac{\theta_j - \theta_k}{2}\right) 
\prod_{1 \leq s < t \leq m} \sin^8\left(\frac{\zeta_s(\boldsymbol{\theta}) - \zeta_t(\boldsymbol{\theta})}{2}\right) 
\prod_{k=1}^{n} \prod_{l=1}^m \sin^{-4}\left(\frac{\theta_k - \zeta_l(\boldsymbol{\theta})}{2}\right).
\]
Then \( \mathcal{Z}(\boldsymbol{\theta}) \) is strictly positive and invariant under rotation of each coordinate of \( \boldsymbol{\theta} \).  

Moreover, define \( \mathcal{U}(\boldsymbol{\theta}) = \log \mathcal{Z}(\boldsymbol{\theta}) \), and for \( j = 1, 2, \ldots, n \), let
\[
U_j(\boldsymbol{\theta}) = \partial_j \mathcal{U} = \sum_{k \neq j} \cot\left(\frac{\theta_j - \theta_k}{2}\right) - 2 \sum_{l=1}^m \cot\left(\frac{\theta_j - \zeta_l(\boldsymbol{\theta})}{2}\right).
\]
Then \( U_j(\boldsymbol{\theta}) \) is real-valued and satisfies the following system of null vector equations:
\[
\frac{1}{2}U_j^2 + \sum_{k \neq j} \cot\left(\frac{\theta_k - \theta_j}{2}\right) U_k 
- \sum_{k \neq j} \frac{3}{2 \sin^2\left(\frac{\theta_j - \theta_k}{2}\right)} = -\frac{(2m - n)^2}{2} + \frac{1}{2}.
\]
\end{thm}

\begin{thm}[Ward's Identities]
The functions \( U_j \) satisfy the following identity:
\[
\sum_{j=1}^n U_j = 0.
\]
\end{thm}
\begin{proof}[Proof of theorem (\ref{Stationary relation imply partition function})]

Strict positivity of $\mathcal{Z}(\boldsymbol{\theta})$ follows from the $\theta_j$ assumed to be real and distinct and the $\zeta_k$ always appearing in complex conjugate pairs. For the differentiability of $\mathcal{Z}(\boldsymbol{\theta})$, we observe that the smoothness of the pole functions follows from the fact that each $\zeta$ obeys the stationary relations. The stationary relations, together with the implicit function theorem, imply that there is a neighborhood of $\theta$ on which the poles $\zeta_{k}=\zeta_{k}(\boldsymbol{\theta})$ are smooth functions of the critical points. We use the existence and continuity of the first and second-order partial derivatives to compute $\partial_j \log \mathcal{Z}(\boldsymbol{\theta})$. Indeed, using
expression for $\mathcal{Z}(\boldsymbol{\theta})$ and computing $\partial_j \log \mathcal{Z}$ directly we have
$$
\begin{aligned}
\partial_j \log \mathcal{Z}(\boldsymbol{\theta})= & \sum_{k\neq j} \cot\frac{\theta_j-\theta_k}{2}+\sum_{i=1}^m 2\cot(\frac{\zeta_i-\theta_j}{2})+2 \sum_{k=1}^{n} \sum_{l=1}^m \cot(\frac{\theta_k-\zeta_l}{2})\partial_{\theta_j} \zeta_l \\
+& 4 \sum_{1 \leq l < s \leq m} \cot(\frac{\zeta_l-\zeta_s}{2})(\partial_{\theta_j} \zeta_l-\partial_{\theta_j} \zeta_s) 
\end{aligned}
$$
For the last two terms use the stationary relation to obtain
$$
\begin{aligned}
\sum_{l=1}^m \partial_{\theta_j} \zeta_l \sum_{k=1}^{n} \cot(\frac{\theta_k-\zeta_l}{2})= & -2\sum_{l=1}^n \partial_{\theta_j} \zeta_l \sum_{s \neq l} \cot(\frac{\zeta_l-\zeta_s}{2}) \\
=& -2\sum_{1 \leq l<s \leq m} \cot(\frac{\zeta_l-\zeta_s}{2})(\partial_{\theta_j} \zeta_l-\partial_{\theta_j} \zeta_s)
\end{aligned}
$$
Thus the last two terms in the above expression vanish, thereby proving $\partial_j \log \mathcal{Z}=U_j$.

Real-valuedness of $U_j$ follows from $\theta_k \in \mathbb{R}$ and the $\zeta_k$ occurring in complex conjugate pairs. 

For the translation invariance, note that the pole functions $\zeta_{k}$ clearly satisfy $\zeta_{k}(\boldsymbol{\theta}+h)=\zeta_{ k}(\boldsymbol{\theta})+h$ for $h \in \mathbb{R}$.
 
For the proof that the $U_j$ satisfy the mull vector equations it is convenient to introduce
$$
u_j:=U_j-\sum_{k \neq j} \cot(\frac{\theta_j-\theta_k}{2}) .
$$
$$
\begin{aligned}
&\frac{1}{2}U_{j}^2+\sum_{k\neq j}\cot(\frac{\theta_k-\theta_j}{2})U_k-\sum_{k\neq j} \frac{3}{2\sin^2(\frac{\theta_j-\theta_k}{2})}\\
& = \frac{1}{2}u_{j}^2+ \sum_{k\neq j}\cot(\frac{\theta_j-\theta_k}{2})
(u_j-u_k)+ \frac{1}{2}\sum_{k\neq j\neq l} \cot(\frac{\theta_j-\theta_k}{2})\cot(\frac{\theta_j-\theta_l}{2})+ \\
&\sum_{k\neq j\neq l} \cot(\frac{\theta_k-\theta_j}{2})\cot(\frac{\theta_k-\theta_l}{2}) 
+\sum_{k\neq j} \frac{3}{2}\cot^2(\frac{\theta_j-\theta_k}{2}) -\sum_{k\neq j} \frac{3}{2\sin^2(\frac{\theta_j-\theta_k}{2})} \\
&= \frac{1}{2} u_j^2+ \sum_{k\neq j}\cot(\frac{\theta_j-\theta_k}{2})(u_j-u_k)
-C_{n-1}^{2}-\frac{3}{2}(n-1)
\end{aligned}
$$
All we need is to compute
$$
\frac{1}{2} u_j^2+ \sum_{k\neq j}\cot(\frac{\theta_j-\theta_k}{2})(u_j-u_k).
$$
Using the stationary relation, we have
$$
\begin{aligned}
&\frac{1}{2}\left(\frac{1}{2} u_j^2+ \sum_{k \neq j} (u_j-u_k)\cot(\frac{\theta_j-\theta_k}{2})\right) \\
&=\left(\sum_{l=1}^m \cot(\frac{\zeta_l-\theta_j}{2})\right)^2+\sum_{k \neq j} \sum_{l=1}^m \cot\left(\frac{\zeta_l-\theta_j}{2}\right)\cot\left(\frac{\zeta_l-\theta_k}{2}\right) \\
& =\sum_{l=1}^m \cot(\frac{\zeta_l-\theta_j}{2}) \sum_{k=1}^{n} \cot(\frac{\zeta_l-\theta_k}{2})+2 \sum_{k<l} \cot(\frac{\zeta_k-\theta_j}{2}) \cot(\frac{\zeta_l-\theta_j}{2})+m(n-1) \\
& =2\sum_{l=1}^m \cot(\frac{\zeta_l-\theta_j}{2}) \sum_{k\neq l} \cot(\frac{\zeta_l-\zeta_k}{2})+2 \sum_{k<l} \cot(\frac{\zeta_k-\theta_j}{2}) \cot(\frac{\zeta_l-\theta_j}{2})+m(n-1)\\
& =-2C_{m}^{2}+m(n-1) 
\end{aligned}
$$
In above computation, we repeatedly applied the trigonometric identity:
$$
\cot(A)\cot(A+B)+\cot(B)\cot(A+B)-\cot(A)\cot(B)=-1
$$
Thus we obtain that
\begin{equation}
\label{m screening null vector constant}
\begin{aligned}
& \frac{1}{2}U_{j}^2+\sum_{k\neq j}\cot(\frac{\theta_k-\theta_j}{2})U_k-\sum_{k\neq j} \frac{3}{2\sin^2(\frac{\theta_j-\theta_k}{2})} \\
& = -4 C_{m}^{2}+2m(n-1)-C_{n-1}^{2}-\frac{3}{2}(n-1) = -\frac{(2m-n)^2}{2}+\frac{1}{2}
\end{aligned}
\end{equation}

\end{proof}

\begin{thm} The functions $U_{j}(\boldsymbol{\theta})$ satisfy 
$$
\sum_{j=1}^{ n} U_{j}=0
$$
\end{thm}
\begin{proof}
By direct computation,
$$
\sum_{j=1}^{n} U_j=\sum_{j=1}^{n} \sum_{k \neq j} \cot(\frac{\theta_j-\theta_k}{2})+2\sum_{j=1}^{n} \sum_{k=1}^m \cot(\frac{\zeta_k-\theta_j}{2})=4\sum_{k=1}^m \sum_{l \neq k} \cot(\frac{\zeta_k-\zeta_l}{2})=0
$$
\end{proof}

\subsection{Residue-free quadratic differentials with prescribed  zeros} \label{residue free quadratic differential}
\
\indent
The locus of real rational functions effectively characterizes the traces of multiple chordal SLE(0) systems. However, in the radial case, rational functions alone are not enough to fully characterizze these traces.

To address this limitation, we propose introducing a new class of analytic functions that extends beyond rational functions. This extended class would capture the behavior near zero, including monodromy that can not described by rational functions.

We first introduce an equivalence class of residue-free quadratic differentials (Definition \ref{trace quadratic differential}), denoted by $\mathcal{QD}(\boldsymbol{z})$. 
For such quadratic differential $Q(z)dz^2 \in \mathcal{QD}(\boldsymbol{z})$, there exists a unique primitive of $\sqrt{Q(z)}$ (up to real additive constant), denoted by $F(z)$, is involution symmetric and locally a meromorphic function with monodromy around $0$. This is precisely the class of functions we need. The advantage of introducing quadratic differential is that the expression of such quadratic differentials can be explicitly formulated in terms of growth point $\boldsymbol{z}=\{z_1,z_2,\ldots,z_n\}$ and poles $\boldsymbol{\xi}=\{\xi_1,\xi_2,\ldots,\xi_m\}$.

\begin{thm}[Stationary Relations and Residue-Free Condition] \label{Stationary- residue free}

The following statements are equivalent:
\begin{enumerate}
    \item The points \( \boldsymbol{\xi} \) are symmetric under the involution \( z^* = \frac{1}{\bar{z}} \), and the zeros \( \boldsymbol{z} \) on the unit circle satisfy the stationary relations.
    \item There exists a quadratic differential \( Q(z)dz^2 \in \mathcal{QD}(\boldsymbol{z}) \) with zeros at \( \boldsymbol{z} \) and poles at \( \boldsymbol{\xi} \).
\end{enumerate}
\end{thm}

\begin{proof}[Proof of theorem (\ref{Stationary- residue free})] Theorem  (\ref{Stationary- residue free}) is equivalent to (iii) in the following lemma
\begin{lemma} \label{partial fraction of H}

 Fix $n \geq 1$ and let $\boldsymbol{z}=\{z_1, \ldots, z_n\}$ be distinct points on the unit circle. The following statements are true.
 
\begin{itemize}
\item[(i)]
 Up to a real multiplicative constant, for any $Q(z)dz^2 \in \mathcal{QD}(\boldsymbol{z})$, $Q(z)$ factorizes as
$$
Q(z)= 
\frac{\prod_{k=1}^{m}\xi_{k}^2}{ \prod_{j=1}^{n}z_j}
z^{2m-n-2}\frac{\prod_{j=1}^{ n}\left(z-z_j\right)^2}{\prod_{k=1}^{m}\left(z-\xi_k\right)^4},
$$
where $\left\{\xi_1, \ldots, \xi_{m}\right\}$ are involution symmetric poles of $Q(z)$ and $\xi_k \neq 0,\infty$, $k=1,2,\ldots,m$.

\item[(ii)] Consider $\sqrt{Q(z)}$

\begin{equation}
\sqrt{Q(z)}=\frac{\prod_{k=1}^{m}\xi_{k}}{ \sqrt{\prod_{j=1}^{n}z_j}} z^{m-\frac{n}{2}-1}\frac{\prod_{j=1}^{n}\left(z-z_j\right)}{\prod_{l=1}^m\left(z-\xi_l\right)^2}
\end{equation}

If $\xi_1, \ldots, \xi_m \in \mathbb{C}$ are all distinct then $\sqrt{Q(z)}$ has a Laurant expansion at $\xi_k$
$$
\sqrt{Q(z)}=g(z-\xi_k)+\frac{A_k}{\left(z-\xi_k\right)^2}+\frac{B_k}{z-\xi_k},
$$
where $g(z-\xi_k)$ is a holomorphic function in the neighbourhood of $\xi_k$, and the constants $A_k$ and $B_k$ are given by:

\begin{equation} 
A_k=\xi_{k}^{m-\frac{n}{2}-1}\frac{\prod_{j=1}^{ n}\left(\xi_k-z_j\right)}{\prod_{l \neq k}\left(\xi_k-\xi_l\right)^2},
\end{equation}

\begin{equation} 
B_k=\left(\sum_{j=1}^{n} \frac{1}{\xi_k-z_j}-2\sum_{l \neq k} \frac{1}{\xi_k-\xi_l}- \frac{\frac{n}{2}-m+1}{\xi_k}\right) A_k, \quad k=1, \ldots, m .
\end{equation}

\item[(iii)] If $\xi_1, \ldots, \xi_m \in \mathbb{C}$ are all distinct,  $Res_{\xi_k}(\sqrt{Q(z)}dz)=0$  if and only if all $B_k=0$ which are equivalent to the \textbf{stationary relations}.

\end{itemize}
\end{lemma}

\begin{proof}
\begin{itemize}

\item[(i)]

By the fact that the double zeros of $Q(z)dz^2$ are precisely $\boldsymbol{z}$, poles of order 4 at $\boldsymbol{\xi}$, poles of order $2m-n-2$ at $0$ and $\infty$ we have:
$$Q(z)= \lambda z^{b}\frac{\prod_{j=1}^{ n}\left(z-z_j\right)}{\prod_{k=1}^{m}\left(z-\xi_k\right)^2} $$

Note that $Q(z)$ is involution symmetric, $\overline{Q(z^*)}\overline{(dz^{*})^2}= Q(z)dz^2$, thus 
$$Q(z)=\overline{\lambda}\frac{1}{z^4}\overline{Q(\frac{1}{\overline{z}})}=\overline{\lambda}z^{4m-2n-b-4}\frac{\prod_{j=1}^{ n}\left(1-\overline{z_j}z\right)^2}{\prod_{k=1}^{m}\left(1-\overline{\xi_k} z\right)^4}$$

By comparing the exponents of $z$, we obtain that $b = 2m-n-2$.

By comparing the constant, we obtain that $\lambda$ is the real constant times of $(-1)^{2m-n-1}
\frac{\prod_{k=1}^{m}\xi_k^2}{ \prod_{j=1}^{n}z_j}$

\item[(ii)]
$\xi_k$ is a second order pole of $\sqrt{Q(z)}$, thus by the laurant expansion,
$$
\sqrt{Q(z)}=g(z-\xi_k)+\frac{A_k}{\left(z-\xi_k\right)^2}+\frac{B_k}{z-\xi_k},
$$
where $g(z-\xi_k)$ is a holomorphic function in the neighbourhood of $\xi_k$.
\begin{equation}
\sqrt{Q(z)}(z-\xi_k)^2= g(z-\xi_k)(z-\xi_k)^2+A_k+ B_k(z-\xi_k)
\end{equation}

By differentiating both sides, we obtain that
\begin{equation}
\left(\sqrt{Q(z)}(z-\xi_k)^2\right)'= g'(z-\xi_k)(z-\xi_k)^2+2g(z-\xi_k)(z-\xi_k)+B_k
\end{equation}

Let $z= \xi_k$, we obtain that

\begin{equation} 
\begin{aligned}
A_k &=(\sqrt{Q(z)}(z-\xi_k)^2\rvert_{z=\xi_k} \\
&=\xi_{k}^{m-\frac{n}{2}-1}\frac{\prod_{j=1}^{ n}\left(\xi_k-z_j\right)}{\prod_{l \neq k}\left(\xi_k-\xi_l\right)^2}, \quad k=1, \ldots, m .
\end{aligned}
\end{equation}

\begin{equation}\label{value of Bk}
\begin{aligned}
B_k &=(\sqrt{Q(z)}(z-\xi_k)^2)'\rvert_{z=\xi_k} \\
&=\left(\sum_{j=1}^{n} \frac{1}{\xi_k-z_j}-2\sum_{l \neq k} \frac{1}{\xi_k-\xi_l}- \frac{\frac{n}{2}-m+1}{\xi_k}\right) A_k, \quad k=1, \ldots, m 
\end{aligned}
\end{equation}
\item[(iii)] Note that  $Res_{\xi_k}(\sqrt{Q(z)}dz)=0$ is equivalent to $B_k=0$, and by equation (\ref{value of Bk}), equivalent to the stationary relations .

\end{itemize}

\end{proof}

\end{proof}

To better understand the horizontal trajectories of this equivalence class of quadratic differentials, we introduce
a class of multi-valued analytic functions $F(z)$ with monodromy at $0$.

$F(z)$ is exactly the primitive of the differential $\sqrt{Q(z)}dz$. However, this primitive only exists locally, and there is monodromy at $z=0$.

\begin{thm} \label{quadratic differential analytic function}
Let $\boldsymbol{z}=\{z_1,z_2,\ldots,z_n \}$ be distinct points on the unit circle,  for a quadratic differential $Q(z)\in \mathcal{QD}(\boldsymbol{z})$ :
\begin{itemize}
    \item 
    when $n$ is an even integer, there exists a unique (up to a real additive constant)
    $$F(z) = R(z)+ i clog(z)$$ 
    whose finite critical points are precisely $\boldsymbol{z}=\{z_1,z_2,\ldots,z_n\}$
    where $R(z)$ is a rational function, such that $F(z)$ is  involution symmetric ($F(z^*)=\overline{F(z)}$ where $z^*=\frac{1}{\bar{z}}$) and 
    
    $Q(z)dz^2=(F^{\prime}(z))^2dz^2$
    \item 
    when $n$ is an odd integer, there exists a unique
    $$F(z)=\sqrt{z}R(z)$$
    whose finite critical points are precisely $\boldsymbol{z}=\{z_1,z_2,\ldots,z_n\}$
    where $R(z)$ is rational function, such that $F(z)$ is involution symmetric ($F(z^*)=\overline{F(z)}$ where $z^*=\frac{1}{\bar{z}}$) and $Q(z)dz^2=(F^{\prime}(z))^2dz^2$.
    
\end{itemize}  
  Note that the involution symmetry of $F(z)=R(z)+iclogz$ implies $c \in \mathbb{R}$. 
 Although $F(z)$ is multivalued, branches only differ by integer multiples of $2\pi c$,  $F'(z)$ are the same in every branch. Consequently, the critical points where $F'(z)=0$ are well-defined.
  
  For $F(z)=\sqrt{z}R(z)$, branches only differ by signs, and $F'(z)$ also differ by signs, ensuring that the critical points where $F'(z)=0$ are well-defined.

  By considering $F(z^2)$ and taking the double cover of $\mathbb{D}\backslash 0$, the odd $n$ case can be reduced to the even $n$ case.
\end{thm}

\begin{proof}
Given a quadratic differential $Q(z)dz^2 \in \mathcal{QD}(\boldsymbol{z})$:
\begin{itemize}
    \item For $n$ even and $2m \leq n$ note that 
\begin{equation}
\sqrt{Q(z)}=\frac{\prod_{k=1}^{m}\xi_{k}}{ \sqrt{\prod_{j=1}^{n}z_j}}z^{m-\frac{n}{2}-1}\frac{\prod_{j=1}^{n}\left(z-z_j\right)}{\prod_{l=1}^m\left(z-\xi_l\right)^2}
\end{equation}
is a rational function. By lemma (\ref{partial fraction of H}),  since $\{\xi_1,\xi_2,\ldots,\xi_m\}$ solve the \textbf{stationary relations},
$$Res_{\xi_j}(\sqrt{Q(z)}dz)=0$$
which implies that $F(z)$ as the primitive of $\sqrt{Q(z)}$ can only have logarithmic singularity at $0$, thus $F(z)=R(z)+ci\log(z)$.
The involution symmetry implies the constant $c \in \mathbb{R}$.
\item For $n$ even and $2m > n$,
\begin{equation}
\sqrt{Q(z)}=\frac{\prod_{k=1}^{m}\xi_{k}}{ \sqrt{\prod_{j=1}^{n}z_j}}z^{m-\frac{n}{2}-1}\frac{\prod_{j=1}^{n}\left(z-z_j\right)}{\prod_{l=1}^m\left(z-\xi_l\right)^2}
\end{equation}
is a rational function. By lemma (\ref{partial fraction of H}),  since $\{\xi_1,\xi_2,\ldots,\xi_m\}$ solve the \textbf{stationary relations},
$$Res_{\xi_j}(\sqrt{Q(z)}dz)=0$$
and since $2m>n$, $z=0$ is a zero of $\sqrt{Q(z)}$ not a pole. $\sqrt{Q(z)}$ has vanishing residues at all its poles $\{\xi_1,\xi_2,\ldots,\xi_m\}$. Therefore, $F(z)$ as the primitive of $\sqrt{Q(z)}$ is a rational function.

\item For $n$ odd, by change of variable, $z = u^2$, we denote:
$$z^{m-\frac{n}{2}-1}\frac{\prod_{j=1}^{n}\left(z-z_j\right)}{\prod_{l=1}^m\left(z-\xi_l\right)^2}dz=2 u^{2m-n-1}\frac{\prod_{j=1}^{n}\left(u^2-z_j\right)}{\prod_{l=1}^m\left(u^2-\xi_l\right)^2}du= S(u^2)du$$

The residue-free condition implies that
$$Res_{\pm \sqrt{\xi_j}}(S(u^2))=Res_{\xi_j}(\sqrt{Q(z)})=0$$
Since $S(u^2)=S((-u)^2)$ is an even rational function,  $$Res_{0}(S(u^2))=0$$

Therefore, the primitive of $S(u^2)$ can be chosen as an odd rational function $uR(u^2)$. Consquently, $F(z)=\sqrt{z}R(z)$. 
\end{itemize}

\end{proof}

\begin{lemma}\label{real locus well defined}
    The real locus $\Gamma(F(z)):=\{ z\in \mathbb{C}| F(z)\in \widehat{\mathbb{R}}\}$ is well-defined.
\end{lemma}

\begin{proof} 
\begin{itemize}
\item If $n$ is even, $F(z)=ic\log(z)+R(z)$, where $c \in \mathbb{R}$.
Since $F(z)$ is a multivalued function, we consider the $\tilde{F}(z)$, the lift of $F(z)$ to the universal cover of $\mathbb{C} \backslash \{0\}$. $\tilde{F}(z)=F(e^{iz})=-cz+R(e^{iz})$, the projection map $\rho(z):\widehat{\mathbb{C}} \rightarrow \mathbb{C} \backslash 0$ is given by $\rho(z)=e^{iz}$.

$\tilde{F}(z+2\pi)= \tilde{F}(z)- 2\pi c$, 
$2\pi c \in \mathbb{R}$, thus the real locus $\Gamma(\tilde{F})$ has translation period of $2\pi$. Since the projection map $\rho(z)=e^{iz}$ also has a translation period of $2\pi$.
Therefore, $\Gamma(F)$ as the image of $\Gamma(\tilde{F})$ under projection $\rho(z)$ is well defined.

\item If $n$ is odd, $F(z)=\sqrt{z}R(z)$.
Similarly when $n$ odd, $\tilde{F}(z+\pi)=-\tilde{F}(z)$, then $\tilde{F}(z+2\pi)=\tilde{F}(z)$, thus the real locus $\Gamma(\tilde{F})$ has translation period of $2\pi$.
Therefore, $\Gamma(F)$ as the image of $\Gamma(\tilde{F})$ under $\rho(z)$ is well defined.
\end{itemize}
\end{proof}

\begin{lemma}\label{horizontal trajectories level lines}
Given a quadratic differential $Q(z) \in \mathcal{QD}(\boldsymbol{z})$ associate to it a vector field $v_Q$ on $\mathbb{C}$ defined by
$$
v_Q(z)=\frac{1}{F^{\prime}(z)}=\frac{1}{\sqrt{Q(z)}}
$$

  Then  the flow lines of $\dot{z} = \frac{1}{\sqrt{Q(z)}}  $ are level lines of ${\rm Im}(F(z))$ and the horizontal trajectories of $Q(z)dz^2$.
\end{lemma}
\begin{proof}
By Cauchy-Riemann equation,
$$v_{Q}(z)=\frac{1}{F'(z)} \perp \nabla {\rm Im}(F(z))$$
implies flow lines of $v_Q(z)$ are level lines of ${\rm Im}(F)$.

$$Q(z)v_{Q}(z)^2 =1 >0$$
implies flow lines of $v_Q(z)$ are horizontal trajectories of $Q(z)dz^2$.
\end{proof}

Next, we characterize the geometry of the horizontal trajectories of $Q(z)dz^2 \in \mathcal{QD}(\boldsymbol{z})$.

\begin{thm} \label{horizontal trajectories form link pattern}
Let \( \boldsymbol{z} = \{z_1, z_2, \ldots, z_n \} \) be distinct points on the unit circle. Consider a quadratic differential \( Q(z) \in \mathcal{QD}(\boldsymbol{z}) \) defined by:
\[
Q(z) = \frac{\prod_{k=1}^{m} \xi_{k}^2}{\prod_{j=1}^{n} z_j} z^{2m-n-2} \frac{\prod_{j=1}^{n} (z - z_j)^2}{\prod_{k=1}^{m} (z - \xi_k)^4},
\]
where \( \{\xi_1, \ldots, \xi_m\} \) are involution-symmetric finite poles of \( Q(z) \), and \( \xi_k \neq 0, \infty \) for \( k = 1, 2, \ldots, m \).

The horizontal trajectories of \( Q(z) \), denoted as \( \Gamma(Q) \), are the trajectories of \( Q(z)dz^2 \in \mathcal{QD}(\boldsymbol{z}) \) that with limiting ends at the zeros \( \{z_1, z_2, \ldots, z_n \} \). These trajectories satisfy the following properties:
\begin{itemize}
    \item[\rm(1)] If \( 2m \leq n \), the trajectories \( \Gamma(Q) \) form a topological radial link pattern in \( \text{LP}(n, m) \).
    \item[\rm(2)] If \( 2m > n \), the trajectories \( \Gamma(Q) \) form a topological radial link pattern in \( \text{LP}(n, n-m) \).
\end{itemize}
\end{thm}

To prepare for this, we first introduce some fundamental concepts from the theory of quadratic differentials.

\begin{defn} \label{C and H}
For a quadratic differential $Q(z)dz^2$ on a Riemann surface $S$, we denote the zeros and simple poles of $Q(z)$ by set $C$ and poles of order at least 2  by set $H$.
\end{defn}
\begin{defn}[F-set]
 A set $K$ on a Riemann surface $S$ is called an $F$-set (with respect to $\left.Q(z) d z^2\right)$ if any trajectory of $Q(z) d z^2$ which meets $K$ lies entirely in $K$.   
\end{defn} 

\begin{defn}[Inner closure]
  By the inner closure of a set on $\Re$ we mean the interior of the closure of the set. The inner closure of a set $K$ will be denoted by $\hat{K}$.  
\end{defn} 

In the following four definitions, we understand in each case $S$ to be a finite oriented Riemann surface, $Q(z) d z^2$ to be a quadratic differential on $S$.

\begin{defn}[End domain]
An end domain U (relative to $Q(z) d z^2$)) is a maximal connected open $F$-set on $S$ with the properties:
\begin{itemize}
\item [(i)] U contains no critical point of $Q(z) d z^2$,
\item[(ii)] U is swept out by trajectories of $Q(z) d z^2$ each of which has a limiting end point in each of its possible senses at a given point $A$ in $H$,
\item[(iii)] U is mapped by $F(z)=\int(Q(z))^{1 / 2} d z$ conformally onto an upper or lower half-plane.

\end{itemize}
\end{defn} 

\begin{defn}[Strip domain]
 A strip domain $U$ (relative to $\Q(z) d z^2$) is a maximal connected open $F$-set on $S$ with the properties:
 \begin{itemize}
     \item [(i)] U contains no critical point of $Q(z) d z^2$,
\item[(ii)] U is swept out by trajectories of $Q(z) d z^2$ each of which has at one point $A$ in $H$ in the one sense a limiting end point and at another (possibly coincident) point $B$ in $H$ in the other sense a limiting end point,
\item[(iii)] U is mapped by $F(z)=\int(Q(z))^{1 / 2} d z$ conformally onto a strip $a< {\rm Im} F<b$, a, b are finite real numbers, $a<b$.
\end{itemize}
   
\end{defn}
\begin{defn}[Circle domain]
  A circle domain $U$ (relative to $\left.Q(z) d z^2\right)$ is a maximal connected open $F$-set on $S$ with the properties
\begin{itemize}
    \item[(i)] U contains a single double pole $A$ of $Q(z) d z^2$,
\item[(ii)] $U-A$ is swept out by trajectories of $Q(z) d z^2$ each of which is a Jordan curve separating $A$ from the boundary of $\mathcal{S}$,
\item[(iii)] for a suitably chosen purely imaginary constant c the function

$$
w=\exp \left\{c \int(Q(z))^{1 / 2} d z\right\}
$$
\end{itemize}
extended to have the value zero at $A$ maps U conformally onto a circle $|w|<R$, $A$ going into the point $w=0$.  
\end{defn}

\begin{defn}[Ring domain]
 A ring domain $U$ (relative to $\left.Q(z) d z^2\right)$ is a maximal connected open $F$-set on $S$ with the properties:

 \begin{itemize}
\item [(i)] $U$ contains no critical point of $Q(z) d z^2$,
\item[(ii)] $U$ is swept out by trajectories of $Q(z) d z^2$ each of which is a Jordan curve,
\item[(iii)] for a suitably chosen purely imaginary constant $c$ the function

$$
w=\exp \left\{c \int(Q(z))^{1 / 2} d z\right\}
$$

maps $U$ conformally onto a circular ring

$$
r_1<|w|<r_2\left(0<r_1<r_2\right)
$$
 \end{itemize}

\end{defn}
In \cite{J12} thm 3.5, the author proves a general result for positive quadratic differentials on finite Riemann surface $S$. In our setting, we only need to consider a special case $S=\widehat{\mathbb{C}}$ where all quadratic differentials are positive.
\begin{thm}[Basic Structure Theorem, thm 3.5 in \cite{J12}] \label{basic structure theorem}
 Let $S$ be a Riemann sphere and $Q(z) d z^2$ a  quadratic differential on $S$ where we exclude the following possibilities and all configurations obtained from them by conformal equivalence:
 \begin{itemize}
     \item[(i)]  S the $z$-sphere, $Q(z) d z^2=d z^2$,
     \item[(ii)] S the $z$-sphere, $Q(z) d z^2=K e^{i x} d z^2 / z^2$, $\alpha$ real, $K$ positive,
 \end{itemize}
Let $\Gamma(Q)$ denote the union of all trajectories which have a limiting end point at a point of $C$ (see definition (\ref{C and H}). Then
\begin{itemize}
    \item [(i)] $S-\Gamma(Q)$ consists of a finite number of end, strip, circle and ring domains,
    \item[(ii)] each such domain is bounded by a finite number of trajectories together with the points at which the latter meet; every boundary component of such a domain contains a point of $C$; for a strip domain the two boundary elements arising from points of $H$ (see definition (\ref{C and H})) divide the boundary into two parts on each of which is a point of $C$,
\item[(iii)] every pole of $Q(z) d z^2$ of order $m$ greater than two has a neighborhood covered by the inner closure of $m-2$ end domains and a finite number (possibly zero) of strip domains,
\item[(iv)] every pole of $Q(z) d z^2$ of order two has a neighborhood covered by the inner closure of a finite number of strip domains or has a neighborhood contained in a circle domain,

\item[(v)] the inner closure $\widehat{\Gamma(Q)}$ of $\Gamma(Q)$ is an $F$-set consisting of a finite number of domains on $S$ each with a finite number (possibly zero) of boundary components,
\item[(vi)] each boundary component of such a domain is a piecewise analytic curve composed of trajectories and their limiting end points in $C$.
\end{itemize}
\end{thm}

\begin{proof}[Proof of theorem (\ref{horizontal trajectories form link pattern})]
We characterize the geometry of $\Gamma(Q)$ by discussing the following cases:
\begin{itemize}
      \item If $n$ is even and $2m<n$, then the poles of $Q(z)$ at $0$ and $\infty$ are of order $n+2-2m \geq 4$

    By the basic structure theorem (\ref{basic structure theorem}), $\widehat{\mathbb{C}} - \Gamma(Q)$, consists of a finite number of end, strip, circle and ring domains.

    Now, we prove that there can not be strip or ring domains.
   For $Q(z)dz^2 \in \mathcal{QD}(z)$, by lemma (\ref{quadratic differential analytic function}) and lemma (\ref{horizontal trajectories level lines}),  $F(z)=\int \sqrt{Q(z)}dz$ take real values on $\Gamma(Q)$.  

   Therefore, $F(z)$ can not map $U$ to a strip domain, in this case on $\partial U$, ${\rm Im}F$ take 2 differential values. Similarly, $F(z)$ can not map $U$ to a ring domian, in this case on $\partial U$, ${\rm Im}F$ also take 2 differential values.

    When $n>2m$, the degree of pole at $z=0$ is a least $4$, so there also can not be any circle domains.

    In conclusion, all domains are end domains. We denote finite domains by
      $\{U_1,U_2,\ldots,U_s\}$ which are bounded away from $0$ and $\infty$, infinite domains by $\{V_1,V_2,\ldots,V_t\}$ whose closure contain $0$ or $\infty$ .
    
     For each finite domain $U_i$, since $F(z)$ maps $U_i$ to the upper half plane or lower half plane and continuously extend to the boundary, there must be a pole $\xi_l$ on $\partial U_i$, and each pole $\xi_l$ must belong to the boundary of two adjoint finite domains. Since we have exactly $m$ poles $\{\xi_1,\xi_2,\ldots,\xi_m \}$, there are $2m$ finite domains, so $s=2m$.

   Since the pole at $z=0$ is of order $2m-n-2$, by the basic structure theorem (\ref{basic structure theorem}), there are $n-2m$ end domains whose closure contains $0$,  by involution symmetry, there are $n-2m$ end domains whose closure contains $\infty$, so $t=2n-4m$. 
   
   Note that each finite domain $U_i$ is curved out by disjoint arcs connecting pairs of zeros in $\{z_1,z_2,\ldots,z_n\}$.  By involution symmetry, there are $m$ finite domains in $\mathbb{D}$. Consequently, there are exactly $m$ arcs in $\Gamma(Q)$ connecting $m$ pairs of zeros.

  Since there are exactly $n-2m$ infinite domains, we have $n-2m$ trajectories in $\Gamma(Q)$ with limiting ends at $0$ and with limiting tangential directions forming equal angles $\frac{2\pi}{n-2m}$. These trajectories will end at the remaining $n-2m$ zeros on the boundary .
   
  $\Gamma(Q)$ form a radial $(n,m)$ link pattern.

      \item If $n$ is even and $n=2m$, which implies that, the poles of $Q(z)$ at $0$ and $\infty$ are both of degree $2$. 

      By the basic structure theorem (\ref{basic structure theorem}), $\widehat{\mathbb{C}} - \Gamma(Q)$, consists of a finite number of end, strip, circle and ring domains.

    Now, we prove that there can not be strip or ring domains.
   For $Q(z)dz^2 \in \mathcal{QD}(z)$, by lemma (\ref{quadratic differential analytic function}) and lemma (\ref{horizontal trajectories level lines}),  $F(z)=\int \sqrt{Q(z)}dz$ take real values on $\Gamma(Q)$.  

   Therefore, $F(z)$ can not map $U$ to a strip , in this case on $\partial U$, ${\rm Im}F$ take 2 differential values. Similarly, $F(z)$ can not map $U$ to a ring domian, in this case on $\partial U$, ${\rm Im}F$ also take 2 differential values.

    When $n=2m$, the degree of pole at $z=0$ and $z=\infty$ are exactly $2$, so there is exactly one circle domain at $z=0$ and involution symmetrically one circle domain at $z=\infty$.

    In conclusion, we have two circle domains and others are end domains.

     We denote these end domains by
      $\{U_1,U_2,\ldots,U_s\}$, which are bounded away from $0$ and $\infty$,
    
     For each end domain $U_i$, since $F(z)$ maps $U_i$ to the upper half plane or lower half plane and continuously extends to the boundary, there must be a pole $\xi_l$ on $\partial U_i$, and each pole $\xi_l$ must belong to the boundary of two adjoint finite domains. Since we have exactly $m$ poles $\{\xi_1,\xi_2,\ldots,\xi_m \}$, there are $2m$ finite domains, so $s=2m$. 

      Note that each finite domain $U_i$ is curved out by disjoint arcs connecting pairs of zeros in $\{z_1,z_2,\ldots,z_n\}$.  By involution symmetry, there are $m$ finite domains in $\mathbb{D}$. Consequently, there are exactly $m$ arcs in $\Gamma(Q)$ connecting $m$ pairs of zeros.
    
    $\Gamma(F)$ form a radial $(n,m)$ link pattern.
    
      \item If $n$ is even and $2m>n$, consider the function $F(z)$ which is the primitive of $\sqrt{Q(z)}$. In this case, by theorem (\ref{quadratic differential analytic function}), 
      $F(z)=R(z)$ is a rational function.
      The degree of $F(z)$ at $0$ is $m-\frac{n}{2}>0$. The real locus $\Gamma(F)$ are thus regular near 0. 
      There are $2m-n$ trajectories with ends at $0$,
 and with limiting tangential directions that make equal angles $\frac{2\pi}{2m-n}$
 with each other.
 
     $\Gamma(F)$ form a radial $(n,n-m)$ link pattern.
      \item If $n$ is odd,  consider the primitive of $\sqrt{Q(z)}$, $F(z)=\int \sqrt{Q(z)}dz$. In this case, by theorem (\ref{quadratic differential analytic function}), 
      $F(z)=\sqrt{z}R(z)$ where $R(z)$ is a rational function.
      
      By taking the double cover of $\mathbb{D}$,  $F(z^2)=zR(z^2)$ is a rational function.
      The degree of $F(z^2)$ at $0$ is $|2m-n|$. 
       There are $|2n-4m|$  trajectories with ends at 0.
    
      By projecting back, we see that $\Gamma(F(z))$ has $|n-2m|$  trajectories with limiting ends at 0
      and with limiting tangential directions forming equal angles $\frac{2\pi}{|n-2m|}$
 with each other.
      
      Therefore if $n>2m$, $\Gamma(F)$ form a radial $(n,m)$ link pattern;  if $n<2m$, $\Gamma(F)$ form a radial $(n,n-m)$ link pattern.
      
  \end{itemize}  
\end{proof}

\subsection{Field integral of motion and horizontal trajectories as flow lines}
\begin{thm}\label{integral of motion in D} 
In the unit disk \( \mathbb{D} \), let \( z_1, z_2, \dots, z_n \) be distinct growth points on \( \partial \mathbb{D} \). For each \( z \in \overline{\mathbb{D}} \), define the following:

\[
\left\{
\begin{aligned}
A(t) &= \frac{\prod_{j=1}^{m} \xi_k^2(t)}{ \prod_{k=1}^{n} z_k(t)}, \\
B_t(z) &= e^{-(2m-n)\left(\int_0^t \sum_j \nu_j(s) \, ds\right)} \, g_t(z)^{2m-n-2} (g'_t(z))^2 \frac{\prod_{k=1}^{n} (g_t(z) - z_k(t))^2}{\prod_{j=1}^{m} (g_t(z) - \xi_j(t))^4}, \\
N_t(z) &= A(t) B_t(z) = e^{-(2m-n)\left(\int_0^t \sum_j \nu_j(s) \, ds\right)} \frac{\prod_{j=1}^{m} \xi_k(t)^2}{\prod_{k=1}^{n} z_k(t)} \, g_t(z)^{2m-n-2} (g'_t(z))^2 \frac{\prod_{k=1}^{n} (g_t(z) - z_k(t))^2}{\prod_{j=1}^{m} (g_t(z) - \xi_j(t))^4}.
\end{aligned}
\right.
\]

Then, \( A(t) \), \( B_t(z) \), and \( N_t(z) \) are field integrals of motion on the interval \( [0, \tau_t \wedge \tau) \) for the multiple radial SLE(0) Loewner flows with parametrization \( \nu_j(t) \), \( j = 1, \dots, n \).
\end{thm}

    \( N_t(z) \) is a field integral of motion for arbitrary initial positions of screening charges \( \boldsymbol{\xi} \) even without assuming stationary relations. The stationary relations imply the existence of a quadratic differential \( Q(z) dz^2 \in \mathcal{QD}(\boldsymbol{z}) \), see Theorem \ref{traces as horizontal trajectories}.

    The integral of motion is motivated by a martingale observable in conformal field theory. For a field \( \mathcal{X} \) in the OPE family \( F_{\beta} \),
    \[
    \hat{\mathbf{E}}[\mathcal{X}] := \frac{\mathbf{E}[\mathcal{X} \mathcal{O}_{\beta}]}{\mathbf{E}[\mathcal{O}_{\beta}]}
    \]
    is a martingale observable where \( \mathcal{O}_{\beta} \) is a vertex field. In our situation, we choose \( \mathcal{X} \) to be the chiral vertex field and take the classical limit as \( \kappa \to 0 \). The martingale observable degenerates to the integral of motion. We will discuss the construction of the field \( \mathcal{X} \) in Section \ref{Multiple radial Martingale Observable}.

In the proof of Theorem \ref{traces as horizontal trajectories}, we also need to consider \( \sqrt{N_t(z)} \) as a field integral of motion. However, an obstacle arises in the expression
\[
\sqrt{N_t(z)} = e^{-(m - \frac{n}{2})\left( \int_0^t \sum_j \nu_j(s) \, ds \right)} \, g_t(z)^{m - \frac{n}{2} - 1} \, g'_t(z) \, \frac{\prod_{k=1}^{n} (g_t(z) - z_k(t))}{\prod_{j=1}^{m} (g_t(z) - \xi_j(t))^2},
\]
where the term \( g_t(z)^{m - \frac{n}{2} - 1} \) becomes multivalued when \( n \) is an odd integer, thus \( \sqrt{N_t(z)} \) is in fact not well defined. To resolve this technical problem, we introduce the angular coordinate.

\begin{cor} \label{angular integral of motion}
    In the angular coordinate, by changing variables, let \( \xi_k = e^{i\zeta_k} \), \( z_k = e^{i\theta_k} \), and \( h_t(z) \) be the covering map of \( g_t(z) \) (i.e., \( e^{ih_t(z)} = g_t(e^{iz}) \)). For each \( z \in \overline{\mathbb{H}} \), we define:
    \begin{equation}
    \left\{
    \begin{aligned}
    A^{ang}(t) &= \frac{\prod_{j=1}^{m} e^{i\zeta_{k}(t)}}{ \prod_{k=1}^{n} e^{i\frac{\theta_k(t)}{2}}}, \\
    B^{ang}_t(z) &= e^{-(m-\frac{n}{2})\left( \int_0^t \sum_j \nu_j(s) \, ds \right)} g_t(z)^{m-\frac{n}{2}-1} e^{i(m-\frac{n}{2}-1) h_t(z)} h_{t}'(z) e^{i h_t(z)} \frac{\prod_{k=1}^{n} \left(e^{i h_t(z)} - e^{i\theta_k(t)}\right)}{\prod_{j=1}^{m} \left(e^{i h_t(z)} - e^{i\zeta_{j}(t)}\right)^2}, \\
    N_{t}^{ang}(z) &= A^{ang}(t) B^{ang}_t(z) \\
    &= e^{-(m-\frac{n}{2})\left( \int_0^t \sum_j \nu_j(s) \, ds \right)} \frac{\prod_{j=1}^{m} e^{i\zeta_{k}(t)}}{ \prod_{k=1}^{n} e^{i \frac{\theta_k(t)}{2}}} e^{i(m-\frac{n}{2}-1) h_t(z)} h_{t}'(z) e^{i h_t(z)} \frac{\prod_{k=1}^{n} \left(e^{i h_t(z)} - e^{i \theta_k(t)}\right)}{\prod_{j=1}^{m} \left(e^{i h_t(z)} - e^{i \zeta_{j}(t)}\right)^2}.
    \end{aligned}
    \right.
    \end{equation}
    Then, \( A^{ang}(t) \), \( B^{ang}_t(z) \), and \( N^{ang}_t(z) \) are field integrals of motion on the interval \( [0, \tau_t \wedge \tau) \) for the multiple radial SLE(0) Loewner flows with parametrization \( \nu_j(t) \), \( j = 1, \dots, n \).
\end{cor}

For a multiple radial SLE(0) with growth points $\{z_1,z_2,\ldots,z_n\}$ and screening charges $\{\xi_1,\xi_2,\ldots,\xi_m\}$, we construct a quadratic differential $Q(z)dz^2 \in \mathcal{QD}(\boldsymbol{z})$, where
$$
Q(z)= 
\frac{\prod_{k=1}^{m}\xi_{k}^2}{ \prod_{j=1}^{n}z_j}
z^{2m-n-2}\frac{\prod_{j=1}^{ n}\left(z-z_j\right)^2}{\prod_{k=1}^{m}\left(z-\xi_k\right)^4},
$$
In this section, we will show that the traces of a multiple radial SLE(0) system coincide with the horizontal trajectories of $Q(z)dz^2 \in \mathcal{QD}(\boldsymbol{z})$ with ends at $\{z_1,z_2,\ldots,z_n\}$ (which are double zeros of $Q(z)$). These trajectories are also part of the real locus of $F(z)=\int \sqrt{Q(z)}dz$.

From the dynamical point of view, the key ingredient in our proof of the main theorem (\ref{traces as horizontal trajectories}) is the field of integral of motions for the multiple radial Loewner flow. This field integral of motion can be derived as the classical limit of a martingale observable constructed via conformal field theory, see section \ref{Multiple radial Martingale Observable}.

\begin{lemma}
In the unit disk $\mathbb{D}$, let $z_1,z_2,\ldots,z_n$ be distinct growth points on the unit circle, and $\xi_1,\xi_2,\ldots,\xi_m$ be marked points, for each $z \in \overline{\mathbb{D}}$, we denote 
\begin{equation}
\left\{
\begin{aligned}
&A(t)= \frac{\prod_{j=1}^{m}\xi_{k}^2(t)}{ \prod_{k=1}^{n}z_k(t)}\\
&B_t(z)=e^{-(2m-n)t}g_t(z)^{2m-n-2}(g'_{t}(z))^2\frac{\prod_{k=1}^{n}(g_t(z)-z_k(t))^2}{\prod_{j=1}^{m}(g_t(z)-\xi_j(t))^4} \\
&N_{t}(z)=A(t)B_t(z)=e^{-(2m-n)t}
\frac{\prod_{j=1}^{m}\xi_k(t)^2}{ \prod_{k=1}^{n}z_k(t)}
g_t(z)^{2m-n-2}(g'_{t}(z))^2\frac{\prod_{k=1}^{n}(g_t(z)-z_k(t))^2}{\prod_{j=1}^{m}(g_t(z)-\xi_j(t))^4}
\end{aligned}
\right.
\end{equation}

then $A(t)$, $B(t)$, $N(t)$ are integrals of motion on the interval $[0,\tau_t \wedge \tau)$ for the multiple radial Loewner flows with parametrization $\nu_j(t)=1$, $\nu_k(t)=0$, $k\neq j$, i.e., only the $j$-th curve is growing
\end{lemma}
\begin{proof}

By the Loewner equation, the following identities hold:
\begin{equation}
\left\{
    \begin{aligned}
&\frac{dz_j(t)}{dt} =\sum_{k\neq j}z_j(t)\frac{z_j(t)+z_k(t)}{z_{k}(t)-z_j(t)}- 2\sum_{l}z_j(t)\frac{z_j(t)+\xi_l(t)}{\xi_{l}(t)-z_j(t)} \\
 &\frac{dz_{k}(t)}{dt}= z_{k}(t)\frac{z_j(t)+z_{k}(t)}{z_j(t)-z_{k}(t)},  k \neq j \\
  &  \frac{d\xi_{l}(t)}{dt}= \xi_{l}(t)\frac{z_j(t)+\xi_{l}(t)}{z_j(t)-\xi_{l}(t)} \\
    \end{aligned}
\right.
\end{equation}
By substituting above equations into $\frac{d \log A(t)}{dt}$, we obtain that:

\begin{equation}
\log A(t)=2\sum_{j=1}^{m}\log\xi_k(t)-\sum_{k=1}^{n}\log z_k(t)
\end{equation}

\begin{equation}\label{d Aj t}
\begin{aligned}
\frac{d \log A(t)}{dt} & =   2\sum_{l=1}^{m} \frac{d\log \xi_l(t)}{dt}- \sum_{j=1}^{n} \frac{d\log z_j(t)}{dt}
\\
&=\underbrace{ 2\sum_{l=1}^{m}\frac{z_j(t)+\xi_{l}(t)}{z_j(t)-\xi_{l}(t)} + \sum_{k\neq j}\frac{z_j(t)+z_{k}(t)}{z_j(t)-z_{k}(t)}+ \sum_{k\neq j}\frac{z_j(t)+z_k(t)}{z_{k}(t)-z_j(t)}- 2\sum_{l}\frac{z_j(t)+\xi_l(t)}{\xi_{l}(t)-z_j(t)} }_{A^{j}(t)}\\
&=0
\end{aligned}
\end{equation}

We denote the sum in the second line of the equation (\ref{d Aj t})  by $A^{j}(t)$. After simplifying, it is clear that the sum $A^{j}(t)=0$.

Again, by Loewner equation, the following identities hold:
\begin{equation}
\left\{
\begin{aligned}
     &  \frac{dg_t(z)}{dt} = g_t(z) \frac{z_j(t)+g_t(z)}{z_j(t)-g_t(z)} \\
    & \frac{d\log g_{t}'(z)}{dt}= \frac{z_j(t)+g_t}{z_j(t)-g_t}+\frac{2z_j(t) g_t}{(z_j(t)-g_t)^2} \\
    & \frac{d\log g_t(z)}{dt}= \frac{z_j(t)+g_t(z)}{z_j(t)-g_t(z)} \\
    &\frac{d\log(g_t'(z)/g_t)}{dt}= \frac{2g_t(z)z_j(t)}{(z_j(t)-g_t)^2}=
    \frac{g_t(z)(g_{t}(z)+z_j(t))}{(z_j(t)-g_t(z))^2}+ \frac{g_t(z)}{z_j(t)-g_t(z)}
\end{aligned}
\right.
\end{equation}

\begin{equation} \label{derivatives of terms}
\left\{
\begin{aligned}
    & \frac{d\log(z_k(t) - g_t(z))}{dt}=\frac{1}{z_k-g_t}(z_{k}(t)\frac{z_j(t)+z_{k}(t)}{z_j(t)-z_{k}(t)}-g_t(z) \frac{z_j(t)+g_t(z)}{z_j(t)-g_t(z)})\\
    &=-\frac{z_j(t)(z_j(t)+z_k(t))}{(z_j(t)-g_t(z))(z_k(t)-z_j(t))} +\frac{g_t}{z_j(t)-g_t(z)},   k\neq j \\
    & \frac{d\log(\xi_l(t) - g_t(z))}{dt}=
     \frac{1}{\xi_l-g_t}(\xi_l(t)\frac{z_j(t)+\xi_{l}(t)}{z_j(t)-\xi_{l}(t)}-g_t(z) \frac{z_j(t)+g_t(z)}{z_j(t)-g_t(z)})
     \\
     &= -\frac{z_j(t)(z_j(t)+\xi_l(t))}{(z_j(t)-g_t(z))(\xi_l(t)-z_j(t))} +\frac{g_t}{z_j(t)-g_t(z)}, \\  
& \frac{d\log(z_j(t)-g_t(z))}{dt}=\frac{1}{z_j(t)-g_t}\left(z_j(t)\sum_{k\neq j}\frac{z_j(t)+z_k(t)}{z_k(t)-z_j(t)}-2\sum_{l}z_j(t)\frac{z_j(t)+\xi_l(t)}{\xi_{l}(t)-z_j(t)} -g_t(z)\frac{z_j(t)+g_t(z)}{z_j(t)-g_t} \right) \\
\end{aligned}
\right.
\end{equation}

\begin{equation} \label{Bjt}
\log B_t(z)=-(2m-n)t+(2m-n-2)\log g_t(z)+ 2\log (g'_{t}(z))+2\sum_{k=1}^{n}\log(g_t(z)-z_k(t))-4\sum_{j=1}^{m}\log(g_t(z)-\xi_j(t))
\end{equation}
By substituting equations (\ref{derivatives of terms}) into equation (\ref{Bjt}).
\begin{equation} \label{dBj t}
\frac{d\log B_t(z)}{dt} =
 \underbrace{\begin{aligned}
 &-(2m-n)+(2m-n-2)\frac{z_j(t)+g_t(z)}{z_j(t)-g_t(z)}+2\left( \frac{z_j(t)+g_t}{z_j(t)-g_t}+\frac{2z_j(t) g_t}{(z_j(t)-g_t)^2}\right)+ \\
&+2\sum_{k=1}^{n}\left(-\frac{z_j(t)(z_j(t)+z_k(t))}{(z_j(t)-g_t(z))(z_k(t)-z_j(t))} +\frac{g_t}{z_j(t)-g_t(z)} \right)\\
&-4 \sum_{j=1}^{m}\frac{1}{z_j(t)-g_t}\left(z_j(t)\sum_{k\neq j}\frac{z_j(t)+z_k(t)}{z_k(t)-z_j(t)}-2\sum_{l}z_j(t)\frac{z_j(t)+\xi_l(t)}{\xi_{l}(t)-z_j(t)} -g_t(z)\frac{z_j(t)+g_t(z)}{z_j(t)-g_t} \right) 
\end{aligned}}_{B^{j}_{t}(z)}
\end{equation}
We denote the sum on the right hand side of (\ref{dBj t}) by $B^{j}_{t}(z)$.
By direct computations we obtain that all terms canceled out,
$$\frac{d \log B_t(z)}{dt}=B^{j}_{t}(z)=0 $$

which implies
$$\frac{d \log N_t(z)}{dt} =\frac{d \log A(t)}{dt}+\frac{d \log B_t(z)}{dt}= 0$$

\end{proof}

\begin{proof}[Proof of theorem (\ref{integral of motion in D})] 
Note that for $\nu(t)$ parametrization
$$
\partial_t g_t(z)=\sum_{j=1}^n \nu_j(t)g_t(z)\frac{z_j(t)+g_t(z)}{z_j(t)-g_t(z)}, \quad g_0(z)=z,
$$

$$
\left\{
\begin{aligned}
    &\frac{dz_j(t)}{dt} =\nu_j(t)\left(\sum_{k\neq j}z_j(t)\frac{z_j(t)+z_k(t)}{z_{k}(t)-z_j(t)}+ 2\sum_{l}z_j(t)\frac{z_j(t)+\xi_l(t)}{\xi_{l}(t)-z_j(t)}\right)+\sum_{k\neq j}\nu_k(t)\left( z_{j}(t)\frac{z_j(t)+z_{k}(t)}{z_k(t)-z_{j}(t)}\right) \\
   & \frac{d\xi_{l}(t)}{dt}=\sum_{j}\nu_j(t)\left( \xi_{l}(t)\frac{z_j(t)+\xi_{l}(t)}{z_j(t)-\xi_{l}(t)}\right)
   \end{aligned}
\right.
$$
$$
\left\{
 \begin{aligned}
    & \frac{dg_t}{dt} =\sum_{j}\nu_j(t)\left( g_t(z) \frac{z_j(t)+g_t(z)}{z_j(t)-g_t(z)}\right) \\
   &  \frac{d\log g_{t}'(z)}{dt}= \sum_{j}\nu_j(t)\left(\frac{z_j(t)+g_t}{z_j(t)-g_t}+\frac{2z_j(t) g_t}{(z_j(t)-g_t)^2}\right) \\
    & \frac{d\log g_t(z)}{dt}= \sum_{j}\nu_j(t)\left(\frac{z_j(t)+g_t(z)}{z_j(t)-g_t(z)} \right)\\
    &\frac{d\log(g_t'(z)/g_t)}{dt}= \sum_{j}\nu_j(t)\frac{2g_t(z)z_j(t)}{(z_j(t)-g_t)^2}= \sum_{j=1}\nu_j(t) \left(
    \frac{g_t(z)(g_{t}(z)+z_j(t))}{(z_j(t)-g_t(z))^2}+ \frac{g_t(z)}{z_j(t)-g_t(z)}\right)
\end{aligned}
\right.
$$
$$
\left\{
\begin{aligned}
    & \frac{d\log(\xi_l(t) - g_t(z))}{dt}=
     \frac{1}{\xi_l-g_t}\sum_{j}\nu_j(t)\left(\xi_l(t)\frac{z_j(t)+\xi_{l}(t)}{z_j(t)-\xi_{l}(t)}-g_t(z) \frac{z_j(t)+g_t(z)}{z_j(t)-g_t(z)}\right)
     \\
     &= \sum_{j}\nu_{j}(t)\left(\frac{z_{j}^2+g_tz_j+\xi_lz_j-\xi_lg_t}{(z_j(t)-\xi_l(t))(z_j(t)-g_t)}\right)= \sum_{j}\nu_{j}(t)\left(-\frac{z_j(t)(z_j(t)+\xi_l(t))}{(z_j(t)-g_t(z))(\xi_l(t)-z_j(t))} +\frac{g_t}{z_j(t)-g_t(z)}\right)   \\
     & \frac{d\log(z_j(t)-g_t(z))}{dt}=\frac{1}{z_j(t)-g_t}\nu_{j}(t)\left(z_j(t)\sum_{k\neq j}\frac{z_j(t)+z_k(t)}{z_k(t)-z_j(t)}-2\sum_{l}z_j(t)\frac{z_j(t)+\xi_l(t)}{\xi_{l}(t)-z_j(t)} -g_t\frac{z_j(t)+g_t(z)}{z_j(t)-g_t}\right) \\
     &+\frac{1}{z_j(t)-g_t}\sum_{k\neq j}\nu_k(t)\left( z_{j}(t)\frac{z_j(t)+z_{k}(t)}{z_k(t)-z_{j}(t)} -g_t\frac{z_j(t)+g_t(z)}{z_j(t)-g_t}\right) 
     \end{aligned} 
     \right.
$$

By plugging in these identities, we obtain that 
$$\frac{d \log A(t)}{dt}= \sum_{j=1}^{n}\nu_j(t)  A^{j}(t)=0$$
$$\frac{d \log B_t(z)}{dt}= \sum_{j=1}^{n}\nu_j(t) B^{j}_t(z)=0$$
$$\frac{d \log N_t(z)}{dt}= \sum_{j=1}^{n}\nu_j(t) \frac{d \log N^{j}_t(z)}{dt}=0$$
\end{proof}

\begin{proof}{Theorem (\ref{traces as horizontal trajectories})}
We first prove that $Q(z) \circ g_t^{-1}$ is in $\mathcal{QD}(\boldsymbol{z}(t))$. 

Since at $t=0$, the screening charges $\boldsymbol{\xi}$ are assumed to be involution symmetric and solve the stationary relations, stationary relations- residue free theorem guarantees the existence of an $Q_0(z)dz^2 \in \mathcal{QD}(\boldsymbol{z}(0))$ with $\boldsymbol{\xi}(0)$ as poles and $\boldsymbol{z}(0)$ as zeros. 

Moreover, $Q_0(z)$ factors as
\begin{equation}
Q_0(z)= 
\frac{\prod_{k=1}^{m}\xi_{k}^2(0)}{ \prod_{j=1}^{n}z_j(0)}
z^{2m-n-2}\frac{\prod_{j=1}^{ n}\left(z-z_j(0)\right)^2}{\prod_{k=1}^{m}\left(z-\xi_k(0)\right)^4},
\end{equation}

Using the integral of motion $N_{\mathrm{t}}(z)$ , we have
\begin{equation}
 \begin{aligned}
N_t(z)= & \frac{\prod_{j=1}^{m}\xi_{k}^2(t)}{ \prod_{k=1}^{n}z_k(t)} e^{-(2m-n-2)(\int_{0}^{t}\sum_{j}\nu_j(s)ds)} g_{t}^{2m-n-2}(z)(g_t^{\prime}(z))^2 \frac{\prod_{j=1}^{n}\left(g_t(z)-z_j(t)\right)^2}{\prod_{k=1}^{m}\left(g_t(z)-g_t\left(\xi_k\right)\right)^4} \\
& =\frac{\prod_{k=1}^{m}\xi_{k}^2(0)}{ \prod_{j=1}^{n}z_j(0)}z^{2m-n-2}\frac{\prod_{j=1}^{n}\left(z-y_j\right)^2}{\prod_{k=1}^{m}\left(z-\xi_k\right)^4}=N_0(z) .
\end{aligned}   
\end{equation}

Denote the constant $\mu(t)=e^{-(m-\frac{n}{2}-1) (\int_{0}^{t}\sum_{j}\nu_j(s)ds)}$.

Let $f_t=g_t^{-1}$, and since the above holds everywhere evaluate it at $f_t(z)$ to obtain

\begin{equation} \label{integral of motion ft}
\begin{aligned}
& \mu^2(t)\frac{\prod_{j=1}^{m}\xi_{k}^2(t)}{ \prod_{k=1}^{n}z_k(t)}
z^{2m-n-2}\frac{\prod_{j=1}^{n}\left(z-z_j(t)\right)^2}{\prod_{k=1}^{m}\left(z-g_t\left(\xi_k\right)\right)^4} \\
=&\frac{\prod_{k=1}^{m}\xi_{k}^2(0)}{ \prod_{j=1}^{n}z_j(0)} f_t^{\prime}(z)^2 f_{t}^{2m-n-2}(z) \frac{\prod_{j=1}^{n}\left(f_t(z)-z_j\right)^2}{\prod_{k=1}^{m}\left(f_t(z)-\zeta_t\right)^4}=f_t^{\prime}(z)^2 Q_0\left(f_t(z)\right)
\end{aligned}
\end{equation}

The left-hand side is exactly
\begin{equation}
     \mu^2(t)z^{m-\frac{n}{2}-1}\frac{\prod_{j=1}^{n}\left(z-z_j(t)\right)}{\prod_{k=1}^{m}\left(z-g_t\left(\xi_k\right)\right)^2}=\mu(t)Q_t(z)
\end{equation}

 we obtain that 
\begin{equation} \label{residue identity}
   \pm \mu(t) Res_{\xi(t)}(\sqrt{Q_t(z)}) = Res_{\xi(0)}(\sqrt{Q_0(f_t(z))} = 0
\end{equation}
The stationary relations at $t=0$ give the last equality. Thus, the residue-free condition holds, and clearly, the involution symmetry of $\boldsymbol{\xi}$ is preserved, which implies $Q(z) \circ g_t^{-1} \in \mathcal{QD}(\boldsymbol{z}(t))$.

Finally, we prove that the hull $K_t$ is a subset of
the horizontal trajectories of $Q(z)dz^2$ with limiting ends at $\{z_1,z_2,\ldots,z_n\}$. By theorem (\ref{horizontal trajectories level lines}), equivalently, we can show that $K_t$ is a subset of the real locus of the $F(z)=\int \sqrt{Q(z)}dz$ (which is well defined as shown in lemma (\ref{real locus well defined})).

Note that $F(z)$ is a multivalued function. To deal with the multi-valuedness, let $\rho(z)=e^{iz}$ be the exponential covering map, and we consider $h_t(z)$ the lifting map of $g_t(z)$ (i.e. $e^{ih_t(z)}=g_t(e^{iz})$).
We denote the lifting of $F(z)$ by $\tilde{F}(z)$, which is now a single-valued function, and the lifting of $Q_t(z)$ by
$\tilde{Q}_t(z)=-Q_t(e^{iz})e^{2iz}dz$, the lifting of the hull $K_t$ by $\tilde{K}_t$.

By the integral of motion in angular coordinate (\ref{angular integral of motion}) 

\begin{equation}
\begin{aligned}
N_{t}^{ang}(z)=&\mu(t)
\frac{\prod_{j=1}^{m}e^{i\zeta_{k}(t)}}{ \prod_{k=1}^{n}e^{i\frac{\theta_k(t)}{2}}}
e^{i(m-\frac{n}{2}-1)h_t(z)}h_{t}'(z)e^{ih_t(z)}\frac{\prod_{k=1}^{n}(e^{ih_t(z)}-e^{i\theta_k(t)})}{\prod_{j=1}^{m}(e^{ih_t(z)}-e^{i\zeta_{j}(t)})^2} \\
=& \frac{\prod_{j=1}^{m}e^{i\zeta_{k}(0)}}{ \prod_{k=1}^{n}e^{i\frac{\theta_k(0)}{2}}}
e^{i(m-\frac{n}{2}-1)z}e^{iz}\frac{\prod_{k=1}^{n}(e^{iz}-e^{i\theta_k(0)})}{\prod_{j=1}^{m}(e^{iz}-e^{i\zeta_{j}(0)})^2}= N^{ang}_{0}(z)
\end{aligned}
\end{equation}
Let $s_t=h_t^{-1}$, and since the above holds everywhere evaluate it at $s_t(z)$ to obtain
\begin{equation}
\begin{aligned}
\mu(t)\sqrt{\tilde{Q}_t(z)}&=\mu(t)
\frac{\prod_{j=1}^{m}e^{i\zeta_{k}(t)}}{ \prod_{k=1}^{n}e^{i\frac{\theta_k(t)}{2}}}
e^{i(m-\frac{n}{2}-1)z}e^{iz}\frac{\prod_{k=1}^{n}(e^{iz}-e^{i\theta_k(t)})}{\prod_{j=1}^{m}(e^{iz}-e^{i\zeta_{j}(t)})^2} \\
&= \frac{\prod_{j=1}^{m}e^{i\zeta_{k}(0)}}{ \prod_{k=1}^{n}e^{i\frac{\theta_k(0)}{2}}}
e^{i(m-\frac{n}{2}-1)s_t(z)}e^{is_t(z)}\frac{\prod_{k=1}^{n}(e^{is_t(z)}-e^{i\theta_k(0)})}{\prod_{j=1}^{m}(e^{is_t(z)}-e^{i\zeta_{j}(0)})^2}(h_{t}^{-1}(z))'\\
&= \sqrt{Q_0(s_t(z))}(s_t(z))'\\
&=(\tilde{F}(s_t(z)))'
\end{aligned}
\end{equation}

Since $\mu(t)$ is a real constant, $\tilde{F}(s_t(z))$ is the primitive of $\sqrt{\tilde{Q}_t(z)}$ (up to a real multiplicative constant).
It suffices to show that $\tilde{K_t}$ is the real locus of $\tilde{F}(z)$

Recall that $g_t$ is the unique conformal map from $\mathbb{D} \backslash K_t$ onto $\mathbb{D}$ with the hydrodynamic normalization. Therefore, $g_t(z)$ maps the subset $K_t$ to $\partial\mathbb{D}$ and $h_t(z)$ maps the subset $\tilde{K_t}$ to the real line.

Since $\tilde{F}(s_t(z))$ is the primitive of $\sqrt{\tilde{Q_t}(z)}$ (up to a real multiplicative constant), the real line is part of the real locus $\tilde{F}(s_t(z))$.  Since $h_t(z)$ maps the subset $\tilde{K_t}$ to the real line,
it follows that $\tilde{K_t}$ is a subset of the real locus of $\tilde{F}(z)$.
\end{proof}

\begin{remark}
 The principle underlying is that for $Q(z) \in \mathcal{QD}(\boldsymbol{z})$ then $\sqrt{Q(z)}$ has a local meromorphic primitive in $\mathbb{D}\backslash 0$, and therefore its residue at all non-zero poles must be zero. When the poles are distinct from the critical points this principle expresses itself algebraically in the form of the stationary relation. When, however, a pole overlaps with a critical point (in which case the pole is necessary of order two) then the partial fraction expansion of $\sqrt{Q(z)}$ becomes more complicated as does the corresponding algebra. However, the underlying principle remains the same.    
\end{remark}

\begin{example}
    
     The traces of $n$-braids multiple SLE(0) in $\mathbb{D}$.
\end{example}
\begin{proof}
For $n$-braids multiple radial SLE(0), $m=0$, and there are no poles, thus $Q(z)dz^2$ is given by
\begin{equation}
Q(z)=c z^{-n-2}\prod_{j=1}^{n}\left(z-z_j\right)^2
\end{equation}
 The $n$-braids multiple SLE(0) has $n$ trajectories with limiting ends at $z=0$, and with limiting tangential directions that forms an equal $\frac{2\pi}{n}$ with each other.

\end{proof}

\subsection{Classical limit of martingale observables* } \label{Multiple radial Martingale Observable}
In this section, we discuss how the field integral of motion is heuristically derived as the classical limit of martingale observables constructed via conformal field theory.

Based on the SLE-CFT correspondence, the multiple radial SLE($\kappa$) system can be coupled to a conformal field theory. We will construct this conformal field theory using vertex operators, following the approach in \cite{KM13, KM21}.

\begin{defn}[Vertex operator]
For a background charge $\boldsymbol{\beta}=\sum_k \beta_k \cdot q_k$ with the neutrality condition $\left(\mathrm{NC}_b\right)$ and  divisor $\boldsymbol{\tau}=\sum_j \tau_j \cdot z_j$ with the neutrality condition $\left(\mathrm{NC}_0\right)$.  
We define the vertex operator $\mathcal{O}_{\boldsymbol{\beta}}[\boldsymbol{\tau}]$ as
 \begin{equation}
 \mathcal{O}_{\boldsymbol{\beta}}[\boldsymbol{\tau}]:=\frac{C_{(b)}[\boldsymbol{\tau}+\boldsymbol{\beta}]}{C_{(b)}[\boldsymbol{\beta}]} e^{\odot i \Phi^{+}[\boldsymbol{\tau}]}.
 \end{equation}
 where $\Phi^{+}[\boldsymbol{\tau}]:=\sum \tau_j \Phi^{+}\left(z_j\right)$ is the chirdal bosonic field and $\odot$ is the wick product.

\end{defn}

\begin{defn}[$n$-leg operator with screening charges]
    
Consider the following charge distribution on the Riemann sphere.

$$
\boldsymbol{\beta}=b \delta_{0}+b\delta_{\infty}
$$

$$\boldsymbol{\tau_1}=\sum_{j=1}^{n} a \delta_{z_j}-\sum_{k=1}^m 2 a \delta_{\xi_k}-(\frac{n-2m}{2})a \delta_{0}-(\frac{n-2m}{2})a\ \delta_{\infty}$$

The $n$-leg operator with screening charges $\boldsymbol{\xi}$ and background charge $\boldsymbol{\beta}$ is given by the OPE exponential: 

\begin{equation}
\mathcal{O}_{\boldsymbol{\beta}}[\boldsymbol{\tau_1}]=\frac{C_{(b)}[\boldsymbol{\tau_1}+\boldsymbol{\beta}]}{C_{(b)}[\boldsymbol{\beta}]} \mathrm{e}^{\odot i \Phi[\boldsymbol{\tau_1}]}.
\end{equation}

\end{defn}

For each link pattern $\alpha$, we can choose closed contours $\mathcal{C}_1, \ldots, \mathcal{C}_n$ along which we may integrate the $\boldsymbol{\xi}$ variables to screen the vertex fields.
Let $\mathcal{S}$ be the screening operator. We define the screening operation as 

\begin{equation}
\mathcal{S}_{\alpha}\mathcal{O}_{\boldsymbol{\beta}}[\boldsymbol{\tau_1}]=\oint_{\mathcal{C}_1} \ldots \oint_{\mathcal{C}_n} \mathcal{O}_{\boldsymbol{\beta}}[\boldsymbol{\tau_1}].
\end{equation}

We integrate the correlation function $\mathbf{E}\mathcal{O}_{\boldsymbol{\beta}}[\boldsymbol{\tau_1}]=\Phi_\kappa(\boldsymbol{z}, \boldsymbol{\xi})$, the conformal dimension is 1 at the $\boldsymbol{\xi}$ points, i.e. since $\lambda_b(-2 a)=1$. This leads to the partition function for the corresponding multiple radial SLE($\kappa$) system:

\begin{equation}
\mathcal{J}_{\alpha}(\boldsymbol{z}):=\mathbf{E}\mathcal{S}_{\alpha}\mathcal{O}_{\boldsymbol{\beta}}[\boldsymbol{\tau_1}]=\oint_{\mathcal{C}_1} \ldots \oint_{\mathcal{C}_n} \Phi_\kappa(\boldsymbol{z}, \boldsymbol{\xi}) d \xi_n \ldots d \xi_1 .
\end{equation}

\begin{thm}[Martingale observable]
 For any tensor product $X$ of fields in the OPE family $\mathcal{F}_{\boldsymbol{\beta}}$ of $\Phi_{\boldsymbol{\beta}}$,
\begin{equation}
M_t(X)=\frac{\mathbf{E}(\mathcal{S}_{\alpha} \mathcal{O}_{\boldsymbol{\beta}}[\boldsymbol{\tau_1}] X)}{\mathbf{E}\mathcal{S}_{\alpha} \mathcal{O}_{\boldsymbol{\beta}}[\boldsymbol{\tau_1}]} \| g_t^{-1}
\end{equation}
is a local martingale, where $g_t(z)$ is the Loewner map for multiple radial SLE($\kappa$) system associated to $\mathcal{J}_{\alpha}(\boldsymbol{z})=\mathbf{E}\mathcal{S} \mathcal{O}_{\boldsymbol{\beta}}[\boldsymbol{\tau_1}]$

\end{thm}

\begin{remark}
The structure of multiple radial $\mathrm{SLE}(\kappa)$ systems is not yet fully understood. We do not provide a rigorous justification of this theorem in this paper, nor is the validity of our results dependent on this. In particular, the integral of motion used in our arguments can be directly verified independently.

Moreover, the Martingale Observable Theorem can be extended to linear combinations of screening fields of the form
\[
\mathcal{S} \mathcal{O}_{\boldsymbol{\beta}} := \sum_{\alpha} \sigma_{\alpha} \, \mathcal{S}_{\alpha} \mathcal{O}_{\boldsymbol{\beta}}[\boldsymbol{\tau}_1],
\]
where each $\mathcal{S}_{\alpha}$ corresponds to a distinct choice of integration contours associated with a link pattern $\alpha$, and $\sigma_{\alpha} \in \mathbb{R}$ are real coefficients.
\end{remark}

\begin{cor}
Let the divisor $ \boldsymbol{\tau_2}=-\frac{\sigma}{2} \delta_{0}-\frac{\sigma}{2}\ \delta_{\infty}+\sigma \delta_z$ where the parameter $\sigma= \frac{1}{a}$, and insert $ X = \mathcal{O}_{\boldsymbol{\beta}}[\boldsymbol{\tau_2}] $

\begin{equation}
M_{t,\kappa}(z)=\frac{\mathbf{E}\mathcal{S} \mathcal{O}_{\boldsymbol{\beta}}[\boldsymbol{\tau_1}]\mathcal{O}_{\boldsymbol{\beta}}[\boldsymbol{\tau_2}]}
{\mathbf{E}\mathcal{S} \mathcal{O}_{\boldsymbol{\beta}}[\boldsymbol{\tau_1}]} \| g_t^{-1}
\end{equation}
is local martingale where $g_t(z)$ is the Loewner map for multiple radial SLE($\kappa$) system associated to $\mathcal{Z}_\kappa(\boldsymbol{z})=\mathbf{E}\mathcal{S} \mathcal{O}_{\boldsymbol{\beta}}[\boldsymbol{\tau_1}]$.
\end{cor}

Explicit computation shows that
$$
\begin{aligned}
    &\mathbf{E}\mathcal{S}_{\alpha}\mathcal{O}_{\beta}[\boldsymbol{\tau_1}]=\mathbf{E}\oint_{\mathcal{C}_1} \ldots \oint_{\mathcal{C}_n} \mathcal{O}_{\boldsymbol{\beta}}[\boldsymbol{\tau_1}]\mathcal{O}_{\boldsymbol{\beta}}[\boldsymbol{\tau_2}]\\
    &=\oint_{\mathcal{C}_1} \ldots \oint_{\mathcal{C}_n} \prod_{1 \leq i<j \leq n}(z_i-z_j)^{a^2} \prod_{1 \leq i<j \leq m}(\xi_i-\xi_j)^{4 a^2} \prod_{i=1}^{n} \prod_{j=1}^m\left(z_i-\xi_j\right)^{-2 a^2} \\
    & \prod_j z_j^{a(b-\frac{n-2m}{2}a-\frac{\sigma}{2})} \prod_k \xi_k^{-2a(b-\frac{n-2m}{2}a-\frac{\sigma}{2})} z^{\sigma(b-\frac{n-2m}{2}a)}
g^{\prime}(z_j) ^{\lambda_b(a)}g^{\prime}(z) ^{\lambda_b(\sigma)} \\
& (z-z_j)^{\sigma a}(z-\xi_k)^{-2\sigma a}|g^{\prime}(0)|^{2\lambda_b(b+\frac{2m-n}{2}a-\frac{\sigma}{2})}
\end{aligned}
$$
$$
\begin{aligned}
    &\mathbf{E}\oint_{\mathcal{C}_1} \ldots \oint_{\mathcal{C}_n} \mathcal{O}_{\boldsymbol{\beta}}[\boldsymbol{\tau_1}] \\
    &=\oint_{\mathcal{C}_1} \ldots \oint_{\mathcal{C}_n} \prod_{1 \leq i<j \leq n}(z_i-z_j)^{a^2} \prod_{1 \leq i<j \leq m}(\xi_i-\xi_j)^{4 a^2} \prod_{i=1}^{n} \prod_{j=1}^m\left(z_i-\xi_j\right)^{-2 a^2} \\
    & \prod_j z_j^{a(b-\frac{n-2m}{2})} \prod_k \xi_k^{-2a(b-\frac{n-2m}{2}a)} z^{\sigma(b-\frac{n-2m}{2}a)}
g^{\prime}(z_j) ^{\lambda_b(a)}|g^{\prime}(0)|^{2\lambda_b(b-\frac{2m-n}{2}a)}.
\end{aligned}
$$

\begin{conjecture}
 
As $\kappa \rightarrow 0$, the contour integral concentrate on the critical points of the master function.

\begin{equation}
\begin{aligned}
N_t(z)=& M_{t,0}(z)=\lim_{\kappa \rightarrow 0}M_{t,\kappa}(z)= \lim_{\kappa \rightarrow 0}\frac{\mathbf{E}\oint_{\mathcal{C}_1} \ldots \oint_{\mathcal{C}_n} \mathcal{O}_{\boldsymbol{\beta}}[\boldsymbol{\tau_1}]\mathcal{O}_{\boldsymbol{\beta}}[\boldsymbol{\tau_2}]}
{\mathbf{E}\oint_{\mathcal{C}_1} \ldots \oint_{\mathcal{C}_n} \mathcal{O}_{\boldsymbol{\beta}}[\boldsymbol{\tau_1}]}\\
&=|g^{\prime}(0)|^{-(m-\frac{n}{2})}\frac{\prod_{j=1}^{m}\xi_k}{ \sqrt{\prod_{k=1}^{n}z_k}}
z^{m-\frac{n}{2}-1}g'(z)\frac{\prod_{k=1}^{n}(z-z_k)}{\prod_{j=1}^{m}(z-\xi_j)^2},
\end{aligned}
\end{equation}
\end{conjecture}

where $ \boldsymbol{\xi}$ solve the \textbf{stationary relations}. This is exactly the integral of motion $N_t(z)$ in the proof of the theorem.

$M_{t,\kappa}(z)$ is a $(\lambda_b(\sigma),0)$ differential with respect to $z$, where $\lambda_b(\sigma)= \frac{1}{2a^2}-\frac{b}{a}$.
By taking the limit $\kappa \rightarrow 0$,
$\lim_{\kappa \rightarrow 0}\lambda_b(\sigma)=1$, and thus $M_{t,0}(z)$ is a (1,0) differential.

\begin{remark}
The integral of motion $N_t(z)$ can be verified through direct computation. This heuristic argument provides valuable insight and motivation for constructing the integral of motion. It is worth noting that our paper is self-consistent and does not depend on the detailed clarification of the classical limits.
\end{remark}

\subsection{Enumerative algebraic geometry and link pattern}\label{Enumerative geometry of radial multiple SLE(0)}

In this section, we propose several illuminating conjectures for the classification of the quadratic differential $Q(z)dz^2\in \mathcal{QD}(\boldsymbol{z})$ and equivalently the critical points of the trigonometric KZ equations:

In the chordal setting, multiple chordal $\mathrm{SLE}(0)$ systems have been constructed and analyzed in detail by \cite{PW20}, and the corresponding stationary (commutation) relations have been completely solved and classified in \cite{SV03, S02a, S02b}. Motivated by these results, we propose the following conjectures concerning the structure of multiple radial $\mathrm{SLE}(0)$ systems.

\begin{conjecture}[$n$ even]\label{radial conjecture n even quadratic differential}

Let \( Q(z)\,dz^2 \in \mathcal{QD}(\boldsymbol{z}) \) be an involution symmetric meromorphic quadratic differential with \( n \) simple zeros located on the unit circle (with even \( n \)) and \( m \) poles. Then, up to multiplication by a nonzero real constant, the horizontal trajectory \( \Gamma(Q) \) with limiting ends at $\boldsymbol{z}$ satisfies:

\begin{itemize}
    \item (Underscreening) If \( m \leq \frac{n}{2} \), then \( \Gamma(Q) \) consists of \( m \) disjoint arcs connecting distinct pairs of zeros, forming a radial \((n,m)\)-link. For each such link pattern, there exists a unique differential \( Q \in \mathcal{QD}(\boldsymbol{z}) \) (up to scaling) whose horizontal trajectories form this pattern.

    \item (Overscreening) If \( \frac{n+1}{2} \leq m \leq n \), then \( \Gamma(Q) \) consists of \( n - m \) disjoint arcs connecting pairs of zeros, forming a radial \((n,n-m)\)-link. For each such link pattern, there exists a continuous family of differential \( Q \in \mathcal{QD}(\boldsymbol{z}) \) (up to scaling) whose horizontal trajectories form this pattern.

    \item (Upper bound) If \( m > n \), there exists no such quadratic differential \( Q \in \mathcal{QD}(\boldsymbol{z}) \).
\end{itemize}

\end{conjecture}

\begin{conjecture}[$n$ odd]\label{radial conjecture n odd quadratic differential}
Let \( Q(z)\,dz^2 \in \mathcal{QD}(\boldsymbol{z}) \) be an involution symmetric meromorphic quadratic differential with \( n \) simple zeros located on the unit circle (with odd \( n \)) and \( m \) poles. Then, up to multiplication by a nonzero real constant, the horizontal trajectory \( \Gamma(Q) \) with limiting ends at $\boldsymbol{z}$ satisfies:

\begin{itemize}
    \item (Underscreening) If \( m \leq \frac{n}{2} \), then \( \Gamma(Q) \) consists of \( m \) disjoint arcs connecting distinct pairs of zeros, forming a radial \((n,m)\)-link. For each such link pattern, there exists a unique differential \( Q \in \mathcal{QD}(\boldsymbol{z}) \) (up to scaling) whose horizontal trajectories form this pattern.

    \item (Overscreening) If \( \frac{n+1}{2} \leq m \leq n \), then \( \Gamma(Q) \) consists of \(n- m \) disjoint arcs connecting pairs of zeros, forming a radial \((n,n-m)\)-link. For each such link pattern, there exists a unique differential \( Q \in \mathcal{QD}(\boldsymbol{z}) \) (up to scaling) whose horizontal trajectories form this pattern.

    \item (Upper bound) If \( m > n \), there exists no such quadratic differential \( Q \in \mathcal{QD}(\boldsymbol{z}) \).
\end{itemize}

\end{conjecture}

\begin{remark}
    When $n$ is an even integer, in the overscreening case, $Q(z)dz^2 = R'(z)^2 dz^2$, where $R(z)$ is an involution symmetric rational function with $\boldsymbol{z}$ as critical points. In this case, the continuous family of solutions can be obtained by post-composition with M{\"o}bius transformations.
\end{remark}

We can equivalently reformulate our conjectures concerning the critical points of the master functions.

\begin{conjecture} \label{ conjugation symmetry of critical points}
    For generic $\boldsymbol{z}$ on the unit circle, critical points $\boldsymbol{\xi}$ of the master function $\Phi_{m,n}(\boldsymbol{z},\boldsymbol{\xi})$ are involution symmetric.
\end{conjecture}

\begin{conjecture}[$n$ even] \label{conjecture n even critical points}
For generic $\boldsymbol{z}$ on the unit circle:
\begin{itemize}
    \item (Underscreening) If $m \leq \frac{n}{2}$, $\Phi_{m,n}(\boldsymbol{z},\boldsymbol{\xi})$ has exactly $|LP(n,m)|$ isolated critical points.
    \item (Overscreening) If $\frac{n+1}{2} \leq m \leq n$,  $\Phi_{m,n}(\boldsymbol{z},\boldsymbol{\xi})$ has non-isolated critical points.
    
    Let $\lambda_1=\sum \xi_i, \lambda_2=\sum \xi_i \xi_j, \ldots, \lambda_m=\xi_1 \cdots \xi_m$ be the standard symmetric functions of $\xi_1, \ldots, \xi_m$. Denote $\mathbb{C}_\lambda^m$ the space with coordinates $\lambda_1, \ldots, \lambda_m$.
    Then written in symmetric coordinates $\lambda_1, \ldots, \lambda_m$, the critical points consist of $|LP(n,n-m)|$ straight lines in the space $\mathbb{C}_\lambda^m$.
    \item (Upperbound) If $m >n$, $\Phi_{m,n}(\boldsymbol{z},\boldsymbol{\xi})$ has no critical points.
\end{itemize}

\end{conjecture}

\begin{conjecture}[$n$ odd]\label{conjecture n odd critical points}
For $\boldsymbol{\xi}$ and generic $\boldsymbol{z}$ on the unit circle:
\begin{itemize}
    \item (Underscreening) If $m \leq \frac{n}{2}$, $\Phi_{m,n}(\boldsymbol{z},\boldsymbol{\xi})$ has  $|LP(n,m)|$ isolated critical points.
    \item (Overscreening) If $\frac{n+1}{2} \leq m \leq n$,   $\Phi_{m,n}(\boldsymbol{z},\boldsymbol{\xi})$ has $|LP(n,n-m)|$ isolated critical points.
    \item (Upperbound) If $m >n$, $\Phi_{m,n}(\boldsymbol{z},\boldsymbol{\xi})$ has no critical points.
\end{itemize}

\end{conjecture}

\subsection{Examples: underscreening}
\label{underscreening}
In this section, we provide a series of figures to illustrate the trace configurations arising from various multiple radial $\mathrm{SLE}(0)$ systems.

For multiple radial SLE(0) system with growth points $\boldsymbol{z}$ and screening charges $\boldsymbol{\xi}$,
the corresponding quadratic differential is given by:
$$Q(z)dz^2= \frac{\prod_{j=1}^{m}\xi_{k}^{2}}{\prod_{k=1}^{n}z_{k}}
z^{2m-n-2}\frac{\prod_{k=1}^{n}(z-z_k)^2}{\prod_{j=1}^{m}(z-\xi_j)^4}dz^2$$

\begin{thm}
Given $Q(z) \in \mathcal{QD}(\boldsymbol{z})$ associate to it a vector field $v_Q$ on $\widehat{\mathbb{C}}$ defined by
\begin{equation}
 v_Q(z)=\frac{1}{\sqrt{Q(z)}}  
\end{equation}    
where $$
\sqrt{Q(z)}=\frac{\prod_{j=1}^{m}\xi_{k}}{\sqrt{\prod_{k=1}^{n}z_{k}}}
z^{m-\frac{n}{2}-1}\frac{\prod_{k=1}^{n}(z-z_k)}{\prod_{j=1}^{m}(z-\xi_j)^2}
$$
 The flow lines of $\dot{z} =  v_Q(z)$ are the horizontal trajectories of $Q(z)dz^2$.
\end{thm}

\begin{remark}
    This lemma provides an elementary way to plot the horizontal trajectories of $Q(z)dz^2$
\end{remark}

In the following figures, The zeros are marked red, the poles are marked yellow, and the marked point $u=0$ is marked green.

\begin{figure}[ht]
\begin{minipage}[t]{0.43\linewidth}
    \centering
    \includegraphics[width=6cm]{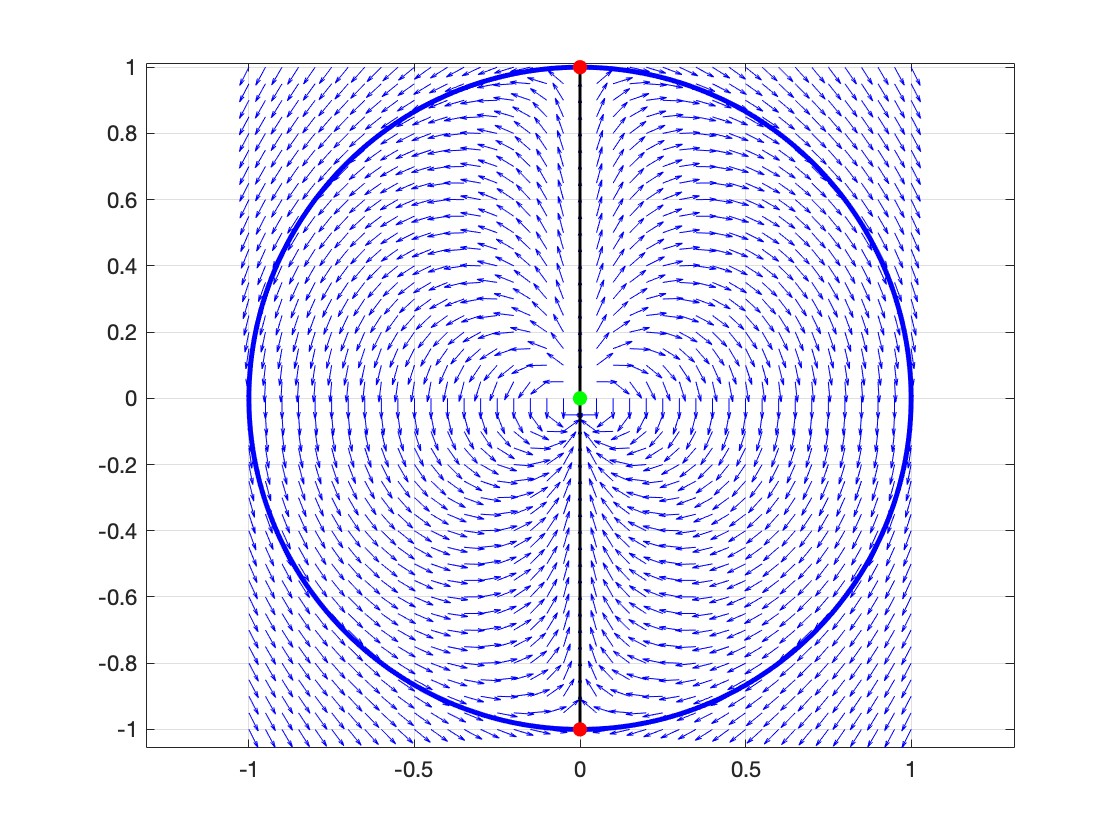}
    \caption{$x_1=i,x_2=-i$}
 
\end{minipage}
\begin{minipage}[t]{0.43\linewidth}
    \centering
	\includegraphics[width=6cm]{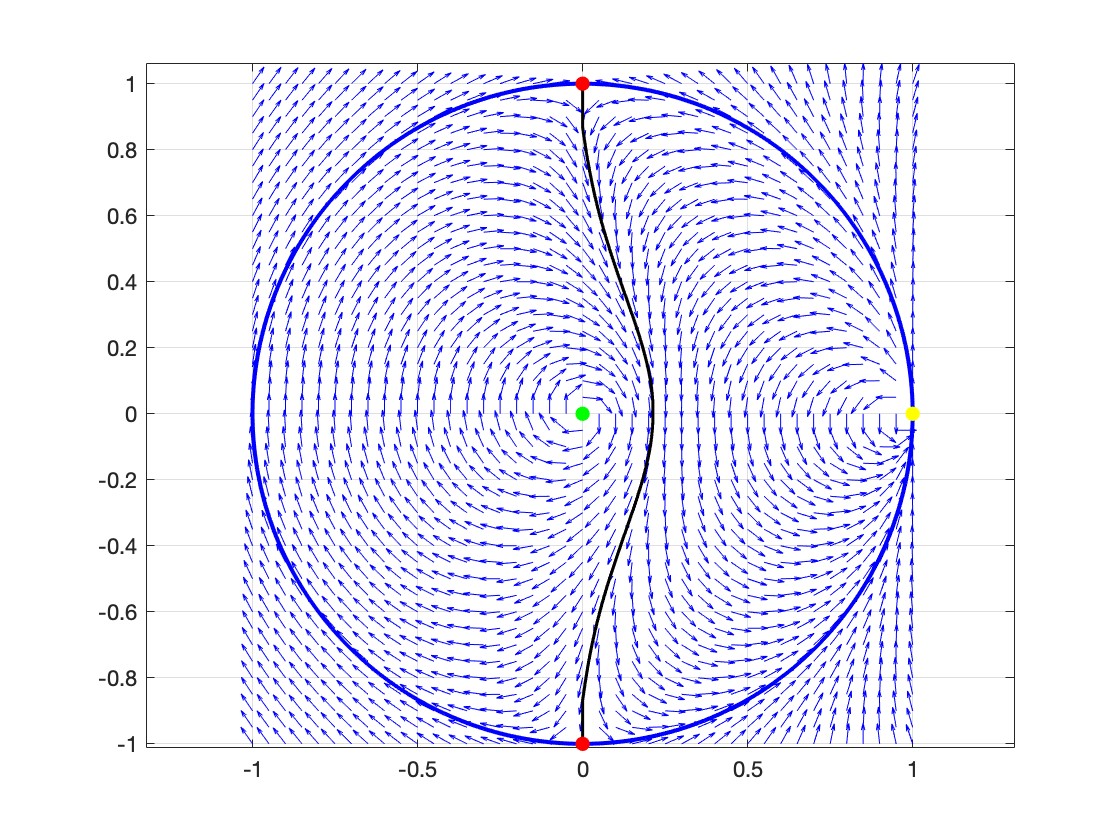}
    \caption{$x_1=i,x_2=-i,\xi_1=-1$}
    
\end{minipage}
\end{figure}
In Figure 4.1, $n=2$,$m=1$,$z_1=i$,$z_2=-i$, the SLE(0) curve connects $z_1$ to $0$ and $z_2$ to $0$.
$$\sqrt{Q(z)} = iz^{-2}(z-i)(z+i) $$

In Figure 4.2, $n=2$,$m=1$,$z_1=i$,$z_2=-i$,$\xi=1$ the SLE(0) curve connects $z_1$ and $z_2$, the curve does not surround $0$.
$$\sqrt{Q(z)} = iz^{-1}\frac{(z-i)(z+i)}{(z-1)^2} $$

\begin{figure}[ht]
\begin{minipage}[t]{0.43\linewidth}
    \centering
    \includegraphics[width=6cm]{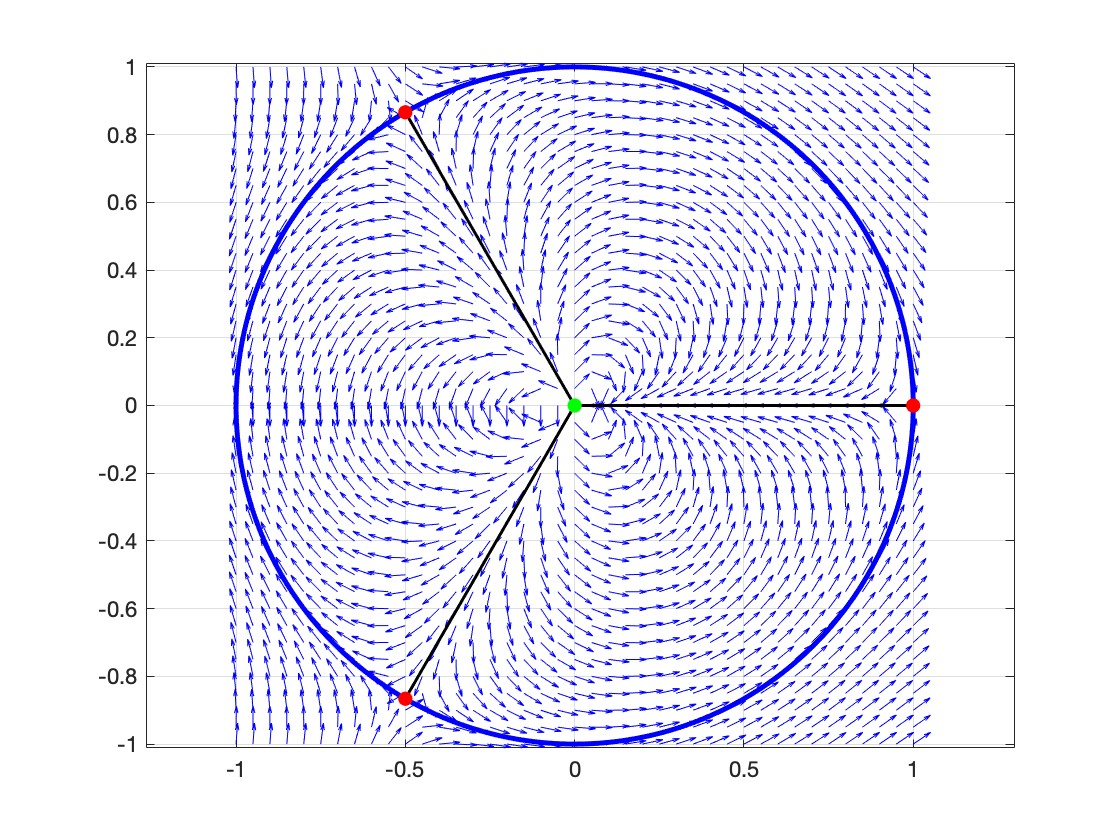}
    \caption{$x_1=1,x_2=e^{\frac{2\pi i}{3}},x_3=e^{\frac{-2\pi i}{3}}$}
    \label{30}
\end{minipage}
\begin{minipage}[t]{0.43\linewidth}
    \centering
	\includegraphics[width=6cm]{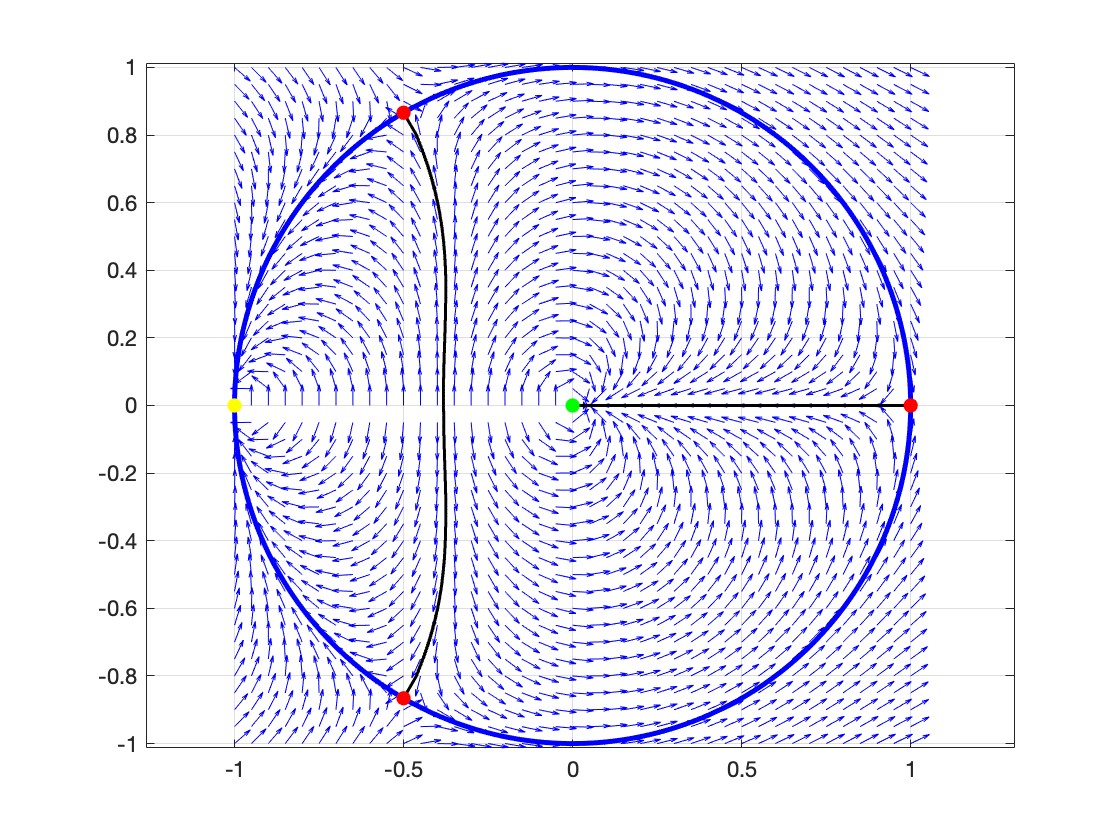}
    \caption{$x_1=1,x_2=e^{\frac{2\pi i}{3}},x_3=e^{\frac{-2\pi i}{3}},\xi_1=-1$}
    \label{31a}
\end{minipage}
\end{figure}
In Figure 4.3, $n=3$,$m=0$,$z_1=i$,$z_2=e^{\frac{2\pi i}{3}}$,$z_3=e^{\frac{4\pi i}{3}}$ the SLE(0) curves connect $z_1$ to $0$, $z_2$ to $0$ and $z_3$ to $0$.

$$\sqrt{Q(z)} = iz^{-2.5}(z-1)(z-e^{\frac{2\pi i}{3}})(z-e^{\frac{4\pi i}{3}}) $$

In Figure 4.4, $n=3$,$m=1$,$z_1=1$,$z_2=e^{\frac{2\pi i}{3}}$,$z_3=e^{\frac{4\pi i}{3}}$, $\xi=-1$ the SLE(0) curves connect $z_2$ and $z_3$ and connect $z_1$ and $0$.
$$\sqrt{Q(z)} = iz^{-\frac{3}{2}}\frac{(z-1)(z-e^{\frac{2\pi i}{3}})(z-e^{\frac{4\pi i}{3}})}{(z-1)^2} $$

\begin{figure}[ht]
\begin{minipage}[t]{0.43\linewidth}
    \centering
    \includegraphics[width=6cm]{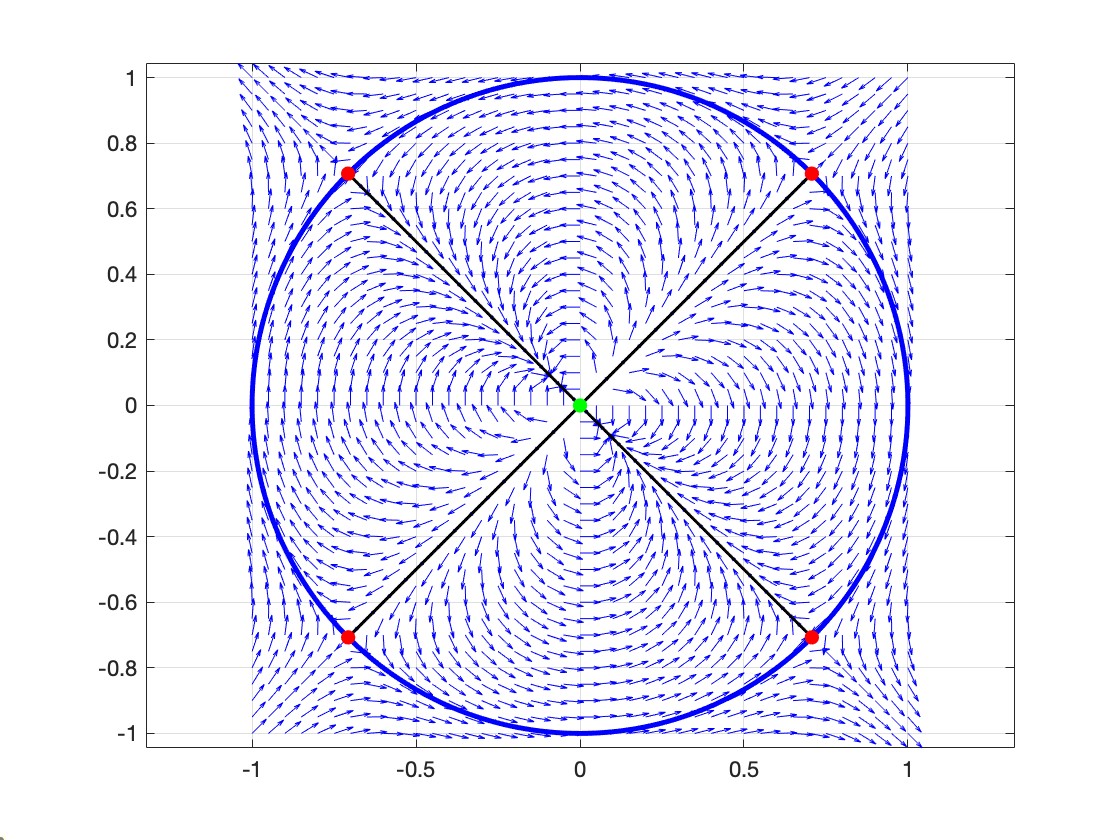}
    \caption{ $x_k=e^{\frac{(2k+1)\pi i}{4}}, k=1,2,3,4$}
 
\end{minipage}
\begin{minipage}[t]{0.43\linewidth}
    \centering
	\includegraphics[width=6cm]{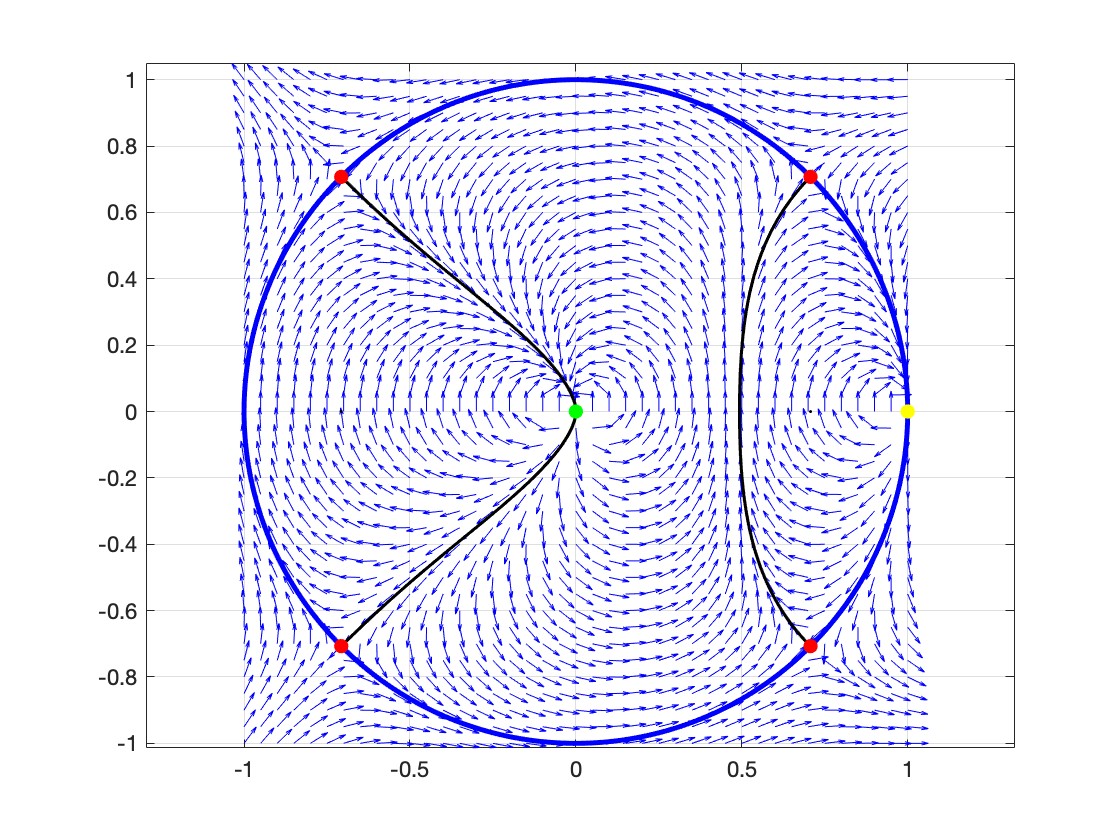}
    \caption{$x_k=e^{\frac{(2k+1)\pi i}{4}}, k=1,2,3,4, \xi_1=1$}
   
\end{minipage}
\end{figure}
In Figure 4.5, $n=4$,$m=0$,$z_k=e^{\frac{(2k+1)\pi i}{4}}$,$k=1,2,3,4$. The SLE(0) curves connect $z_k$ to $0$.

$$\sqrt{Q(z)} = iz^{-3}(z-e^{\pi i/4})(z-e^{3\pi i/4})(z-e^{5 \pi i/4})(z-e^{7\pi i/4})$$

In Figure 4.6, $n=4$,$m=1$,$z_k=e^{\frac{(2k+1)\pi i}{4}}$,$k=1,2,3,4$, $\xi=1$ the SLE(0) curves connect $z_3$ and $z_4$ to $0$ and connect $z_1$ to $z_2$.
$$\sqrt{Q(z)} = -iz^{-2}\frac{(z-e^{\pi i/4})(z-e^{3\pi i/4})(z-e^{5 \pi i/4})(z-e^{7\pi i/4})}{(z-1)^2} $$

\begin{figure}[ht]
\begin{minipage}[h]{0.43\linewidth}
    \centering
    \includegraphics[width=6cm]{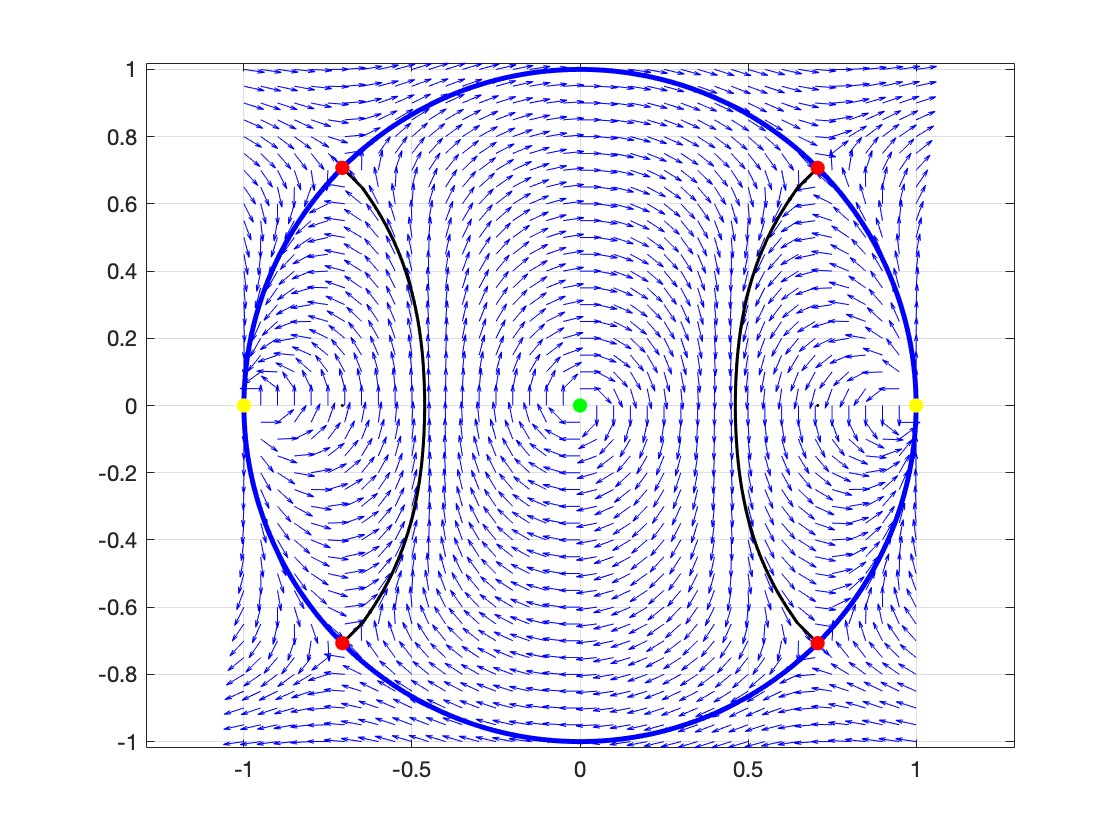}
    \caption{$x_k=e^{\frac{k\pi i}{4}}, k=1,2,3,4$, $\xi_1=1$,$\xi_2=-1$}
 
\end{minipage}
\begin{minipage}[h]{0.43\linewidth}
    \centering
	\includegraphics[width=6cm]{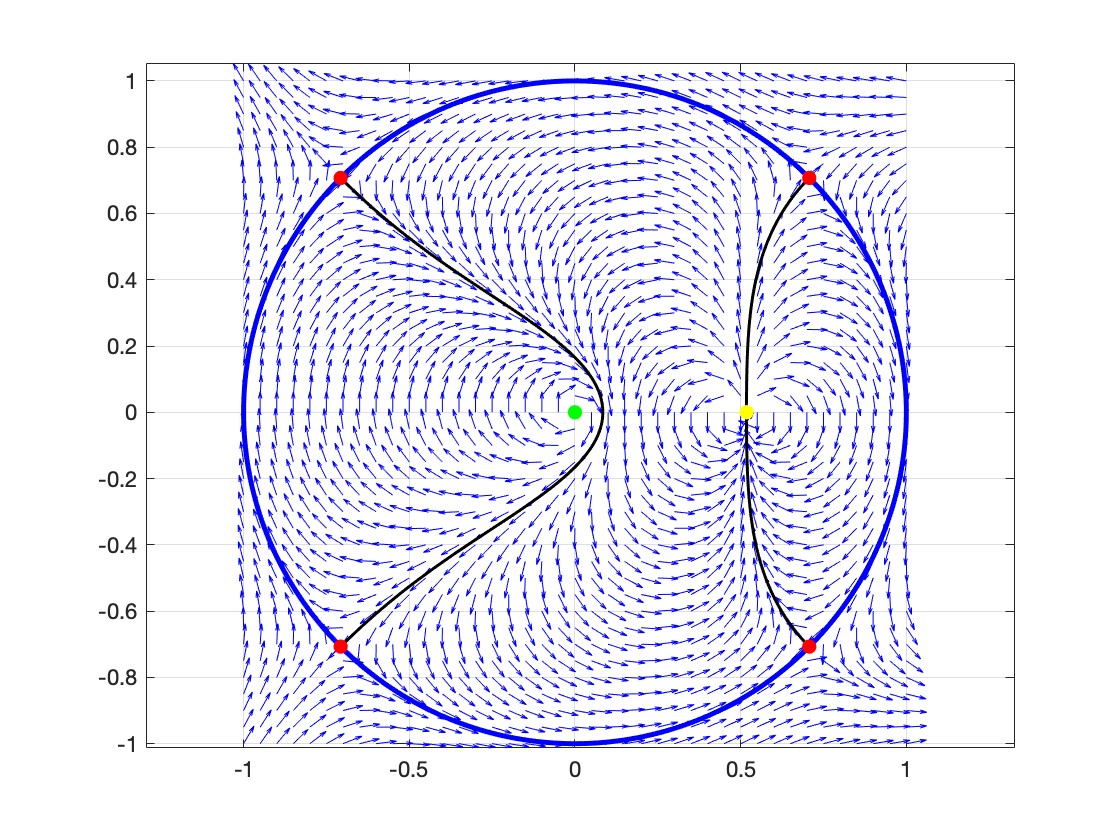}
    \caption{$x_k=e^{\frac{(2k+1)\pi i}{4}}$,$k=0,1,2,3, \xi_1=-1$}

\end{minipage}
\end{figure}

In Figure 4.7, $n=4$,$m=2$,$z_k=e^{\frac{(2k+1)\pi i}{4}}$,$k=1,2,3,4$,$\xi_1=-1$,$\xi_2=1$. The SLE(0) curves connect $z_1$ and $z_4$, $z_2$ and $z_3$.

$$\sqrt{Q(z)} = iz^{-1}\frac{(z-e^{\pi i/4})(z-e^{3\pi i/4})(z-e^{-\pi i/4})(z-e^{-3 \pi i/4})}{(z-1)^2(z+1)^2} $$

In Figure 4.8, $n=4$,$m=2$,$z_i=e^{\frac{(2k+1)\pi i}{4}}$,$k=1,2,3,4$, $\xi_1=\sqrt{2 - \sqrt{3}}$, $\xi_2=\sqrt{2 + \sqrt{3}}$ the SLE(0) curve connects $z_3$ and $z_4$ to $0$ and connects $z_1$ to $z_2$.
$$\sqrt{Q(z)} = iz^{-1}\frac{(z-e^{\pi i/4})(z-e^{3\pi i/4})(z-e^{-\pi i/4})(z-e^{-3 \pi i/4})}{(z-\sqrt{2 - \sqrt{3}})^2(z-\sqrt{2 + \sqrt{3}})^2} $$

\subsection{Examples: overscreening}
\label{overscreening}
Let us recall the definition of the trace quadratic differential \begin{defn} 
[An equivalence class of quadratic differentials with prescribed zeros] 
  Let $\boldsymbol{z}=\{z_1,z_2,\ldots,z_n \}$ be distinct points on the unit circle,\
 a class of quadratic differentials with prescribed zeros denoted by $\mathcal{QD}(\boldsymbol{z})$: 
  \begin{itemize}
 
      \item[{\rm(1)}] involution symmetric:
      $\overline{Q(z^*)}\overline{(dz^{*})^2}= Q(z)dz^2$, where $z^*=\frac{1}{\bar{z}}$
       \item[{\rm (2)}] zeros of order 2 at $\{z_1,z_2,\ldots,z_n \}$
      \item[{\rm (3)}] $\{\xi_1,\ldots,\xi_m\}$ are poles of order 4  and $Res_{\xi_j}(\sqrt{Q}dz)=0$, $j=1,\ldots,m$. (\textcolor{red}{Residue-free})
      \item[{\rm (4)}] poles of order $n+2-2m$ at marked points $0$ and $\infty$

  \end{itemize}

\end{defn}

Note that when $m > \frac{n}{2}+1$, the poles at $0$ and $\infty$ are in fact zeros. $m=\frac{n}{2}+1$ is a threshold for screening.

\begin{figure}[ht]
\begin{minipage}[t]{0.43\linewidth}
    \centering
    \includegraphics[width=6cm]{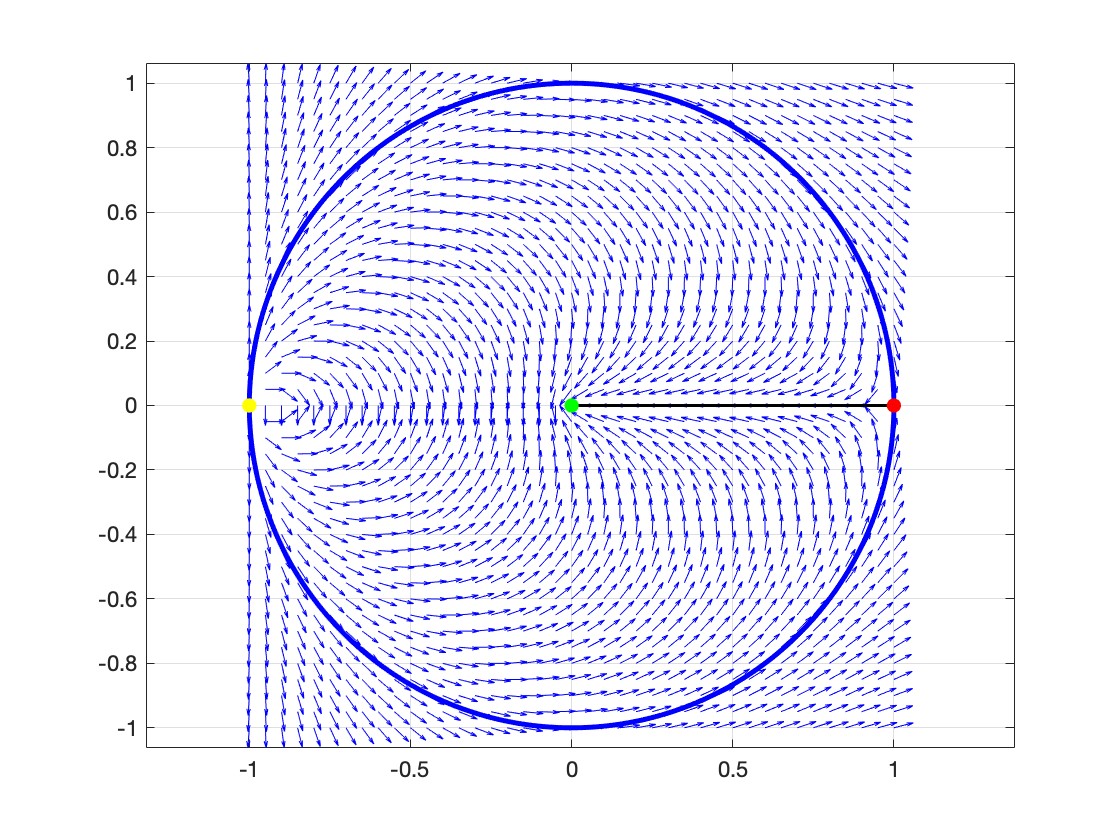}
    \caption{$x_1=1,\xi_1=-1$}
    \label{11}
\end{minipage}
\begin{minipage}[t]{0.43\linewidth}
    \centering
    \includegraphics[width=6cm]{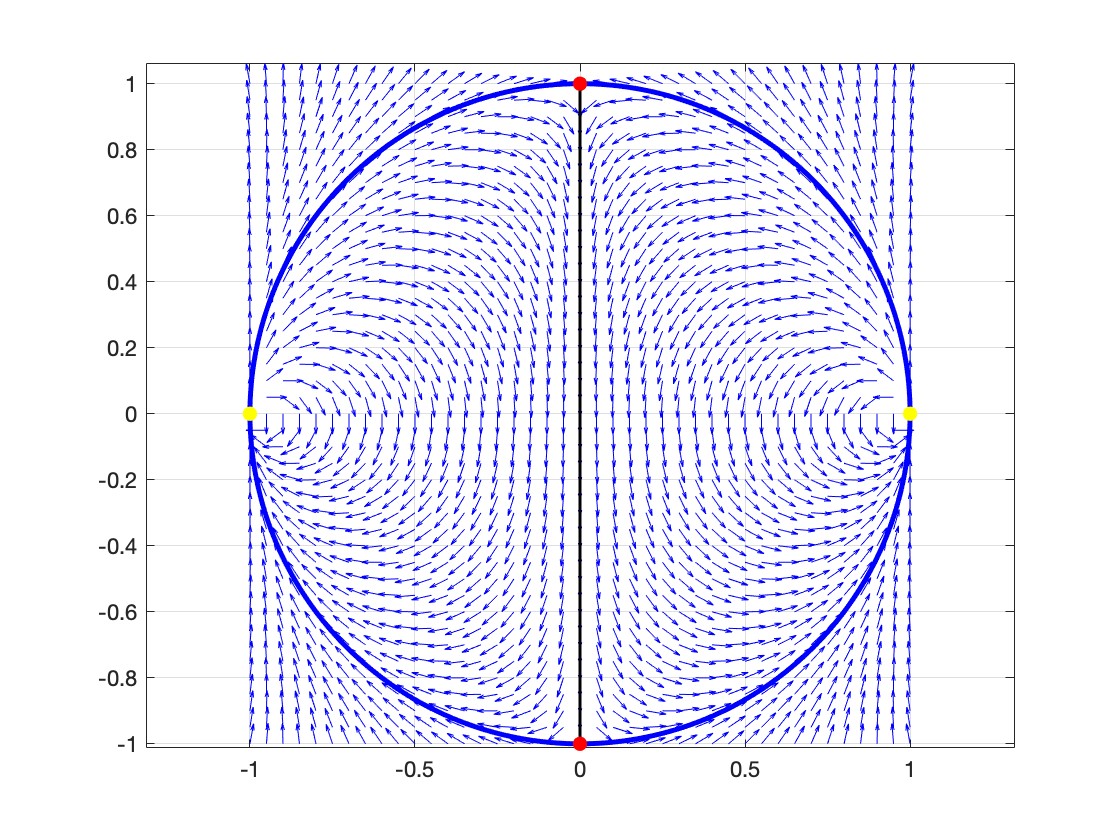}
    \caption{$x_1=i,x_2=-i,\xi_1=1,\xi_2=-1$}
    \label{2c}
\end{minipage}

\end{figure}
In Figure 4.9, $n=1$,$m=1$,$z_1=1$,$\xi_1=-1$. The SLE(0) curve connects $z_1$ and $0$.

$$\sqrt{Q(z)} = z^{-\frac{1}{2}}\frac{(z-1)}{(z-i)^2}$$

In Figure 4.10, $n=2$,$m=2$,$z_1=-i$,$z_2=i$,$\xi_1=-1$,$\xi_2=1$. The SLE(0) curve connects $z_1$ and $z_2$.

$$\sqrt{Q(z)} = i\frac{(z-i)(z+i)}{(z-1)^2(z+1)^2}$$

\begin{figure}[ht]
\begin{minipage}[t]{0.43\linewidth}
    \centering
	\includegraphics[width=6cm]{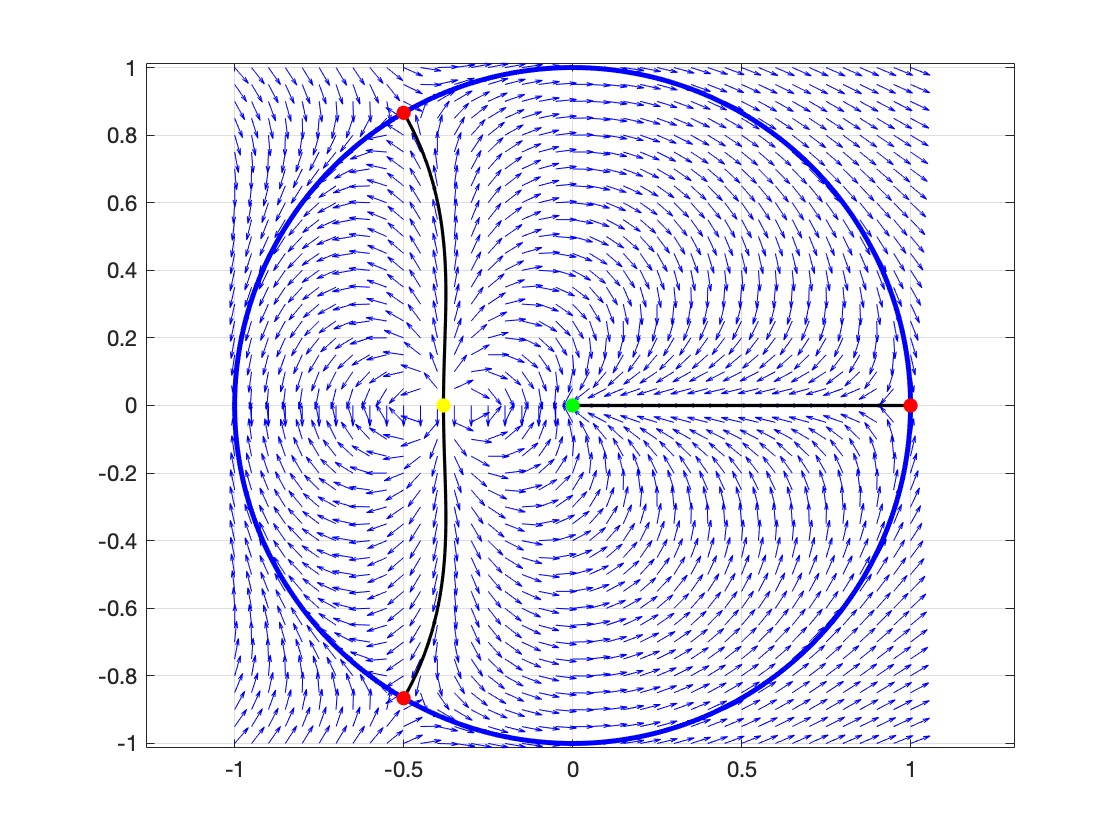}
    \caption{$x_k=e^{\frac{2k\pi i}{3}},k=0,1,2$,$\xi_1=-\frac{3}{2}+\frac{\sqrt{5}}{2},\xi_2=-\frac{3}{2}-\frac{\sqrt{5}}{2}$}
    \label{32a}
\end{minipage}
\begin{minipage}[t]{0.43\linewidth}
    \centering
	\includegraphics[width=6cm]{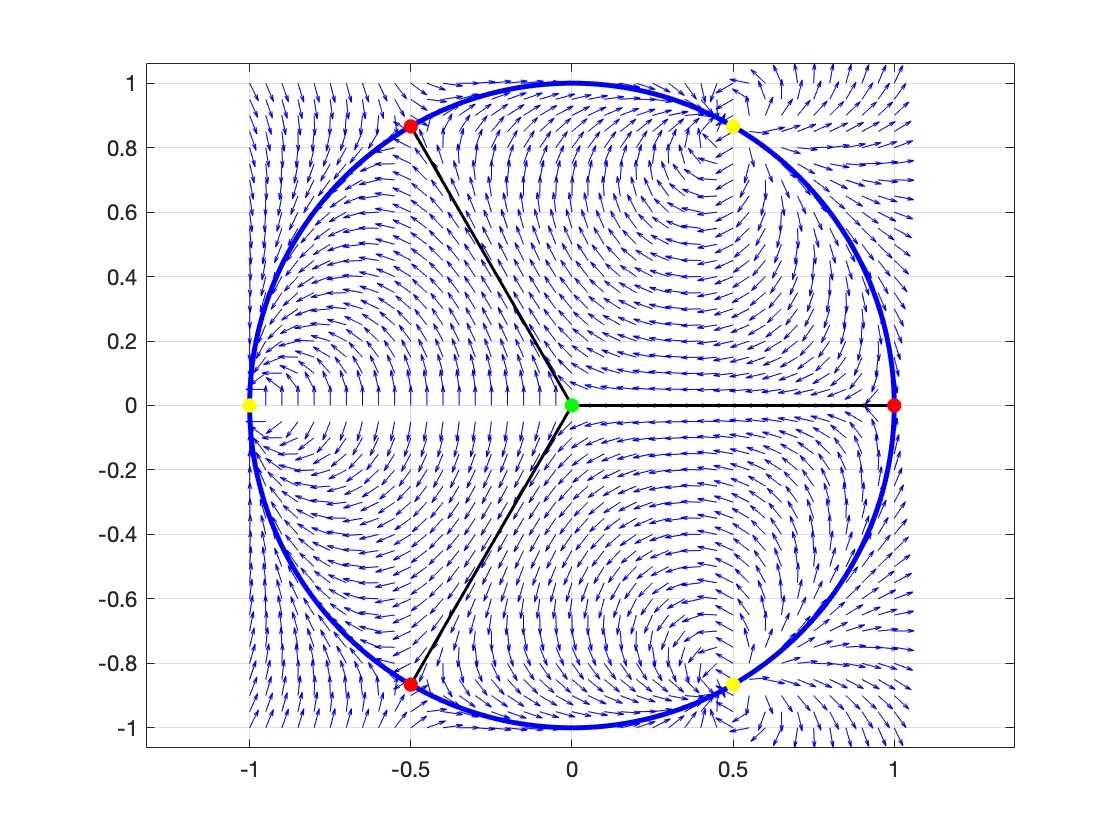}
    \caption{$x_k=e^{\frac{2k\pi i}{3}},\xi_k=e^{\frac{(2k-1)\pi i}{3}},k=0,1,2$}
    \label{33}
\end{minipage}
\end{figure}

In Figure 4.11, $n=3$,$m=2$,$z_k=e^{\frac{k\pi i}{4}}$,$k=0,1,2$, $\xi_1=-\frac{3}{2}+\frac{\sqrt{5}}{2}$,$\xi_2=-\frac{3}{2}-\frac{\sqrt{5}}{2}$ the SLE(0) curve connects $z_3$ and $z_4$ to $0$ and connects $z_1$ to $z_2$.
$$\sqrt{Q(z)} = z^{-1/2}\frac{(z-1)(z-e^{2\pi i/3})(z-e^{4 \pi i/3})}{(z+\frac{3}{2}-\frac{\sqrt{5}}{2})^2(z+\frac{3}{2}+\frac{\sqrt{5}}{2})^2} $$

 In Figure 4.12, $n=3$,$m=3$,$z_k=e^{\frac{2k\pi i}{3}}$,$\xi_k=e^{\frac{(2k-1)\pi i}{3}}$,$k=0,1,2$, the SLE(0) curve connects $z_3$ and $z_4$ to $0$ and connects $z_1$ to $z_2$.
$$\sqrt{Q(z)} = z^{1/2}\frac{(z-1)(z-e^{2 \pi i/3})(z-e^{4\pi i/3})}{(z-e^{\pi i/3})^2(z-e^{3\pi i/3})^2(z-e^{5\pi i/3})^2} $$

\begin{figure}[ht]
\begin{minipage}[t]{0.43\linewidth}
\includegraphics[width=6cm]{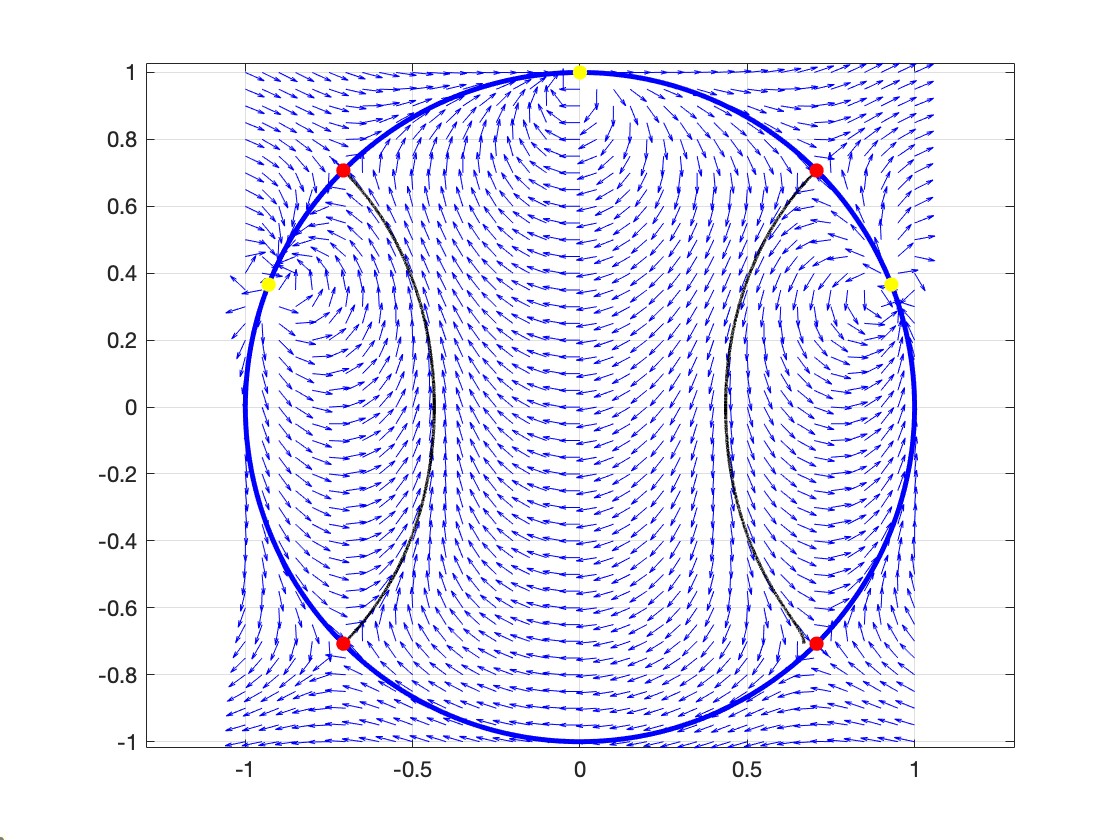}
    \caption{$x_k=e^{\frac{(2k+1)\pi i}{4}},k=0,1,2,3,$ $\xi_1 = i$, $\xi_2=-\frac{\sqrt[4]{3}}{\sqrt{2}} + \frac{-1 + \sqrt{3}}{2}i$, $\xi_3=\frac{\sqrt[4]{3}}{\sqrt{2}} + \frac{-1 + \sqrt{3}}{2}i$}
    \label{fig:43a}
\end{minipage}  
\begin{minipage}[t]{0.43\linewidth}
    \includegraphics[width=6cm]{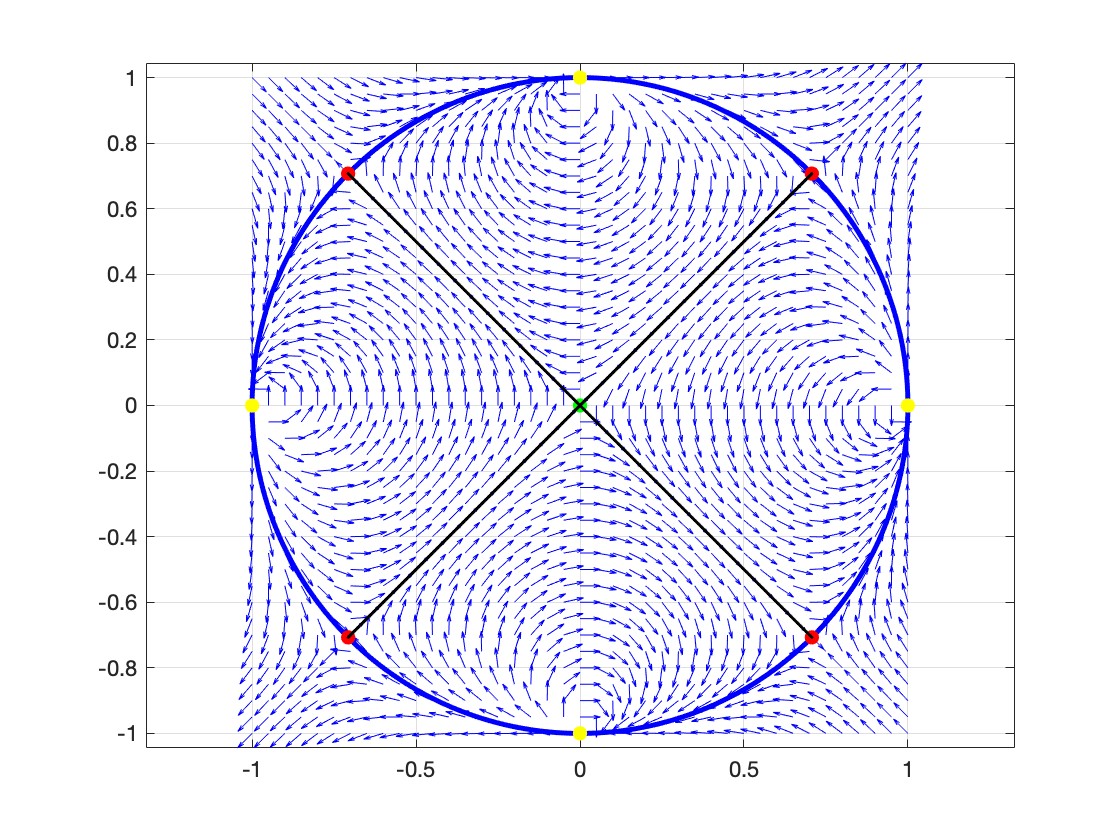}
    \caption{$x_k=e^{\frac{(2k+1)\pi i}{4}}$,$\xi_k = e^{i\frac{k\pi i}{4}}$,$k=0,1,2,3$}
    \label{fig:43a1}
\end{minipage}

\end{figure}
In Figure 4.13, $n=4$,$m=3$,$x_k=e^{\frac{(2k+1)\pi i}{4}},k=0,1,2,3,$ $\xi_1 = i$, $\xi_2=-\frac{\sqrt[4]{3}}{\sqrt{2}} + \frac{-1 + \sqrt{3}}{2}i$, $\xi_3=\frac{\sqrt[4]{3}}{\sqrt{2}} + \frac{-1 + \sqrt{3}}{2}i$. The SLE(0) curves connects $z_1$ and $z_4$, $z_2$ and $z_3$.

$$\sqrt{Q(z)} = \frac{ z^4+1}{(z-i)^2(z-\frac{\sqrt[4]{3}}{\sqrt{2}} - \frac{(1 - \sqrt{3}) i}{2})^2(z+\frac{\sqrt[4]{3}}{\sqrt{2}}- \frac{(1 - \sqrt{3}) i}{2})^2}$$

In Figure 4.14, $n=4$,$m=3$,$z_k=e^{\frac{(2k+1)\pi i}{4}}$, $\xi=e^{\frac{k\pi i}{2}}$,$k=0,1,2,3$. the SLE(0) curves connect $z_k$,$k=0,1,2,3$ to $0$.
$$\sqrt{Q(z)} = z\frac{z^4+1}{z^4-1} $$

\section{Multiple radial SLE(0) system with spin}

\subsection{Residue-free quadratic differentials with prescribed zeros}

\begin{defn}[Quadratic differentials with prescribed zeros and spin] \label{trace quadratic differential spin}
Let \( \boldsymbol{\theta} = \{\theta_1, \theta_2, \ldots, \theta_n \} \) be distinct points on the unit circle \( \partial \mathbb{D} \). We define \( \mathcal{QD}(\boldsymbol{\theta}) \) to be the class of meromorphic quadratic differentials on \( \mathbb{C}  \) of the form
\[
Q(\theta) \, d\theta^2 = 
\frac{\prod_{k=1}^{m} e^{2i\zeta_{k}}}{\prod_{j=1}^{n} e^{i\theta_j}} 
\, e^{i(2m - n) \theta}
\, \frac{\prod_{j=1}^{n} \left(e^{i\theta} - e^{i\theta_j}\right)^2}{\prod_{k=1}^{m} \left(e^{i\theta} - e^{i\zeta_k}\right)^4}
\, d\theta^2,
\]
satisfying the following conditions:
\begin{enumerate}
    \item symmetric under the involution \( \theta^* =\bar{\theta} \), meaning
    \[
    \overline{Q(\theta^*)}\overline{(d\theta^*)^2} = Q(\theta)d\theta^2.
    \]

    \item  distinct zeros at \( \{\theta_1, \theta_2, \ldots, \theta_n\} \), each of order 2.

    \item  distinct finite poles at \( \{\zeta_1, \ldots, \zeta_m\} \), each of order 4, and the residues vanish (Residue-free condition):
    \[
    \text{Res}_{\zeta_j}(\sqrt{Q(\theta)}d\theta) = 0, \quad \text{for } j = 1, \ldots, m.
    \]

\end{enumerate}

Here, the poles \( \{\zeta_1, \ldots, \zeta_m\} \) are finite, meaning they do not coincide with \( \infty \). 

\end{defn}
\begin{thm}[Traces as horizontal trajectories in angular coordinates] \label{traces as horizontal trajectories angular}
Let \( \boldsymbol{\theta} = \{\theta_1, \theta_2, \ldots, \theta_n\} \) be distinct angular coordinates on the unit circle, i.e., \( z_j = e^{i\theta_j} \in \partial \mathbb{D} \), and let \( \boldsymbol{\zeta} = \{\zeta_1, \zeta_2, \ldots, \zeta_m\} \) be positions of poles satisfying the conjudgation symmetry \( \zeta_k^* = -\zeta_k \) and the stationary relations.

Then there exists a quadratic differential \( Q(\theta)\,d\theta^2 \in \mathcal{QD}(\boldsymbol{\theta}) \), with double zeros at \( \theta_1, \ldots, \theta_n \) and poles of order 4 at \( \zeta_1, \ldots, \zeta_m \), such that the hulls \( K_t \) generated by the multiple radial Loewner flow with driving functions \( \boldsymbol{\theta}(t) \) and screening charges \( \boldsymbol{\zeta}(t) \) are a subset of the horizontal trajectories of \( Q(\theta)\,d\theta^2 \) whose limiting ends are at \( \boldsymbol{\theta} \), up to any time \( t \) prior to a collision among poles and zeros.

Moreover, for such times \( t \),
\[
Q(\boldsymbol{\theta}) \circ h_t^{-1} \in \mathcal{QD}(\boldsymbol{\theta}(t)),
\]
where \( h_t \) denotes the covering map associated with the Loewner evolution, and \( \boldsymbol{\theta}(t) \) are the angles of the time-evolved growth points.
\end{thm}
\begin{proof}
The proof proceeds by adapting the argument used in Theorem (\ref{angular integral of motion}), now expressed entirely in angular coordinates. Specifically, we apply the angular version of the integral of motion (see Corollary~\ref{angular integral of motion}) to show that the time-evolved hulls \( K_t \) remain embedded in the horizontal trajectories of a quadratic differential \( Q(\theta) d\theta^2 \in \mathcal{QD}(\boldsymbol{\theta}) \), as claimed.
\end{proof}

\subsection{Field integral of motion and horizontal trajectories as flow lines }
\label{field integral of motion with spin}
In this section, we generalize the integral of motion for multiple radial SLE(0) systems to the case where the spin $\eta$ is non-zero.

We begin by considering the following integral of motion $N_t(z)$:  let $z_1,z_2,\ldots,z_n$ be distinct points on the unit circle, and $z\in \overline{\mathbb{D}}$. Let
$$N_t(z)=e^{-(m-\frac{n}{2})(\int_{0}^{t}\sum_{j}\nu_j(s)ds)}g_t(z)^{m-\frac{n}{2}-1-\frac{\eta i}{2}}g'_{t}(z)\frac{\prod_{k=1}^{n}(g_t(z)-z_k(t))}{\prod_{j=1}^{m}(g_t(z)-\xi_j(t))^2}.$$ However, the term $g_t(z)^{m-\frac{n}{2}-1-\frac{\eta i}{2}}$ is multivalued and $N_t(z)$ is in fact not well-defined. To resolve this technical issue, we will write this expression in angular coordinates.

\begin{thm}
In angular coordinates, let \( \xi_k = e^{i\zeta_k} \), \( z_k = e^{i\theta_k} \), and let \( h_t(z) \) be the covering map of the radial Loewner flow \( g_t(z) \), i.e.,(
$e^{i h_t(z)} = g_t(e^{i z})$.)
Then for each \( z \in \overline{\mathbb{H}} \), define the observable
\begin{equation}
\label{eq:Nang}
N_t^{\mathrm{ang}}(z) = e^{-(m - \frac{n}{2}) t}
\cdot \frac{\prod_{j=1}^{m} e^{i \zeta_j(t)}}{\prod_{k=1}^{n} e^{i \frac{\theta_k(t)}{2}}}
\cdot e^{i(m - \frac{n}{2} - 1) h_t(z)} \cdot e^{\frac{\eta}{2} h_t(z)} \cdot h_t'(z)
\cdot \frac{\prod_{k=1}^{n} (e^{i h_t(z)} - e^{i \theta_k(t)})}{\prod_{j=1}^{m} (e^{i h_t(z)} - e^{i \zeta_j(t)})^2}
\cdot e^{-i z}.
\end{equation}
Then \( N_t^{\mathrm{ang}}(z) \) is an integral of motion on the time interval \( [0, \tau_z \wedge \tau) \), where \( \tau \) is the first collision time of any poles or critical points, and \( \tau_z \) is the swallowing time of the point \( z \) under the multiple radial Loewner flow with parametrization \( \nu_j(t) = 1 \), \( \nu_k(t) = 0 \) for \( k \neq j \).
\end{thm}

\begin{proof}
The expression \( N_t^{\mathrm{ang}}(z) \) can be factorized as the product of a part depending only on time,
\[
A^{\mathrm{ang}}(t) = \frac{\prod_{j=1}^{m} e^{i \zeta_j(t)}}{\prod_{k=1}^{n} e^{i \frac{\theta_k(t)}{2}}}
\cdot e^{-(m - \frac{n}{2})t},
\]
and a part depending on \( z \),
\[
B^{\mathrm{ang}}_t(z) = e^{i(m - \frac{n}{2} - 1) h_t(z)} \cdot e^{\frac{\eta}{2} h_t(z)} \cdot h_t'(z)
\cdot \frac{\prod_{k=1}^{n} (e^{i h_t(z)} - e^{i \theta_k(t)})}{\prod_{j=1}^{m} (e^{i h_t(z)} - e^{i \zeta_j(t)})^2}
\cdot e^{-i z}.
\]
By direct computation,
\[
\frac{d}{dt} \log A^{\mathrm{ang}}(t) = -\frac{i \eta}{2}, \quad 
\frac{d}{dt} \log B^{\mathrm{ang}}_t(z) = \frac{i \eta}{2}.
\]
These terms cancel, and hence
\[
\frac{d}{dt} \log N_t^{\mathrm{ang}}(z)
= \frac{d}{dt} \log A^{\mathrm{ang}}(t) + \frac{d}{dt} \log B^{\mathrm{ang}}_t(z) = 0.
\]
Therefore, \( N_t^{\mathrm{ang}}(z) \) is conserved under the flow.
\end{proof}

\begin{thm}
In angular coordinates, define \( \xi_k = e^{i\zeta_k} \), \( z_k = e^{i\theta_k} \), and let \( h_t(z) \) be the covering map of the Loewner flow \( g_t(z) \), i.e., \( e^{i h_t(z)} = g_t(e^{i z}) \). For any \( z \in \overline{\mathbb{H}} \), define:
\begin{align}
A^{\mathrm{ang}}(t) &= \frac{\prod_{j=1}^{m} e^{i\zeta_j(t)}}{\prod_{k=1}^{n} e^{i\frac{\theta_k(t)}{2}}}, \\
B^{\mathrm{ang}}_t(z) &= e^{-(2m - n)\int_0^t \sum_j \nu_j(s) ds} \cdot g_t(z)^{2m - n - 2} \cdot e^{i(m - \frac{n}{2} - 1 + \frac{\eta}{2}) h_t(z)} \cdot h_t'(z) \cdot e^{i h_t(z)} \notag \\
&\qquad \cdot \frac{\prod_{k=1}^{n} \left( e^{i h_t(z)} - e^{i \theta_k(t)} \right)}{\prod_{j=1}^{m} \left( e^{i h_t(z)} - e^{i \zeta_j(t)} \right)^2}, \\
N^{\mathrm{ang}}_t(z) &= A^{\mathrm{ang}}(t) \cdot B^{\mathrm{ang}}_t(z).
\end{align}

Then \( N^{\mathrm{ang}}_t(z) \) defines a field integral of motion for the multiple radial SLE(0) Loewner flows with driving weights \( \nu_j(t) \), on the interval \( [0, \tau_t \wedge \tau) \), where \( \tau \) is the first collision time among the poles or driving points.
\end{thm}

\begin{proof}
The computation is a deformation of the zero-spin case (\( \eta = 0 \)), with the additional spin term contributing to the angular prefactor. By direct differentiation:
\begin{align}
\frac{d}{dt} \log A^{\mathrm{ang}}(t) &= -\frac{i \eta}{2} \sum_{j=1}^{n} \nu_j(t), \\
\frac{d}{dt} \log B^{\mathrm{ang}}_t(z) &= \frac{i \eta}{2} \sum_{j=1}^{n} \nu_j(t), \\
\frac{d}{dt} \log N^{\mathrm{ang}}_t(z) &= \frac{d}{dt} \log A^{\mathrm{ang}}(t) + \frac{d}{dt} \log B^{\mathrm{ang}}_t(z) = 0.
\end{align}
Hence, \( N^{\mathrm{ang}}_t(z) \) is preserved under the flow and is therefore a field integral of motion.
\end{proof}

\subsection{Classical limit of martingale observables* } \label{Martingale Observable}
In this section, we discuss how the field integral of motion is heuristically derived as the classical limit of martingale observables constructed as the correlation functions of conformal fields.

Based on the SLE-CFT correspondence, we can couple the multiple radial SLE($\kappa$) system to a conformal field theory constructed via vertex operators, following the approach outlined in  \cite{KM13,KM21} 

\begin{defn}[$n$-leg operator with screening charges]
    Consider the following charge distribution on the Riemann sphere

$$
\boldsymbol{\beta}=b \delta_{0}+b\delta_{\infty}
$$
\begin{equation}
\boldsymbol{\tau_1}=\sum_{j=1}^{n} a \delta_{z_j}-\sum_{k=1}^m 2 a \delta_{\xi_k}+ (b+(m-\frac{n}{2})a-\frac{i\eta a}{2}) \delta_{0}+(b+(m-\frac{n}{2})a
+\frac{i\eta a}{2}) \delta_{\infty}
\end{equation}

$$
\boldsymbol{\tau_2}=-\frac{\sigma}{2} \delta_{0}-\frac{\sigma}{2}\ \delta_{\infty}+\sigma \delta_z,
$$
where the parameter $\sigma= \frac{1}{a}$.

The $n$-leg operator with screening charges $\boldsymbol{\xi}$ and background charge $\boldsymbol{\beta}$ is given by the OPE exponential: 

\begin{equation}
\mathcal{O}_{\boldsymbol{\beta}}[\boldsymbol{\tau_1}]=\frac{C_{(b)}[\boldsymbol{\tau_1}+\boldsymbol{\beta}]}{C_{(b)}[\boldsymbol{\beta}]} \mathrm{e}^{\odot i \Phi[\boldsymbol{\tau_1}]}.
\end{equation}

\end{defn}

\begin{defn}[Screening fields]

For each link pattern $\alpha$, we can choose closed contours $\mathcal{C}_1, \ldots, \mathcal{C}_n$ along which we may integrate the $\boldsymbol{\xi}$ variables to screen the vertex fields.
Let $\mathcal{S}$ be the screening operator, we define the screening operation as 
$$
\mathcal{S}_{\alpha}\mathcal{O}_{\boldsymbol{\beta}}[\boldsymbol{\tau_1}]=\oint_{\mathcal{C}_1} \ldots \oint_{\mathcal{C}_n} \mathcal{O}_{\boldsymbol{\beta}}[\boldsymbol{\tau_1}].
$$

Meanwhile, we integrate the correlation function $\mathbf{E}\mathcal{O}_{\boldsymbol{\beta}}[\boldsymbol{\tau_1}]=\Phi_\kappa(\boldsymbol{z}, \boldsymbol{\xi})$ , the conformal dimension is 1 at the $\boldsymbol{\xi}$ points, i.e. since $\lambda_b(-2 a)=1$. This leads to the partition function for the corresponding multiple radial SLE($\kappa$) system:
$$
\mathcal{J}_{\alpha}^{\eta}(\boldsymbol{z}):=\mathbf{E}\mathcal{S}_{\alpha}\mathcal{O}_{\boldsymbol{\beta}}[\boldsymbol{\tau_1}]=\oint_{\mathcal{C}_1} \ldots \oint_{\mathcal{C}_n} \Phi_\kappa(\boldsymbol{z}, \boldsymbol{\xi}) d \xi_n \ldots d \xi_1 .
$$

\end{defn}

\begin{thm}[Martingale observable]
 For any tensor product $X$ of fields in the OPE family $\mathcal{F}_{\boldsymbol{\beta}}$ of $\Phi_{\boldsymbol{\beta}}$,
\begin{equation}
M_{t,\kappa}(X)=\frac{\mathbf{E}\mathcal{S}_{\alpha} \mathcal{O}_{\boldsymbol{\beta}}[\boldsymbol{\tau_1}] X}{\mathbf{E}\mathcal{S}_{\alpha} \mathcal{O}_{\boldsymbol{\beta}}[\boldsymbol{\tau_1}]} \| g_t^{-1}
\end{equation}
is a local martingale, where $g_t(z)$ is the Loewner map for multiple radial SLE($\kappa$) system associated to $\mathcal{J}_{\alpha}^{\eta}(\boldsymbol{z})=\mathbf{E}\mathcal{S}_{\alpha} \mathcal{O}_{\boldsymbol{\beta}}[\boldsymbol{\tau_1}]$.

\end{thm}

\begin{cor}
Let the divisor $ \boldsymbol{\tau_2}=-\frac{\sigma}{2} \delta_{0}-\frac{\sigma}{2}\ \delta_{\infty}+\sigma \delta_z$ where the parameter $\sigma= \frac{1}{a}$, and insert $ X = \mathcal{O}_{\boldsymbol{\beta}}[\boldsymbol{\tau_2}] $.

\begin{equation}
M_{t,\kappa}(z)=\frac{\mathbf{E}\mathcal{S} \mathcal{O}_{\boldsymbol{\beta}}[\boldsymbol{\tau_1}]\mathcal{O}_{\boldsymbol{\beta}}[\boldsymbol{\tau_2}]}
{\mathbf{E}\mathcal{S} \mathcal{O}_{\boldsymbol{\beta}}[\boldsymbol{\tau_1}]} \| g_t^{-1}
\end{equation}
is local martingale where $g_t(z)$ is the Loewner map for multiple radial SLE($\kappa$) system associated to $\mathcal{Z}_\kappa(\boldsymbol{z})=\mathbf{E}\mathcal{S} \mathcal{O}_{\boldsymbol{\beta}}[\boldsymbol{\tau_1}]$.
\end{cor}

Explicit computation shows that
$$
\begin{aligned}
    &\mathbf{E}\oint_{\mathcal{C}_1} \ldots \oint_{\mathcal{C}_n} \mathcal{O}_{\boldsymbol{\beta}}[\boldsymbol{\tau_1}]\mathcal{O}_{\boldsymbol{\beta}}[\boldsymbol{\tau_2}]\\
    &=\oint_{\mathcal{C}_1} \ldots \oint_{\mathcal{C}_n} \prod_{1 \leq i<j \leq n}(z_i-z_j)^{a^2} \prod_{1 \leq i<j \leq m}(\xi_i-\xi_j)^{4 a^2} \prod_{i=1}^{n} \prod_{j=1}^m\left(z_i-\xi_j\right)^{-2 a^2} \\
    & \prod_j z_j^{a(b-\frac{n-2m}{2}a-\frac{i\eta a}{2}-\frac{\sigma}{2})} \prod_k \xi_k^{-2a(b-\frac{n-2m}{2}a-\frac{i\eta a}{2}-\frac{\sigma}{2})} z^{\sigma(b-\frac{n-2m}{2}a-\frac{i\eta a}{2}-\frac{\sigma}{2})}
g^{\prime}(z_j) ^{\lambda_b(a)}g^{\prime}(z) ^{\lambda_b(\sigma)} \\
& (z-z_j)^{\sigma a}(z-\xi_k)^{-2\sigma a}|g^{\prime}(0)|^{\lambda_b(b+\frac{2m-n}{2}a+\frac{i\eta a}{2}-\frac{\sigma}{2})+\lambda_b(b+\frac{2m-n}{2}a-\frac{i\eta a}{2}-\frac{\sigma}{2})}
\end{aligned}
$$
$$
\begin{aligned}
    &\mathbf{E}\oint_{\mathcal{C}_1} \ldots \oint_{\mathcal{C}_n} \mathcal{O}_{\boldsymbol{\beta}}[\boldsymbol{\tau_1}] \\
    &=\oint_{\mathcal{C}_1} \ldots \oint_{\mathcal{C}_n} \prod_{1 \leq i<j \leq n}(z_i-z_j)^{a^2} \prod_{1 \leq i<j \leq m}(\xi_i-\xi_j)^{4 a^2} \prod_{i=1}^{n} \prod_{j=1}^m\left(z_i-\xi_j\right)^{-2 a^2} \\
    & \prod_j z_j^{a(b-\frac{n-2m}{2}-\frac{i\eta a}{2})} \prod_k \xi_k^{-2a(b-\frac{n-2m}{2}a-\frac{i\eta a}{2})}  g^{\prime}(z_j) ^{\lambda_b(a)}
    \\
    & |g^{\prime}(0)|^{\lambda_b(b-\frac{2m-n}{2}a
+\frac{i\eta a}{2})+\lambda_b(b-\frac{2m-n}{2}a-\frac{i\eta a}{2})}
\end{aligned}.
$$

\begin{conjecture}
 
As $\kappa \rightarrow 0$, the Coulomb gas contour integrals concentrate on the critical points of the master function.

\begin{equation}
\begin{aligned}
N_t(z)=& M_{t,0}(z)=\lim_{\kappa\rightarrow 0} M_{t,\kappa}(z)=\lim_{\kappa \rightarrow 0}\frac{\mathbf{E}\oint_{\mathcal{C}_1} \ldots \oint_{\mathcal{C}_n} \mathcal{O}_{\boldsymbol{\beta}}[\boldsymbol{\tau_1}]\mathcal{O}_{\boldsymbol{\beta}}[\boldsymbol{\tau_2}]}
{\mathbf{E}\oint_{\mathcal{C}_1} \ldots \oint_{\mathcal{C}_n} \mathcal{O}_{\boldsymbol{\beta}}[\boldsymbol{\tau_1}]}\\
&=|g^{\prime}(0)|^{-(m-\frac{n}{2})}\frac{\prod_{j=1}^{m}\xi_k}{ \sqrt{\prod_{k=1}^{n}z_k}}
z^{m-\frac{n}{2}-1-\frac{\eta i}{2}}g'(z)\frac{\prod_{k=1}^{n}(z-z_k)}{\prod_{j=1}^{m}(z-\xi_j)^2}
\end{aligned}
\end{equation}
\end{conjecture}

which is exactly the integral of motion we use.

$M_{t,\kappa}(z)$ is a $(\lambda_b(\sigma),0)$ differential with respect to $z$, where $\lambda_b(\sigma)= \frac{1}{2a^2}-\frac{b}{a}$.
By taking the limit $\kappa \rightarrow 0$,
$\lim_{\kappa \rightarrow 0}\lambda_b(\sigma)=1$, thus $M_{t,0}(z)$ is a (1,0) differential.

\begin{remark}
   The integral of motion $N_{t}(z)$ can be verified through direct computation.
\end{remark}

\subsection{Examples: spin}
In this section, we provide a series of figures to illustrate the trace configurations arising from various multiple radial $\mathrm{SLE}(0)$ systems with spin.

\begin{remark}
In the case of multiple radial SLE(0) with spin $\eta$, for $\boldsymbol{z}$ and $\boldsymbol{\xi}$, the quadratic differential $Q(z)dz^2$ can be written as
$$Q(z)dz^2= \frac{\prod_{j=1}^{m}\xi_{k}^{2}}{\prod_{k=1}^{n}z_{k}}
z^{2m-n-2-\eta i}\frac{\prod_{k=1}^{n}(z-z_k)^2}{\prod_{j=1}^{m}(z-\xi_j)^4}dz^2. $$

\end{remark}

\begin{figure}[h]
\centering
\begin{minipage}[t]{0.43\linewidth}
    \centering
    \includegraphics[width=6cm]{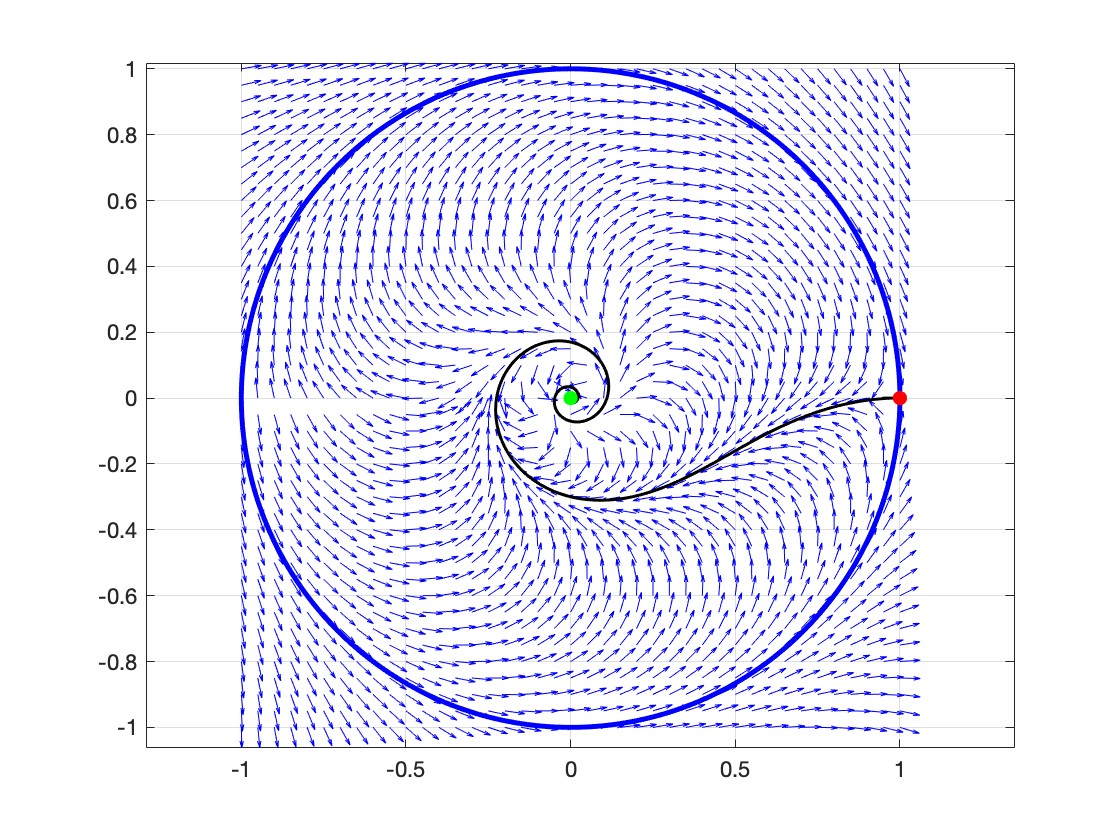}
    \caption{$n = 1$, $\eta = -4$}
\end{minipage}
\hfill
\begin{minipage}[t]{0.43\linewidth}
    \centering
    \includegraphics[width=6cm]{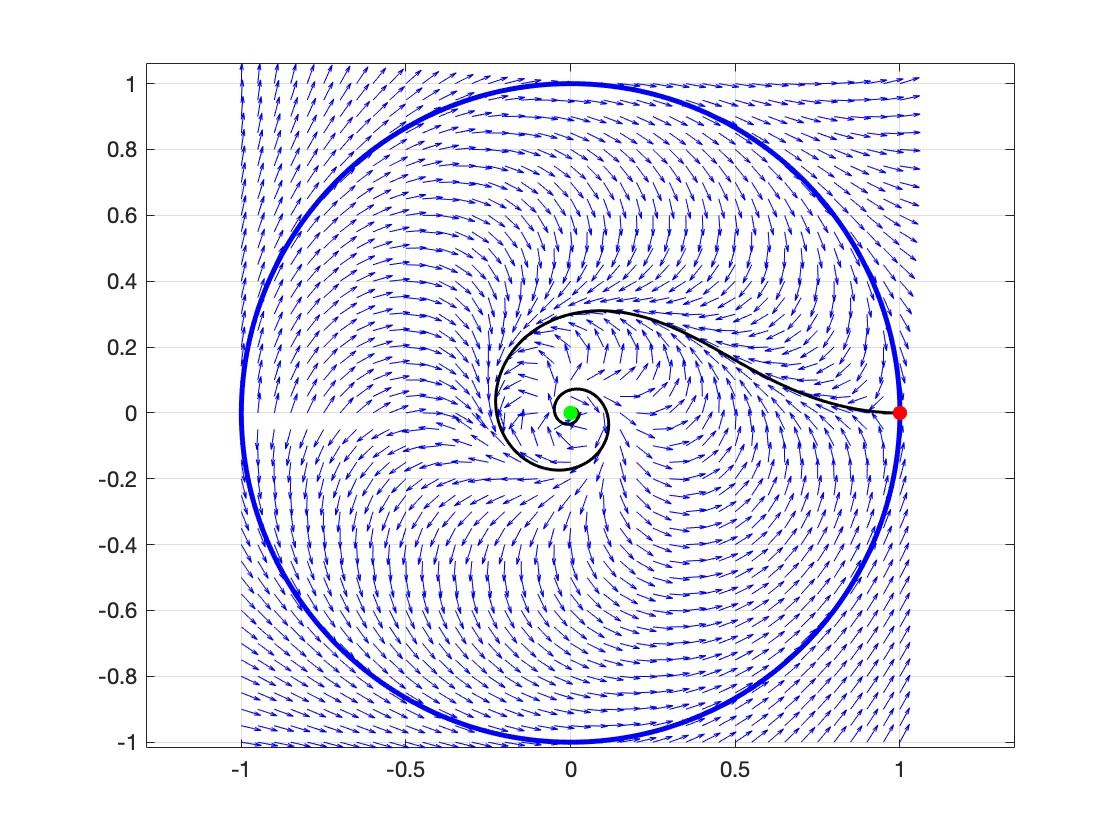}
    \caption{$n = 1$, $\eta = 4$}
\end{minipage}
\end{figure}

\textbf{Figure 5.1:} \( n = 1 \), \( z_1 = 1 \), \( \eta = -4 \). A clockwise spiral connects \( z_1 \) to 0.
\[
\sqrt{Q(z)} = z^{-3/2 + 2i}(z - 1)
\]

\textbf{Figure 5.2:} \( \eta = 4 \). A counterclockwise spiral connects \( z_1 \) to 0.
\[
\sqrt{Q(z)} = z^{-3/2 - 2i}(z - 1)
\]

---

\begin{figure}[h]
\centering
\begin{minipage}[t]{0.43\linewidth}
    \includegraphics[width=6cm]{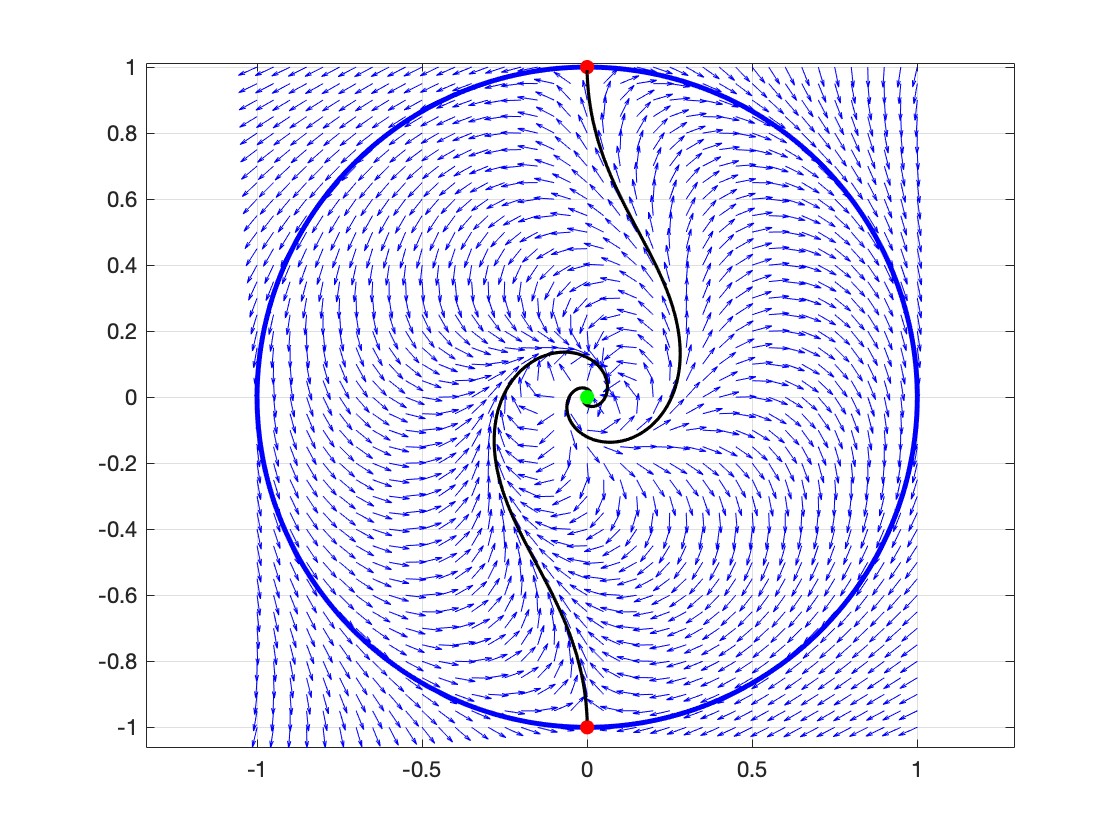}
    \caption{$n = 2$, $m = 0$, $\eta = -4$}
\end{minipage}
\hfill
\begin{minipage}[t]{0.43\linewidth}
    \includegraphics[width=6cm]{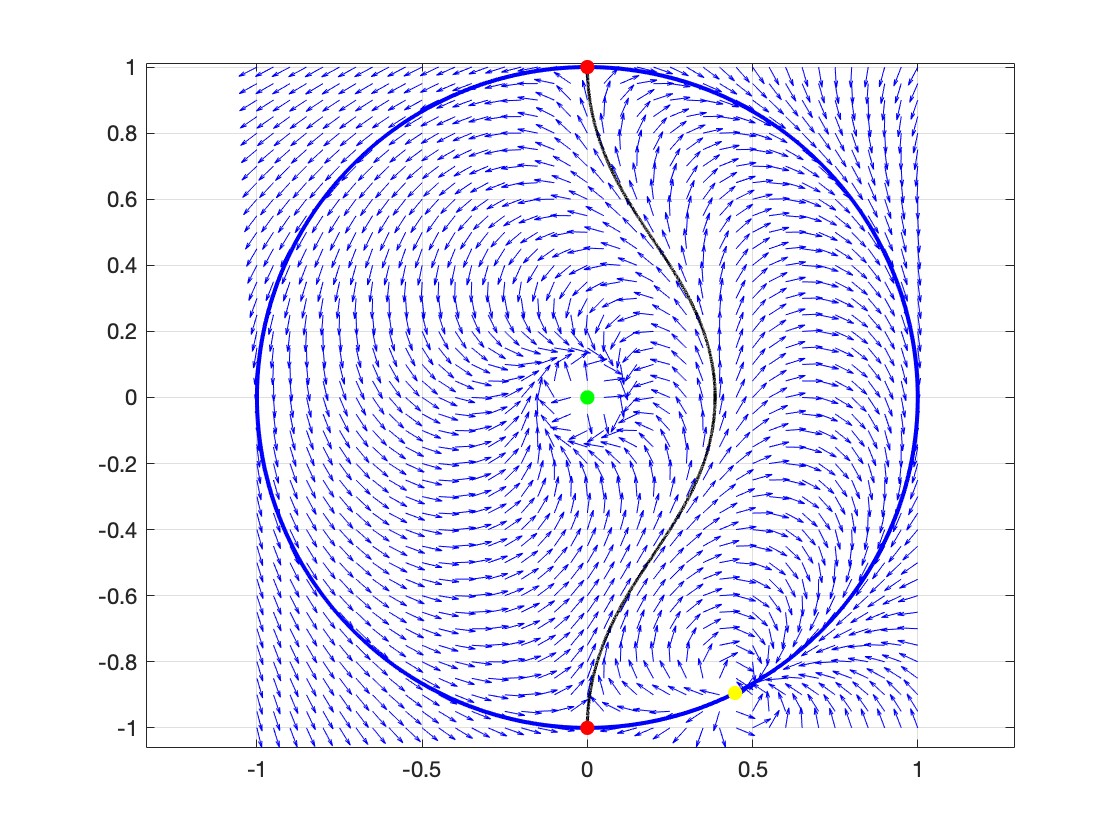}
    \caption{$n = 2$, $m = 1$, $\eta = -4$}
\end{minipage}
\end{figure}

\textbf{Figure 5.3:} \( z_1 = i, z_2 = -i \), two spirals connect \( z_1, z_2 \) to 0.
\[
\sqrt{Q(z)} = z^{-2 + 2i}(z - i)(z + i)
\]

\textbf{Figure 5.4:} Pole \( \xi = \frac{\sqrt{-4 - 2i}}{\sqrt{4 - 2i}} \). The link pattern remains stable under spin perturbation; the pole moves clockwise as \( \eta < 0 \). A closed orbit is observed.
\[
\sqrt{Q(z)} = i z^{-1 + 2i} \frac{(z - i)(z + i)}{\left(z - \frac{\sqrt{-4 - 2i}}{\sqrt{4 - 2i}}\right)^2}
\]

---

\begin{figure}[h]
\centering
\begin{minipage}[t]{0.43\linewidth}
    \includegraphics[width=6cm]{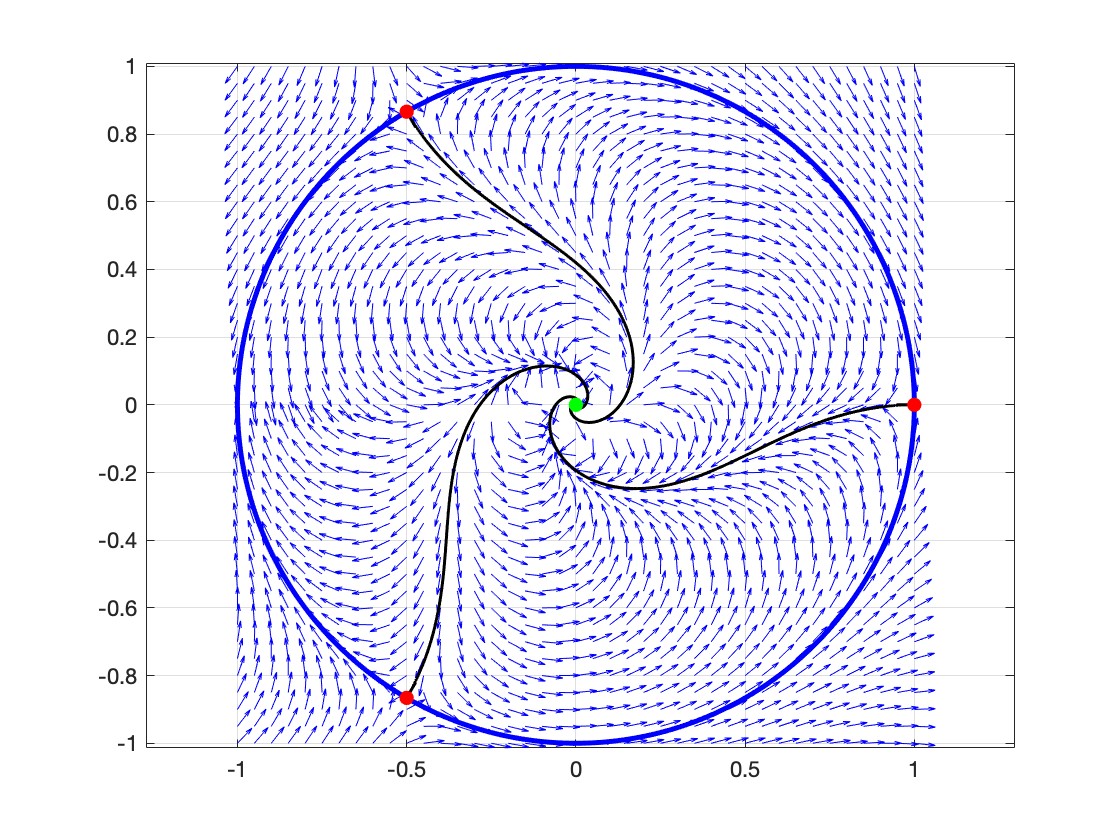}
    \caption{$n = 3$, $m = 0$, $\eta = -4$}
\end{minipage}
\hfill
\begin{minipage}[t]{0.43\linewidth}
    \includegraphics[width=6cm]{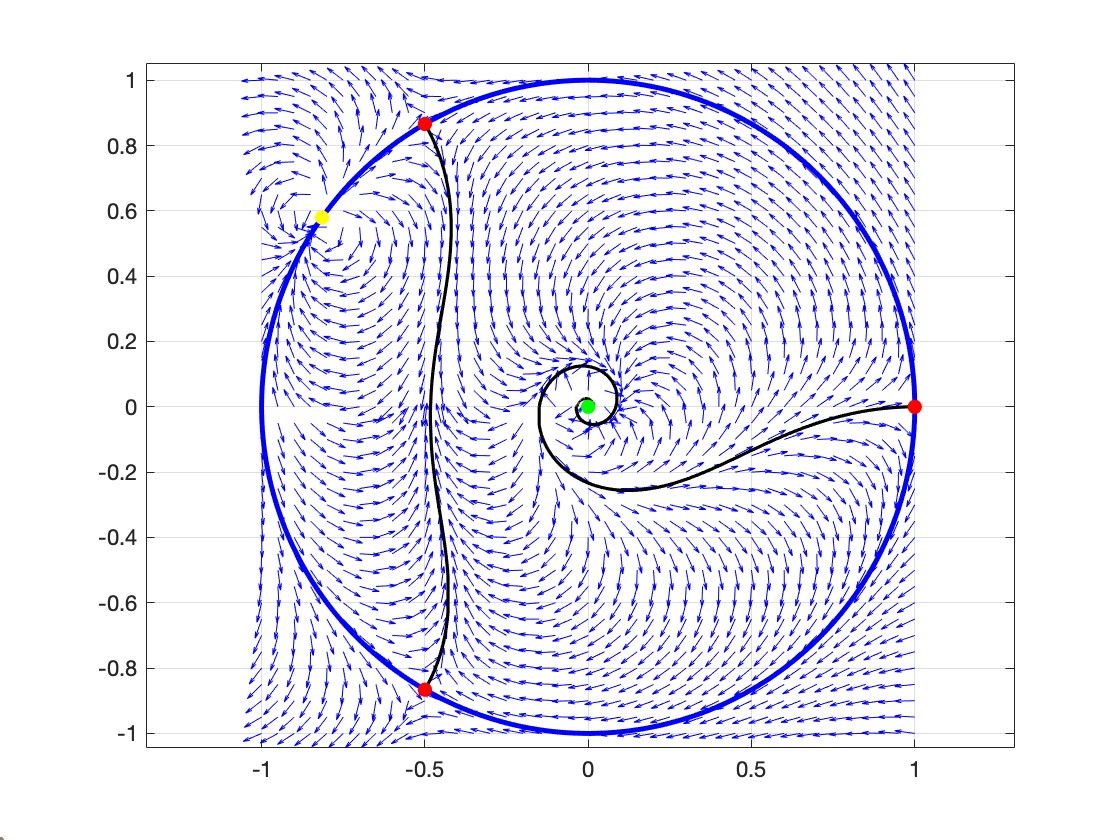}
    \caption{$n = 3$, $m = 1$, $\eta = -4$}
\end{minipage}
\end{figure}

\textbf{Figure 5.5:} \( z_k = e^{2k\pi i / 3} \). Three spirals connect each \( z_k \) to 0.
\[
\sqrt{Q(z)} = z^{-5/2 + 2i}(z - 1)(z - e^{2\pi i/3})(z - e^{4\pi i/3})
\]

\textbf{Figure 5.6:} Pole \( \xi = \frac{(4 + 3i)^{1/3}}{(4 - 3i)^{1/3}} \).
\[
\sqrt{Q(z)} = z^{-3/2 + 2i} \frac{(z - 1)(z - e^{2\pi i/3})(z - e^{4\pi i/3})}{(z - \xi)^2}
\]

---

\begin{figure}[h]
\centering
\begin{minipage}[t]{0.43\linewidth}
    \includegraphics[width=6cm]{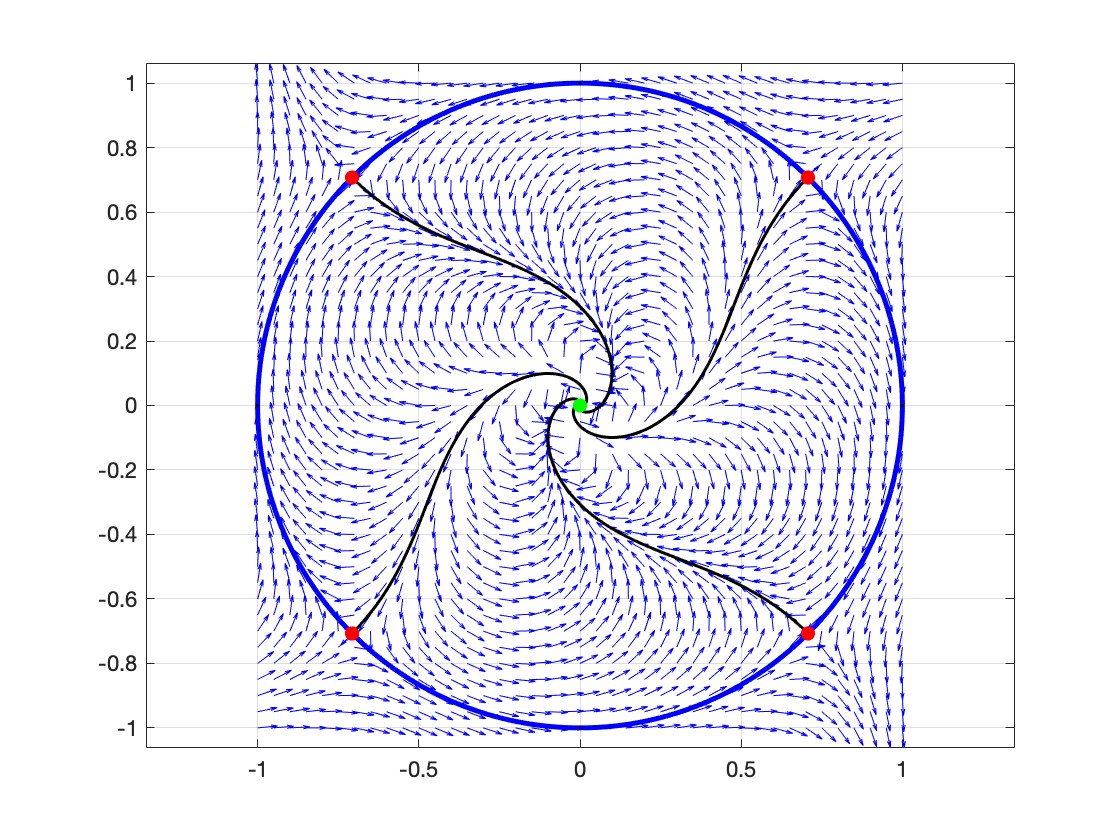}
    \caption{$n = 4$, $m = 0$, $\eta = -4$}
\end{minipage}
\hfill
\begin{minipage}[t]{0.43\linewidth}
    \includegraphics[width=6cm]{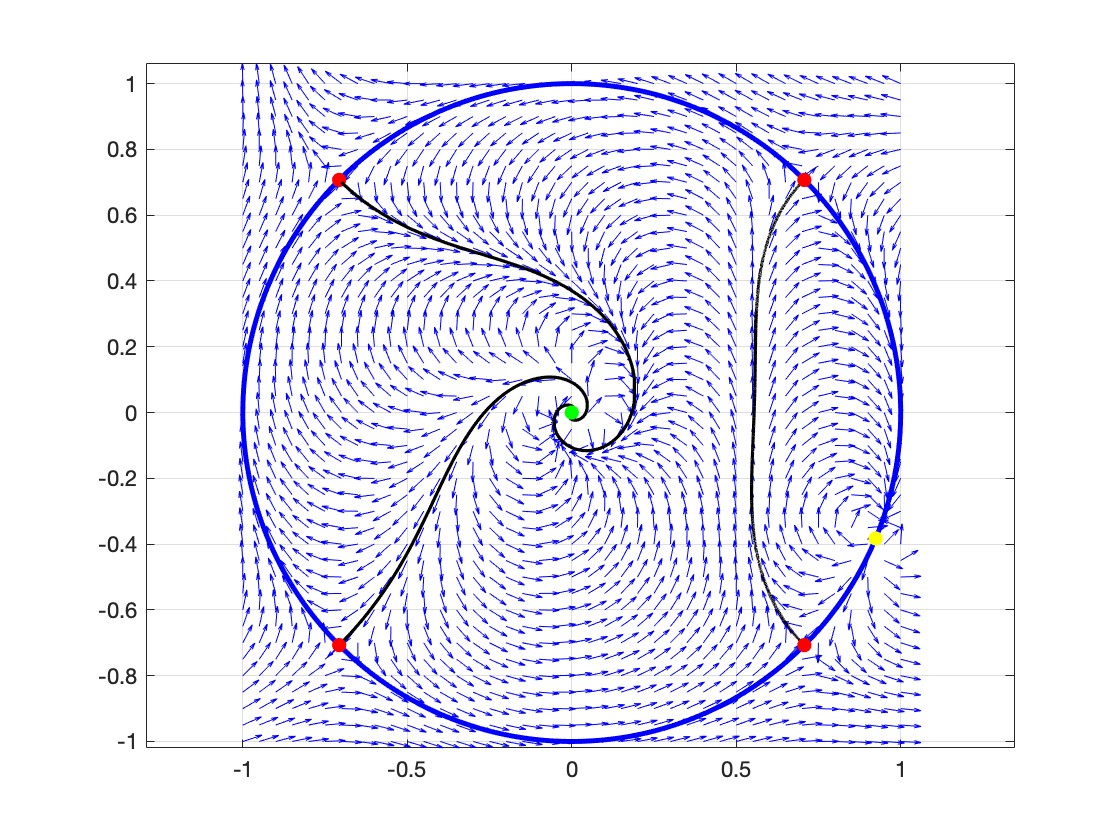}
    \caption{$n = 4$, $m = 1$, $\eta = -4$}
\end{minipage}
\end{figure}

\textbf{Figure 5.7:} \( z_k = e^{(2k+1)\pi i/4} \), \( \frac{\eta}{2} = 2 \).
\[
\sqrt{Q(z)} = z^{-3 + 2i} \prod_{k=0}^{3} \left(z - e^{(2k+1)\pi i/4} \right)
\]

\textbf{Figure 5.8:} Pole \( \xi = \frac{(-4 - 4i)^{1/4}}{(4 - 4i)^{1/4}} \).
\[
\sqrt{Q(z)} = -i \frac{(-4 - 4i)^{1/4}}{(4 - 4i)^{1/4}} z^{-2 + 2i} 
\frac{\prod_{k=0}^{3} (z - e^{(2k+1)\pi i/4})}{(z - \xi)^2}
\]

---

\begin{figure}[h]
\centering
\begin{minipage}[t]{0.43\linewidth}
    \includegraphics[width=6cm]{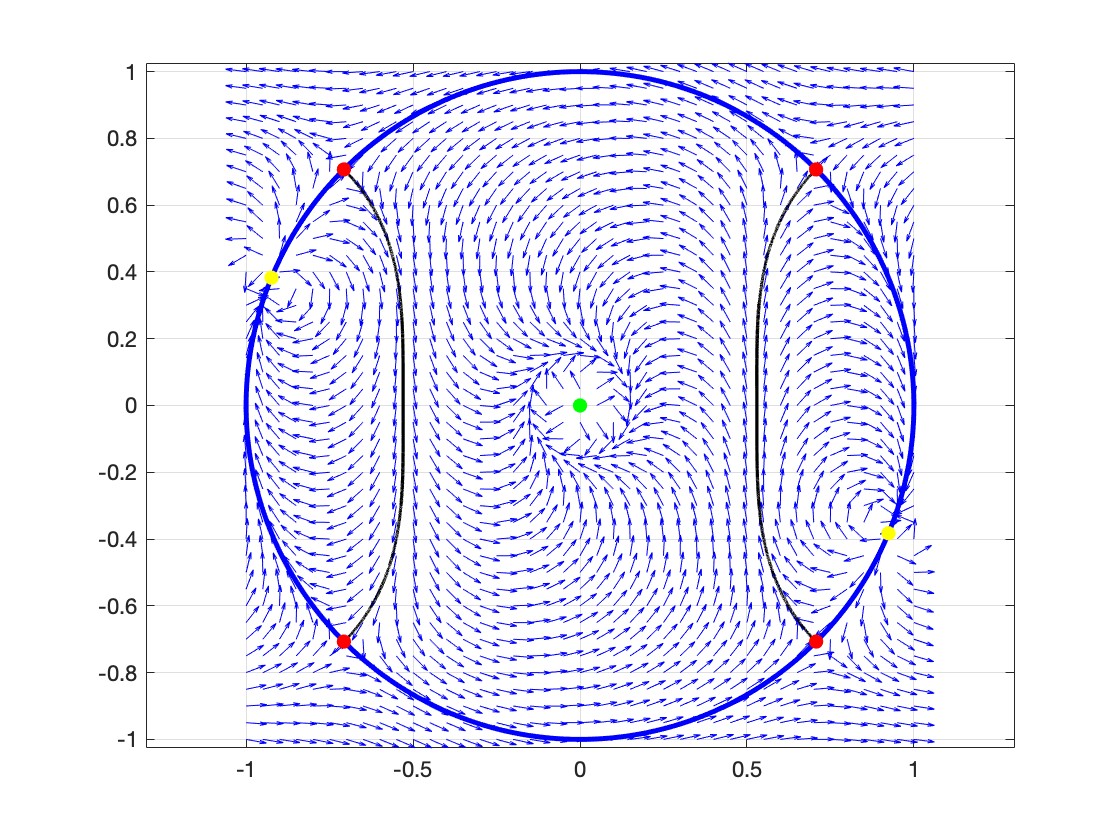}
    \caption{$n = 4$, $m = 2$, $\eta = -4$}
\end{minipage}
\hfill
\begin{minipage}[t]{0.43\linewidth}
    \includegraphics[width=6cm]{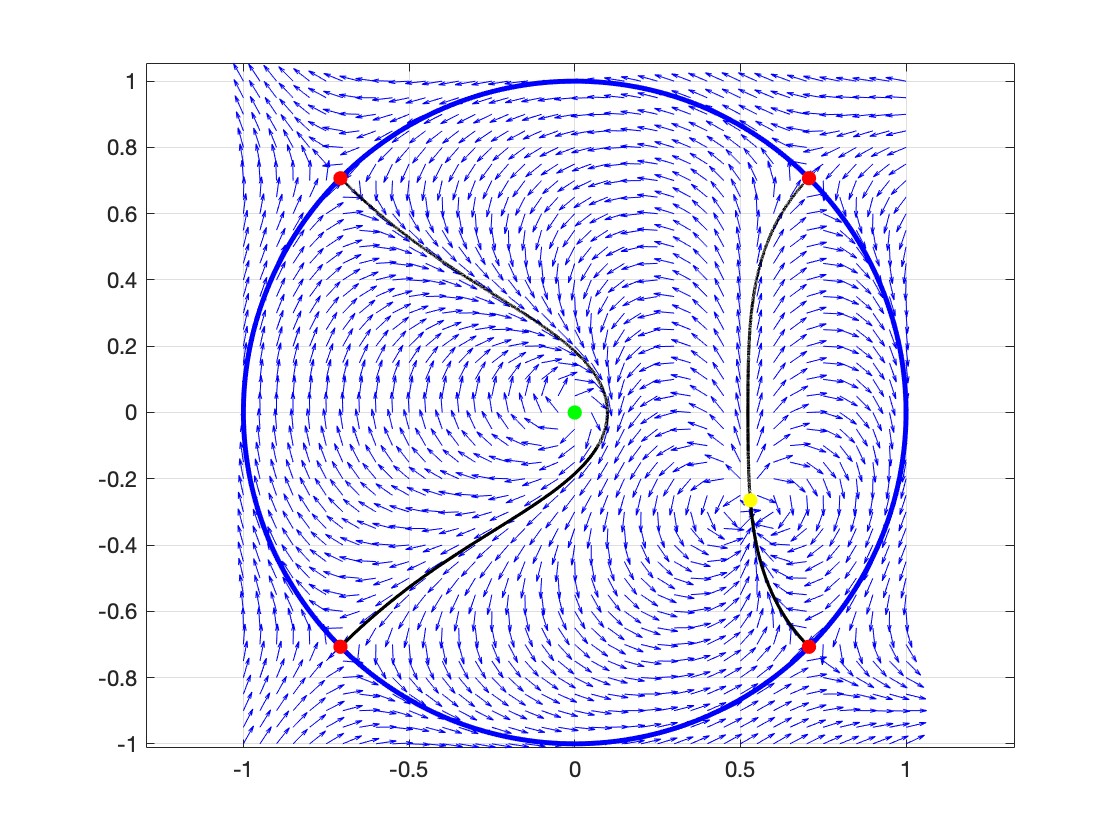}
    \caption{$n = 4$, $m = 2$, $\eta = -1$}
\end{minipage}
\end{figure}

\textbf{Figure 5.9:} Poles at  
\[
\xi_1 = \frac{(-4 - 4i)^{1/4}}{(4 - 4i)^{1/4}}, \quad 
\xi_2 = -\frac{(-4 - 4i)^{1/4}}{(4 - 4i)^{1/4}}
\]
The link pattern is stable under perturbation; a closed orbit appears.
\[
\sqrt{Q(z)} = -i \frac{(-4 - 4i)^{1/2}}{(4 - 4i)^{1/2}} z^{-2 + i}
\frac{\prod_{k=0}^{3}(z - e^{(2k+1)\pi i/4})}{(z - \xi_1)^2 (z - \xi_2)^2}
\]

\textbf{Figure 5.10:} \( \eta = -1 \). Let \( \xi_1 = 0.5299 - 0.2650i \), \( \xi_2 = 1.5097 - 0.7549i \) be the roots of
\[
\sum_{k=0}^3 \frac{\xi - e^{(2k+1)\pi i/4}}{\xi + e^{(2k+1)\pi i/4}} + i 
= 2\frac{\xi + 1/\xi^*}{\xi - 1/\xi^*}
\]
Then
\[
\sqrt{Q(z)} = (0.8 + 0.6i) z^{-1 + i/2} 
\frac{\prod_{k=0}^{3} (z - e^{(2k+1)\pi i/4})}{(z - \xi_1)^2 (z - \xi_2)^2}
\]

\section{Relations to Calogero-Sutherland system}

\subsection{Multiple radial SLE(0) and classical Calogero-Sutherland system}

In this section, we study the relations between the multiple radial SLE(0) and classical Calogero-Sutherland system.

\begin{defn}
The time evolution of a $n$-particle system on the circle is given by the Hamilton's equations:
$$\dot{\theta}_j= \frac{\partial{H}}{\partial{p_j}}, \dot{p}_j=-\frac{\partial{H}}{\partial{\theta_j}}, j=1,\ldots,n, $$
The Hamiltonian is of the form:
$H(\boldsymbol{\theta},\boldsymbol{p}) = \sum \frac{p_{j}^{2}}{2}+ U(\boldsymbol{\theta})$ where $U(\boldsymbol{\theta})$ is a smooth real-valued function on $R^n$.  The initial state of the system is encoded in a position vector $\boldsymbol{\theta}=(\theta_1,...,\theta_n) \in (\mathbb{R}/2\pi \mathbb{Z})^n$ and a momentum vector $p=(p_1,...,p_n)\in \mathbb{R}^n$ where $p_j = \dot{\theta}_j$, $j=1,2\ldots,n$. \\
We consider the special case where $U(\boldsymbol{\theta})$ is a sum of pair potentials
$$U(\boldsymbol{\theta}) = \sum_{j<k} V_{jk}(\theta_j-\theta_k).$$
For the Calogero-Sutherland system, the pair potential is given by $$V_{jk}= -\frac{2}{\sin^2(\frac{\theta_j-\theta_k}{2})}.$$

\end{defn}

The following theorem describes the evolution of the growth points $\boldsymbol{\theta}$ and $\boldsymbol{\zeta}$ for the multiple radial SLE(0) systems constructed using stationary relations.

\begin{thm}\label{Evolution of growth points and screening charges}
 Let $\boldsymbol{\theta}=\left\{\theta_1, \ldots, \theta_{n}\right\}$ be distinct real points and $\zeta=\left\{\zeta_1, \ldots, \zeta_{m}\right\}$ closed under conjugation and solve the stationary relation. Let $\boldsymbol{\theta}(t)$ and $\boldsymbol{\zeta}(t)$ evolve according to multiple radial SLE(0) system with common parametrization of capacity ($\nu_j(t)=1$). 
 
\begin{itemize}
\item[(i)]  The pair $(\boldsymbol{\theta}(t), \boldsymbol{\zeta}(t))$ forms the closed dynamical system satisfying
\begin{equation}\label{x derivative}
\dot{\theta}_j=2 \left(\sum_{k \neq j} \cot(\frac{\theta_j-\theta_k}{2})-\sum_{k=1}^{m} \cot(\frac{\theta_j-\zeta_k}{2}) \right),
\end{equation}
and

\begin{equation} \label{xi derivative}
\dot{\zeta}_k=2\left(-\sum_{l \neq k}\cot(\frac{\zeta_k-\zeta_l}{2})+\sum_{j=1}^{n}\cot(\frac{\zeta_k-\theta_j}{2})\right)
\end{equation}

\item[(ii)]$\boldsymbol{\theta}(t)$ evolve according to the classical Calegero-Sutherland Hamiltonian, in other words:
$$
\ddot{\theta}_j=-\sum_{k \neq j} \frac{\cos(\frac{\theta_j-\theta_k}{2})}{\sin^3(\frac{\theta_j-\theta_k}{2})} .
$$
\item[(iii)]$\zeta_k$ follows the second-order dynamics.
\begin{equation}
\ddot{\zeta}_k=-\sum_{l \neq k}  \frac{\cos(\frac{\zeta_k-\zeta_l}{2})}{\sin^3(\frac{\zeta_k-\zeta_l}{2})} .
\end{equation}
\item[(iv)]
The energy of the system is given by $$\mathcal{H}(\boldsymbol{\theta},\boldsymbol{p})=-\frac{n(2m-n)^2}{2}+\frac{n}{2}-\frac{n(n^2-1)}{6}$$

\end{itemize} 
\end{thm}

\begin{proof}[Proof of theorem (\ref{Evolution of growth points and screening charges})]
\

\begin{itemize}
\item[(i)]
 The evolution of $\theta_j(t)$ is
$$
\begin{aligned}
\dot{\theta}_j=&\left(\sum_{k \neq j} \cot(\frac{\theta_j-\theta_k}{2})-2\sum_{k=1}^{m} \cot(\frac{\theta_j-\zeta_k}{2})+\sum_{k \neq j} \cot(\frac{\theta_j-\theta_k}{2})\right) \\
=&2\sum_{k \neq j} \cot(\frac{\theta_j-\theta_k}{2})-2\sum_{k=1}^{m} \cot(\frac{\theta_j-\zeta_k}{2}),
\end{aligned}
$$
On the other hand, since the poles follow the Loewner flow we have $\zeta_k(t):=$ $g_t\left(\zeta_k(0)\right)$, and therefore
$$
\dot{\zeta}_k=\dot{g}_t\left(\zeta_k(0)\right)= \sum_{j=1}^{n} \cot(\frac{g_t\left(\zeta_k(0)\right)-\theta_j}{2})=\sum_{j=1}^{n} \cot(\frac{\zeta_k-\theta_j}{2}) .
$$
The stationary relation implies that
$$
\dot{\zeta}_k=2\sum_{l \neq k} \cot(\frac{\zeta_k-\zeta_l}{2})=-2\sum_{l \neq k}\cot( \frac{\zeta_k-\zeta_l}{2})+ 2\sum_{j=1}^{n} \cot(\frac{\zeta_k-\theta_j}{2}),
$$
\item[(ii)]
By differentiating, we have
$$
\ddot{\theta}_j=-\sum_{k \neq j} \frac{\dot{\theta}_j-\dot{\theta}_k}{\sin^2(\frac{\theta_j-\theta_k}{2})}+\sum_l \frac{\dot{\theta}_j-\dot{\zeta}_l}{\sin^2(\frac{\theta_j-\zeta_l}{2})} .
$$
Using the formula (\ref{x derivative}) for $\dot{\theta}_j, \dot{\theta}_k$ and the equality (\ref{xi derivative})  for $\dot{\zeta}_l$ we obtain
$$
\begin{aligned}
\ddot{\theta}_j= & -\frac{1}{2}\sum_{k \neq j} \frac{1}{\sin^2(\frac{\theta_j-\theta_k}{2})} \left( 2\cot(\frac{\theta_j-\theta_k}{2})\right)+\\
&-\frac{1}{2}\sum_{k \neq j} \frac{1}{\sin^2(\frac{\theta_j-\theta_k}{2})} \left(\sum_{l \neq j, k}\left(\cot(\frac{\theta_j-\theta_l}{2})-\cot(\frac{\theta_k-\theta_l}{2})\right)  +\sum_l\left(\cot(\frac{\zeta_l-\theta_j}{2})-\cot(\frac{\zeta_l-\theta_k}{2})\right) \right)\\ & +\frac{1}{2}\sum_l \frac{1}{\sin^2(\frac{\theta_j-\zeta_l}{2})}\left(\sum_{k \neq j} \cot(\frac{\theta_j-\theta_k}{2})+\sum_{s=1}^m \cot(\frac{\zeta_s-\theta_j}{2})+\sum_{s \neq l} \cot(\frac{\zeta_l-\zeta_s}{2})-\sum_{k=1}^{n} \cot(\frac{\zeta_l-\theta_k}{2})\right) .
\end{aligned}
$$
Rearranging terms gives
$$
\begin{aligned}
\ddot{\theta}_j & +\sum_{k \neq j} \frac{\cos(\frac{\theta_j-\theta_k}{2})}{\sin^3(\frac{\theta_j-\theta_k}{2})}- \frac{1}{2}\sum_{k \neq j} \sum_{l \neq j, k} \frac{1}{\sin(\frac{\theta_j-\theta_k}{2})\sin(\frac{\theta_j-\theta_l}{2})\sin(\frac{\theta_k-\theta_l}{2})} \\
& =-\frac{1}{2}\sum_{k \neq j} \sum_l \frac{1}{\sin\left(\frac{\theta_j-\theta_k}{2}\right)\sin\left(\frac{\theta_j-\zeta_l}{2}\right)\left(\frac{\theta_k-\zeta_l}{2}\right)}+\\
& +\frac{1}{2}\sum_l \frac{1}{\sin\left(\frac{\theta_j-\zeta_l}{2}\right)^2}\left(\sum_{k \neq j} \cot(\frac{\theta_j-\theta_k}{2})+\sum_m \cot
(\frac{\zeta_m-\theta_j}{2})-\sum_{m \neq l} \cot(\frac{\zeta_l-\zeta_m}{2})\right) .
\end{aligned}
$$
The last term on the right hand side used the stationary relation and then use the stationary relation again to obtain
$$
\begin{aligned}
&\frac{1}{2}\sum_l \frac{1}{\sin\left(\frac{\theta_j-\zeta_l}{2}\right)^2}\left(\sum_{k \neq j} \cot(\frac{\theta_j-\theta_k}{2})+\sum_m \cot
(\frac{\zeta_m-\theta_j}{2})-\sum_{m \neq l} \cot(\frac{\zeta_l-\zeta_m}{2})\right)  \\
& =\frac{1}{2}\sum_l \frac{1}{\sin\left(\frac{\theta_j-\zeta_l}{2}\right)^2} \sum_{m \neq l}\left(\cot(\frac{\zeta_l-\zeta_m}{2})+\cot(\frac{\zeta_m-\theta_j}{2})\right)\\
&=\frac{1}{2}\sum_l \sum_{m \neq l} \frac{1}{\sin\left(\frac{\theta_j-\zeta_l}{2}\right)\sin\left(\frac{\theta_j-\zeta_m}{2}\right)\sin\left(\frac{\zeta_l-\zeta_m}{2}\right)} .
\end{aligned}
$$
Combining all of the above, we obtain
$$
\begin{aligned}
\ddot{\theta}_j+\sum_{k \neq j} \frac{\cos(\frac{\theta_j-\theta_k}{2})}{\sin^3(\frac{\theta_j-\theta_k}{2})} & =\frac{1}{2}\sum_{k \neq j} \sum_{l \neq j, k} \frac{1}{\sin\left(\frac{\theta_j-\theta_k}{2}\right)\sin\left(\frac{\theta_j-\theta_l}{2}\right)\sin\left(\frac{\theta_k-\theta_l}{2}\right)} \\
& +\frac{1}{2}\sum_l \sum_{m \neq l} \frac{1}{\sin\left(\frac{\theta_j-\zeta_l}{2}\right)\sin\left(\frac{\theta_j-\zeta_m}{2}\right)\sin\left(\frac{\zeta_l-\zeta_m}{2}\right)}
\end{aligned}
$$
The right-hand side is canceled by symmetry.

\item[(iii)]

Differentiating the equality (\ref{xi derivative}), we have
$$
\ddot{\zeta}_k=-\sum_{l \neq k} \frac{\dot{\zeta}_k-\dot{\zeta}_l}{ \sin^2(\frac{\zeta_k-\zeta_l}{2})} .
$$
Now by using the first equality of (6.4) again for $\dot{\zeta}_k, \dot{\zeta}_l$ we obtain
$$
\ddot{\zeta}_k=-\frac{1}{2}\sum_{l \neq k} \frac{1}{\sin^2(\frac{\zeta_k-\zeta_l}{2})}\left(\cot(\frac{\zeta_k-\zeta_l}{2})+\sum_{m \neq k, l} \cot(\frac{\zeta_k-\zeta_m}{2})-\cot(\frac{\zeta_l-\zeta_k}{2})-\sum_{m \neq k, l} \cot(\frac{\zeta_l-\zeta_m}{2})\right)
$$
Rearranging terms gives
$$
\ddot{\zeta}_k=-\sum_{l \neq k}  \frac{\cos(\frac{\zeta_k-\zeta_l}{2})}{\sin^3(\frac{\zeta_k-\zeta_l}{2})}+ \frac{1}{2}\sum_{l \neq k} \sum_{m \neq k, l} \frac{1}{\sin(
\frac{\zeta_k-\zeta_l}{2})\sin(\frac{\zeta_k-\zeta_m}{2})\sin(\frac{\zeta_l-\zeta_m}{2})} .
$$
The last term is canceled by symmetry.
\item[(iv)]
  For a multiple radial SLE(0) system with $n$ growth points and $m$ screening charges that solve the stationary relations, by equation (\ref{m screening null vector constant}) in the proof of the theorem (\ref{Stationary relation imply partition function}), 
  $$U_j=\sum_{k \neq j} \cot(\frac{\theta_j-\theta_k}{2})-2\sum_{k=1}^{m} \cot(\frac{\theta_j-\zeta_k}{2})$$
  satisfies the null vector equation (\ref{null vector equation for kappa 0}) with constant
    $$h_{m,n}=-\frac{(2m-n)^2}{2}+\frac{1}{2}$$
   Plugging into equation (\ref{null Hamiltonian}) and equation (\ref{CS and Null}), we obtain the desired result.
\end{itemize}
\end{proof}

\begin{proof}[Proof of theorem (\ref{CS results kappa=0})]
For multiple radial SLE(0) system with common parametrization of capacity (i.e. $\nu_j(t)=1$ for $j=1,2,\ldots,n$), let $\left\{\left(\theta_j, U_j\right), j=1, \ldots,  n\right\}$ are related to $\left\{\left(\theta_j, p_j\right), j=1, \ldots,  n\right\}$ via
$$
p_j=\left(U_j+\sum_{k \neq j}\cot(\frac{\theta_j-\theta_k}{2}) \right)
$$
where $U_j$ solves the null vector equations (\ref{null vector equation for kappa 0}).
\begin{itemize}
 
\item[(i)]

Solving for $U_j$ and inserting the result into the left-hand side of the null vector equation leads to the identity.

\begin{equation}\label{null Hamiltonian}
\begin{aligned}
h=&\frac{1}{2} U_j^2+\sum_k f_{k j} U_k-\sum_k \frac{3}{2}(1+ f_{ jk}^2) \\
&=\frac{1}{2} p_j^2- \sum_k\left(p_j+p_k\right) f_{j k}+\sum_k \sum_{l \neq k} f_{j k} f_{j l}-2\sum_k f_{j k}^2 +C_{n-1}^{2}+\frac{3}{2}(n-1) \\
&=\mathcal{H}_j(\boldsymbol{\theta}, \boldsymbol{p}) + C_{n-1}^{2}+\frac{3}{2}(n-1) 
\end{aligned} 
\end{equation}

where 
$$
f_{j k}=f_{j k}(\boldsymbol{\theta})= \begin{cases}0, & j=k \\ \cot(\frac{\theta_j-\theta_k}{2}), & j \neq k\end{cases}
$$

Therefore, $\mathcal{H}_j$ is preserved under the Loewner flow.

Futheremore, for each $c \in \mathbb{R}$, the submanifolds defined by the null vector Hamiltonian
\begin{equation}
N_c=\left\{(\boldsymbol{\theta}, \boldsymbol{p}): \mathcal{H}_j(\boldsymbol{\theta}, \boldsymbol{p})=c \text { for all } j\right\}
\end{equation}
are invariant under the Loewner flow

By direct computation, $\mathcal{H}_j$ is related to the Calogero-Sutherland Hamiltonian $\mathcal{H}$ by:
\begin{equation} \label{CS and Null}
 \sum_j \mathcal{H}_j=\mathcal{H}   
\end{equation}

Our next result shows that null vector Hamiltonian $\mathcal{H}_j$ has a nice interpretation in terms of the Lax pair for the Calogero-Sutherland system.

\begin{thm}
The Lax pair is two square matrices $L=L(\boldsymbol{\theta}, \boldsymbol{p})$ and $M=M(\boldsymbol{\theta}, \boldsymbol{p})$ each of size $ n \times  n$, and by \cite{Mos75} the entries are given by
$$
L_{j k}=\left\{\begin{array}{ll}
p_j, & j=k, \\
2f_{j k}, & j \neq k,
\end{array} \quad \text { and } \quad M_{j k}= \begin{cases}-\sum_l f_{j l}^2, & j=k \\
f_{j k}^2, & j \neq k\end{cases}\right.
$$
This leads to the following representation of $\mathcal{H}_j$ in terms of $L^2$.
\begin{equation}
\mathcal{H}_j=\frac{1}{2} \mathrm{e}_j^{\prime} L^2 \mathbf{1}    
\end{equation}

where $\mathrm{e}_j^{\prime}$ is the transpose of the $j$ th standard basis vector and 1 is the vector of all ones. 

Consequently, the $U_j, j=1, \ldots, n$, defined by solving the null vector equations for a given $\boldsymbol{\theta}$ iff the $\boldsymbol{p}$ variables satisfy $L^2(\boldsymbol{\theta}, \boldsymbol{p}) \mathbf{1}=\mathbf{0}$.
\end{thm}
\begin{proof}
Write $L=P-X_1$, where $P=P(\boldsymbol{p})=\operatorname{diag}(\boldsymbol{p})$ is the square matrix with entries of $\boldsymbol{p}$ along its diagonal, and $X_1=X_1(\boldsymbol{\theta})$ is the square matrix with entries $\left(X_1\right)_{j k}=f_{j k}$. Note that $P$ is symmetric and $X_1$ is anti-symmetric. Then
$$
L^2=P^2-P X_1-X_1 P+X_1^2
$$
It is straightforward to compute the entries of $P^2-P X_1-X_1 P$ and see that they give the first two terms on the right-hand side of the Hamiltonian. For $X_1^2$ we have
$$
\begin{aligned}
\mathrm{e}_j^{\prime} X_1^2 1=\sum_k\left(X_1^2\right)_{j k}&=-4(\sum_k \sum_i f_{j l} f_{k l})\\
& =-4\left(\sum_l f_{j l}^2+\sum_{k\neq j} \sum_{l \neq j} f_{j l} f_{k l}\right) \\
& =-4\left(\sum_l f_{j 1}^2+\frac{1}{2} \sum_{k \neq j } \sum_{l \neq k}\left(f_{j l} f_{k l}+f_{j k} f_{l k}\right)\right) \\
& =-4\left(\sum_l f_{j 1}^2-\frac{1}{2} \sum_{k \neq j } \sum_{l \neq k} f_{j k} f_{j l} +\frac{1}{2}C_{n-1}^{2}\right).
\end{aligned}
$$
\end{proof}

\item[(ii)]
\begin{defn}[Poisson Bracket]
For any smooth function \( F = F(\boldsymbol{x}, \boldsymbol{p}) \) defined on phase space, the associated vector field is given by  
\[
X_F = \sum_{j=1}^{n} \frac{\partial F}{\partial p_j} \partial_{x_j} - \sum_{j=1}^{n} \frac{\partial F}{\partial x_j} \partial_{p_j}.
\]  
Given two smooth functions \( F = F(\boldsymbol{x}, \boldsymbol{p}) \) and \( G = G(\boldsymbol{x}, \boldsymbol{p}) \), the commutator of their associated vector fields satisfies  
\[
[X_F, X_G] = X_{\{F, G\}},
\]  
where \( \{F, G\} \), the Poisson bracket of \( F \) and \( G \), is defined by  
\[
\{F, G\} = \sum_{j=1}^{n} \left( \frac{\partial F}{\partial p_j} \frac{\partial G}{\partial x_j} - \frac{\partial F}{\partial x_j} \frac{\partial G}{\partial p_j} \right).
\]  
\end{defn}

By direct computation, for all \( j, k \), the null vector Hamiltonians \( \mathcal{H}_j \) and \( \mathcal{H}_k \) satisfy the Poisson bracket identity  
\[
\{\mathcal{H}_j, \mathcal{H}_k\} = \frac{1}{f_{jk}^2} \left(\mathcal{H}_k - \mathcal{H}_j\right).
\]  

By the definition of \( N_c \), we have \( \{\mathcal{H}_j, \mathcal{H}_k\} = 0 \) along \( N_c \).  

Thus, the vector fields \( X_{\mathcal{H}_j} \) induced by the Hamiltonians \( \mathcal{H}_j \) commute along the submanifolds \( N_c \).

\end{itemize}
\end{proof}

 The relationship between multiple radial SLE(0) and the classical Calogero-Sutherland system serves as the classical analog of its quantum counterpart.

From the perspective of partition functions, we expect that as \( \kappa \to 0 \), the following classical limits exist (at least for appropriately chosen \( \mathcal{Z}(\boldsymbol{\theta}) \)):
\[
\lim_{\kappa \to 0} \frac{\partial \log\left(\mathcal{Z}^{\kappa}(\boldsymbol{\theta})\right)}{\partial \theta_j} = U_j, \quad 
\lim_{\kappa \to 0} \frac{\partial \log\left(\tilde{\mathcal{Z}}^{\kappa}(\boldsymbol{\theta})\right)}{\partial \theta_j} = p_j.
\]

Through direct verification, we can show that a partition function \( \mathcal{Z}(\boldsymbol{\theta}) \) degenerates to a pair \( (U_j, p_j) \) satisfying Theorem~\ref{CS results kappa=0} as \( \kappa \to 0 \).

From the operator perspective, the quantum and classical Hamiltonians are connected through the process of quantization. Specifically, canonical quantization transforms classical position and momentum functions into their corresponding operators:
\[
\begin{aligned}
\theta_j & \Rightarrow \hat{\theta}_j,  \\
p_j & \Rightarrow \hat{p}_j = \kappa \frac{\partial}{\partial \theta_j}.
\end{aligned}
\]
The Poisson brackets are replaced by Lie brackets:
\[
\begin{aligned}
& \left\{p_i, \theta_j\right\} = \delta_{ij} \Rightarrow \frac{1}{\kappa}\left[\hat{p}_i, \hat{\theta}_j\right] = \delta_{ij}, \\
& \left\{\theta_i, \theta_j\right\} = 0 \Rightarrow \frac{1}{\kappa}\left[\hat{\theta}_i, \hat{\theta}_j\right] = 0, \\
& \left\{p_i, p_j\right\} = 0 \Rightarrow \frac{1}{\kappa}\left[\hat{p}_i, \hat{p}_j\right] = 0.
\end{aligned}
\]

The classical Hamiltonians \( \mathcal{H} \) and \( \mathcal{H}_j \) correspond to the quantum Hamiltonian operators \( \mathcal{L} \) and \( \mathcal{L}_j \), respectively. For further discussion on the quantization of Calogero-Sutherland systems, see~\cite{E07}.

\section*{Acknowledgement}
I express my sincere gratitude to Professor Nikolai Makarov for his invaluable guidance. I am also thankful to Professor Eveliina Peltola and Professor Hao Wu for their enlighting discussions.

\end{document}